\documentclass[reqno]{amsart}
\usepackage[utf8]{inputenc}
\usepackage{amsthm,thmtools,amssymb}
\usepackage{stmaryrd}
\usepackage{mathrsfs}
\usepackage[a4paper, margin=2cm]{geometry}
\usepackage{accents}
\usepackage{comment}
\usepackage{multirow}
\usepackage[colorlinks=true,pagebackref]{hyperref}
\usepackage[bb=stix]{mathalpha}
\usepackage{cleveref}
\hypersetup{%
    colorlinks=true,
    linkcolor=violet,
    citecolor=violet
}
\usepackage{tikz}
\usetikzlibrary{arrows,arrows.meta,decorations.pathreplacing,decorations.markings,shapes,calc}
\tikzset{labelsize/.style={font=\scriptsize}}
\tikzset{string/.style={very thick}}
\tikzset{
  pto/.style={->,postaction={decorate},
    decoration={
        markings,
        mark=at position 0.5 with {\arrow{|}}}
  },
}
\tikzset{2cell/.style={-implies,double,double equal sign distance,shorten >=9pt, shorten <=10pt}}
\usepackage{tikz-cd}

\mathchardef\mhyphen="2D

\newtheorem{theorem}{Theorem}[section]

\newtheorem{proposition}[theorem]{Proposition}

\theoremstyle{definition}
\newtheorem{definition}[theorem]{Definition}
\newtheorem{example}[theorem]{Example}

\theoremstyle{remark}
\newtheorem{remark}[theorem]{Remark}
\numberwithin{equation}{section}

\crefname{theorem}{Theorem}{Theorems}
\crefname{section}{Section}{Sections}
\crefname{subsection}{Subsection}{Subsections}
\crefname{definition}{Definition}{Definitions}
\crefname{notation}{Notation}{Notations}
\crefname{example}{Example}{Examples}
\crefname{remark}{Remark}{Remarks}
\crefname{equation}{}{}
\crefname{corollary}{Corollary}{Corollaries}
\crefname{proposition}{Proposition}{Propositions}
\crefname{lemma}{Lemma}{Lemmas}
\crefname{figure}{Figure}{Figures}

\makeatletter
\newcommand{\pto}{}
\newcommand{\pgets}{}

\DeclareRobustCommand{\pto}{\mathrel{\mathpalette\p@to@gets\to}}
\DeclareRobustCommand{\pgets}{\mathrel{\mathpalette\p@to@gets\gets}}

\newcommand{\p@to@gets}[2]{%
  \ooalign{\hidewidth$\m@th#1\mapstochar\mkern5mu$\hidewidth\cr$\m@th#1\to$\cr}%
}
\makeatother

\newcommand{\cat}[1]{\mathcal{#1}} 
\newcommand{\bcat}[1]{\mathscr{#1}} 
\newcommand{\tcat}[1]{\mathscr{#1}} 
\newcommand{\ecat}[1]{\mathsf{#1}} 
\newcommand{\vdcat}[1]{\mathbb{#1}} 
\newcommand{\mnd}[1]{\mathsf{#1}} 

\newcommand{\enCat}[1]{{#1}\mhyphen\mathbf{Cat}}

\newcommand{\id}{\mathrm{id}}

\newcommand{\vc}{\mathrm{vert}} 
\newcommand{\hc}{\mathrm{hc}} 
\newcommand{\he}{\mathrm{he}} 

\newcommand{\CATpb}{\mathbf{CAT}_\mathrm{pb}}

\newcommand{\Set}{\mathbf{Set}}
\newcommand{\Setlc}{{\mathbf{Set}_\mathrm{lc}}}
\newcommand{\Cat}{\mathbf{Cat}}
\newcommand{\CAT}{\mathbf{CAT}}

\newcommand{\Enr}{\mathrm{Enr}}
\newcommand{\EEnr}{\mathbb{Enr}}
\newcommand{\twoCAT}{{2\mhyphen\mathbf{CAT}}}

\newcommand{\Gph}{\mathbf{Gph}}

\newcommand{\PsDbl}{{\mathbf{PsDBL}_\ell}}
\newcommand{\StrDbl}{{\mathbf{StrDBL}_\ell}}
\newcommand{\VDbl}{\mathbf{VDBL}}
\newcommand{\VDbln}{\mathbf{UVDBL}}
\newcommand{\MMod}{{\mathbb{Mod}}}
\newcommand{\FF}{{\mathbb{F}}}
\newcommand{\hu}{\mathrm{I}} 
\newcommand{\vmm}[1]{\underline{#1}} 
\newcommand{\hbm}[1]{\underline{#1}} 

\newcommand{\Inc}{\mathrm{Inc}}
\newcommand{\Und}{\mathrm{Und}}

\newcommand{\MonCAT}{{\mathbf{MonCAT}_\ell}}
\newcommand{\BICAT}{{\mathbf{BICAT}_\ell}}

\newcommand{\vect}[1]{\boldsymbol{#1}}

\newcommand{\FFam}{{\mathbb{Fam}}}
\newcommand{\Fam}{\mathrm{Fam}}

\newcommand{\ob}{\mathrm{ob}}
\DeclareMathOperator{\EElts}{{\mathbb{Elt}}}
\DeclareMathOperator{\Elts}{\mathrm{Elt}}

\newcommand{\SSpan}{{\mathbb{Span}}}

\newcommand{\MMat}{{\mathbb{Mat}}}
\newcommand{\enMMat}[1]{{#1}\mhyphen{\mathbb{Mat}}}
\newcommand{\PProf}[1]{{#1}\mhyphen{\mathbb{Prof}}}
\newcommand{\intPProf}{{\mathbb{Prof}}}

\newcommand{\V}{\mathscr{V}}

\newcommand{\UU}{\mathbb{U}}
\newcommand{\VV}{\mathbb{V}}
\newcommand{\LL}{\mathbb{L}}

\DeclareMathOperator{\Spl}{Spl}

\title{The familial nature of enrichment over virtual double categories}

\author{Soichiro Fujii}
\address{Department of Mathematics and Statistics, Faculty of Science, Masaryk University, Kotl\'a\v{r}sk\'a 2, 611 37 Brno, Czech Republic}
\email{s.fujii.math@gmail.com}

\author{Stephen Lack}
\address{School of Mathematical and Physical Sciences, Macquarie University NSW 2109, 
Australia}
\email{steve.lack@mq.edu.au}



\date{July 8, 2025}

\begin{document}

\begin{abstract}
Originally enriched categories were defined over a monoidal category, but it was gradually realized that important examples can only be included when one enriches over more general structures such as bicategories and virtual double categories. We show that, as well as allowing more examples, working over virtual double categories also gives better formal properties. We study the 2-functor sending a virtual double category to the 2-category of categories enriched over it. We show that this is a parametric right 2-adjoint, and in fact is familial. We also show how a ``families construction'' for virtual double categories can be used to give a formal construction of the 2-category of categories enriched over a virtual double category.
\end{abstract}

\maketitle
\setcounter{tocdepth}{1}
\tableofcontents

\section{Introduction}\label{sec:intro}
Enriched category theory is usually developed over monoidal categories: for each monoidal category $\cat{V}$, we have (small) $\cat{V}$-categories, $\cat{V}$-functors, and $\cat{V}$-natural transformations, which comprise the 2-category $\enCat{\cat{V}}$. We have a 2-functor 
\begin{equation}
\label{eqn:moncat-enrichment}
\Enr\colon \MonCAT\to \twoCAT
\end{equation}
mapping each monoidal category $\cat{V}$ to the 2-category $\Enr(\cat{V})=\enCat{\cat{V}}$. Here, $\MonCAT$ is the 2-category of (large) monoidal categories, lax monoidal functors, and monoidal natural transformations, whereas $\twoCAT$ is the 2-category of (large) 2-categories, 2-functors, and 2-natural transformations.

One of the main goals of this paper is to show that the 2-functor \cref{eqn:moncat-enrichment} becomes a \emph{parametric right $2$-adjoint}, provided that we extend its domain appropriately.
Recall \cite{Street-petit-topos, Weber-familial} that a $2$-functor $R\colon \tcat{K}\to\tcat{L}$ from a $2$-category $\tcat{K}$ with a terminal object $1$ is called a \emph{parametric right $2$-adjoint}, if the $2$-functor $R_1\colon \tcat{K}\to\tcat{L}/R1$, sending each $K\in \tcat{K}$ to $(R!_K\colon RK\to R1)\in \tcat{L}/R1$, is a right $2$-adjoint.

There are various reasons to generalize enrichment bases from monoidal categories to bicategories: this allows important new examples to be captured as well as having better formal properties (see e.g.\ \cite{Walters-sheaves,Street-cohomology,Fujii_Lack}), and we can extend the 2-functor \eqref{eqn:moncat-enrichment} to
\begin{equation}
\label{eqn:bicat-enrichment}
\Enr\colon\BICAT\to \twoCAT,
\end{equation}
where $\BICAT$ is the 2-category of (large) bicategories, lax functors, and icons \cite{Lack_icons}. $\MonCAT$ can be regarded as a full sub-2-category of $\BICAT$, by identifying monoidal categories with one-object bicategories. 
Now, $\BICAT$ has the terminal object $1$, and hence we can factorize \eqref{eqn:bicat-enrichment} as 
\[
\begin{tikzpicture}[baseline=-\the\dimexpr\fontdimen22\textfont2\relax ]
      \node(0) at (0,-0.75) {$\BICAT$};
      \node(1) at (3,-0.75) {$\twoCAT$,};
      \node(2) at (1.5,0.75) {$\twoCAT/\Enr(1)$};
      \draw [->] (0) to node[auto,swap, labelsize] {$\Enr$}  (1);  
      \draw [->] (2) to node[auto, labelsize] {forgetful}  (1); 
      \draw [->] (0) to node[auto, labelsize] {$\Enr_{1}$} (2); 
\end{tikzpicture}
\]
where the $2$-functor $\Enr_1$ maps each bicategory $\tcat{B}$ to $\bigl(\Enr(!_\tcat{B})\colon \Enr(\tcat{B})\to \Enr(1)\bigr)\in\twoCAT/\Enr(1)$. Here $\Enr(1)$ is the locally chaotic 2-category of sets and functions: see \cref{subsec:enrichment-over-vdbl} below.
In \cite[Theorem~2.1]{Fujii_Lack}, we showed that the 2-functor $\Enr_1\colon \BICAT\to \twoCAT/\Enr(1)$ preserves all limits which happen to exist in $\BICAT$ (which is far from being complete).\footnote{Incidentally, this is not the case for the 2-functor $\Enr_1\colon \MonCAT\to \twoCAT/\Enr(1)$: there exists an equalizer in $\MonCAT$ which is not preserved by $\Enr_1\colon \MonCAT\to \twoCAT/\Enr(1)$ (and hence neither by the inclusion $\MonCAT\to \BICAT$); see \cref{rmk:MonCAT-not-pra}.}
This raises the following natural question: does $\Enr_1\colon \BICAT\to \twoCAT/\Enr(1)$ have a left 2-adjoint? 
In other words, is $\Enr\colon \BICAT\to \twoCAT$ a parametric right 2-adjoint?

The answer to this question turns out to be \emph{negative} (see \cref{rmk:BICAT-not-pra}). However, by further extending the bases of enrichment from bicategories to \emph{virtual double categories} (or to an intermediate class of \emph{pseudo double categories}), the answer becomes positive.
Enrichment over virtual double categories was defined in  \cite{Leinster-generalized-enrichment}; we review the main definitions in \cref{subsec:enrichment-over-vdbl}. 
We can extend \eqref{eqn:bicat-enrichment} to a 2-functor
\begin{equation}
    \label{eqn:vdbl-enrichment}
    \Enr\colon \VDbl\to\twoCAT,
\end{equation}
where $\VDbl$ is the 2-category of (large) virtual double categories, virtual double functors, and vertical natural transformations, whose definitions we recall in \cref{subsec:vdbl}.
The 2-functor $\Enr_1\colon \VDbl\to \twoCAT/\Enr(1)$, obtained by factorizing $\Enr\colon\VDbl\to \twoCAT$, does have a left 2-adjoint $\LL\colon \twoCAT/\Enr(1)\to \VDbl$:
\[
\begin{tikzpicture}[baseline=-\the\dimexpr\fontdimen22\textfont2\relax ]
      \node(0) at (0,-0.75) {$\VDbl$};
      \node(1) at (3,-0.75) {$\twoCAT$.};
      \node(2) at (1.5,0.75) {$\twoCAT/\Enr(1)$};
      \draw [->] (0) to node[auto,swap, labelsize] {$\Enr$}  (1);  
      \draw [->] (2) to node[auto, labelsize] {forgetful}  (1); 
      \draw [->, transform canvas={xshift=6}] (0) to node[auto, swap,labelsize] {$\Enr_{1}$} (2); 
      \draw [<-, transform canvas={xshift=-6}] (0) to node[auto,  labelsize] {$\LL$} (2); 
      \path(0) to node[rotate=-45] {$\dashv$} (2);
\end{tikzpicture}
\]

In fact, it is not difficult to give an explicit description of $\LL$ (see \cref{sec:L-explicitly}).
However, the 2-adjunction $\LL\dashv \Enr_1$ turns out to be related to other fundamental constructions involving virtual double categories, and that will be the main theme of this paper. 

As we recall in \cref{sec:poly-pra-fam}, the class of parametric right $2$-adjoints is closed under composition and contains both the right $2$-adjoints and the polynomial $2$-functors.
Therefore, in order to show that $\Enr\colon \VDbl\to \twoCAT$ is a parametric right $2$-adjoint, it suffices to express $\Enr$ 
as a composite of right $2$-adjoints and polynomial $2$-functors. 
To this end, we note that $\Enr$ is the following composite: 
\begin{equation}\label{eqn:Enr-decomposition}
\begin{tikzpicture}[baseline=-\the\dimexpr\fontdimen22\textfont2\relax ]
      \node(0) at (0,0)  {$\VDbl$};
      \node(1) at (3,0)  {$\VDbl$};
      \node(2) at (6.5,0)  {$\VDbln$};
      \node(3) at (10,0) {$\twoCAT.$};
      \draw [->] (0) to node[auto, labelsize] {$\MMat$}  (1);
      \draw [->] (1) to node[auto, labelsize] {$\MMod$}  (2);
      \draw [->] (2) to node[auto, labelsize] {$\V$}  (3);
\end{tikzpicture}
\end{equation}
Basically, this is saying that, for any virtual double category $\vdcat{A}$, an $\vdcat{A}$-category is a horizontal monad in the virtual double category $\enMMat{\vdcat{A}}$ of \emph{matrices in $\vdcat{A}$}.
When $\vdcat{A}$ is a monoidal category or a bicategory, this is a well-known observation; see e.g.\ \cite{variation-through-enrichment}.
Here, $\VDbln$ is a $2$-category of \emph{unital virtual double categories} (i.e., virtual double categories with \emph{chosen} horizontal units), while $\MMod$ and $\V$ are variants of 2-functors studied in \cite{Cruttwell-Shulman-unified}. Both $\MMod$ and $\V$ turn out to be right $2$-adjoints.

On the other hand, $\MMat\colon \VDbl\to \VDbl$ is a polynomial $2$-functor, induced by a certain \emph{polynomial} 
\begin{equation}
\label{eqn:polynomial-for-Mat}
\begin{tikzpicture}[baseline=-\the\dimexpr\fontdimen22\textfont2\relax ]
      \node(0) at (0,0) {$1$};
      \node(1)at (2,0) {$(\Set_\ast)_\hc$};
      \node(2) at (4,0) {$\Set_\hc$};
      \node(3) at (6,0) {$1$};
      \draw [->] (1) to node[auto,swap, labelsize] {$!_{(\Set_\ast)_\hc}$}  (0);
      \draw [->] (1) to node[auto,labelsize] {$P_\hc$}  (2);
      \draw [->] (2) to node[auto, labelsize] {$!_{\Set_\hc}$}  (3);
\end{tikzpicture}
\end{equation}
in $\VDbl$; in other words, $\MMat$ is the composite of 
\begin{equation*}
\begin{tikzpicture}[baseline=-\the\dimexpr\fontdimen22\textfont2\relax ]
      \node(0) at (-1,0)  {$\VDbl$};
      \node(1) at (3,0)  {$\VDbl/(\Set_\ast)_\hc$};
      \node(2) at (6.5,0)  {$\VDbl/\Set_\hc$};
      \node(3) at (10,0) {$\VDbl$.};
      \draw [->] (0) to node[auto, labelsize] {$(\Set_\ast)_\hc\times (-)$}  (1);
      \draw [->] (1) to node[auto, labelsize] {$\prod_{P_\hc}$}  (2);
      \draw [->] (2) to node[auto, labelsize] {forgetful}  (3);
\end{tikzpicture}
\end{equation*}
The $2$-functor $\prod_{P_\hc}$ (which is defined as the right $2$-adjoint of the pullback $P_\hc^\ast$ along $P_\hc$) exists because $P_\hc$ is \emph{powerful} (or \emph{exponentiable}).
This in turn follows from the fact that it is 
a \emph{discrete opfibration} between virtual double categories, as we show in \cref{prop:disc-opfib-powerful}.

Incidentally, the composite of the first two factors of \eqref{eqn:Enr-decomposition} is the 
$2$-functor $\EEnr\colon \VDbl\to \VDbln$, sending each $\vdcat{A}\in\VDbl$ to the (unital) virtual double category $\EEnr(\vdcat{A})=\PProf{\vdcat{A}}$ whose horizontal morphisms are \emph{$\vdcat{A}$-profunctors}. 
Such virtual double categories play an important role in formal category theory; see e.g.\ \cite{Shulman-enriched-indexed,Koudenburg-formal,Kawase-profunctor}.
We thus see that this $\EEnr$ is also a parametric right $2$-adjoint.

The polynomial \eqref{eqn:polynomial-for-Mat} in $\VDbl$ can be obtained from the polynomial 
\begin{equation}\label{eqn:polynomial-for-Fam-in-CAT}
\begin{tikzpicture}[baseline=-\the\dimexpr\fontdimen22\textfont2\relax ]
      \node(0) at (0,0) {$1$};
      \node(1)at (2,0) {$\Set_\ast$};
      \node(2) at (4,0) {$\Set$};
      \node(3) at (6,0) {$1$};
      \draw [->] (1) to node[auto,swap, labelsize] {$!_{\Set_\ast}$}  (0);
      \draw [->] (1) to node[auto,labelsize] {$P$}  (2);
      \draw [->] (2) to node[auto, labelsize] {$!_{\Set}$}  (3);
\end{tikzpicture}
\end{equation}
in $\CAT$ by applying the $2$-functor $(-)_\hc\colon\CAT\to\VDbl$ which maps each category $\cat{C}$ to the \emph{horizontally chaotic} virtual double category $\cat{C}_\hc$ whose vertical category is $\cat{C}$ (see \cref{sec:MMat} for details).
The polynomial $2$-functor induced by \eqref{eqn:polynomial-for-Fam-in-CAT} is $\Fam\colon \CAT\to \CAT$ given by the \emph{families} construction. 
As is well known, the left $2$-adjoint of $\Fam_1\colon \CAT\to \CAT/\Fam(1)=\CAT/\Set$ is 
$\Elts\colon \CAT/\Set\to \CAT$ given by the \emph{category of elements} construction (see e.g.\ \cite[Proposition~5.4]{Johnson-PhD}).
The functor $P\colon \Set_\ast\to \Set$ is the \emph{universal discrete opfibration} (with small fibers) in $\CAT$, in the sense that any discrete opfibration (with small fibers) in $\CAT$ can be obtained by a pullback of $P$; in fact, one can use $\Elts$ to show this. 

Now, it turns out that $\VDbl$ has a universal discrete opfibration as well:
this is not $P_\hc\colon (\Set_\ast)_\hc\to \Set_\hc$, but a ``larger'' discrete opfibration $Q\colon \SSpan_\ast\to\SSpan$ defined in \cref{sec:FFam}.
This gives rise to a polynomial
\begin{equation}
\label{eqn:polynomial-for-fam}
\begin{tikzpicture}[baseline=-\the\dimexpr\fontdimen22\textfont2\relax ]
      \node(0) at (0,0) {$1$};
      \node(1)at (2,0) {$\SSpan_\ast$};
      \node(2) at (4,0) {$\SSpan$};
      \node(3) at (6,0) {$1$};
      \draw [->] (1) to node[auto,swap, labelsize] {$!_{\SSpan_\ast}$}  (0);
      \draw [->] (1) to node[auto,labelsize] {$Q$}  (2);
      \draw [->] (2) to node[auto, labelsize] {$!_{\SSpan}$}  (3);
\end{tikzpicture}
\end{equation}
in $\VDbl$, inducing the polynomial $2$-functor $\FFam\colon \VDbl\to \VDbl$ which gives rise to a suitable {\em families} construction for virtual double categories. 
The left $2$-adjoint of $\FFam_1\colon \VDbl\to \VDbl/\FFam(1)=\VDbl/\SSpan$ is $\EElts\colon \VDbl/\SSpan\to \VDbl$, a suitable \emph{virtual double category of elements} construction.
Our $\FFam$ and $\EElts$ generalize known constructions \cite{Pare-Fam,Pare-Yoneda,Patterson-products} for pseudo double categories to virtual double categories; see \cref{rmk:Fam-for-PsDbl,rmk:Elt-for-PsDbl}.

In fact $\Enr\colon\VDbl\to\twoCAT$ and various others of the
parametric right 2-adjoints we construct have the stronger property of
being {\em familial} \cite{Weber-familial}. This implies among other
things that they
preserve various classes of fibrations and that they preserve
bicategorical pullbacks.

\subsection*{Convention on size}
We fix two Grothendieck universes $\cat{U}_1$ and $\cat{U}_2$ with $\cat{U}_1\in\cat{U}_2$. 
Sets, categories, etc.\ are called \emph{small} (resp.\ \emph{large}) if they are in $\cat{U}_1$ (resp.\ in $\cat{U}_2$).

\subsection*{Outline of the paper}

In Section~\ref{sec:vdbl-and-enrichment} we recall the basic definitions of virtual double categories and enrichment over virtual double categories, then in Section~\ref{sec:L-explicitly} we give an explicit construction of the left $2$-adjoint of $\Enr_1\colon\VDbl\to\twoCAT/\Enr(1)$, as well as arguments showing that the monoidal category and bicategory versions of $\Enr_1$ do not have adjoints. In Section~\ref{sec:poly-pra-fam} we recall the basic definitions and basic facts about parametric right 2-adjoints, familial 2-functors, and polynomial 2-functors, while in Section~\ref{sec:dof} we show that any discrete opfibration of virtual double categories is powerful. In Section~\ref{sec:MMat} we construct the polynomial 2-functor $\MMat\colon\VDbl\to\VDbl$ and show that it is (not just a parametric right 2-adjoint but) familial. In Section~\ref{sec:units} we introduce a 2-category $\VDbln$ of unital virtual double categories (that is, virtual double categories with chosen horizontal units), and construct various adjoints involving $\VDbln$, allowing us to prove that the 2-functor $\Enr\colon\VDbl\to\twoCAT$ is familial. Finally in Section~\ref{sec:FFam} we construct the polynomial 2-functor $\FFam\colon\VDbl\to\VDbl$, show that it is familial, and explain how its status as a parametric right 2-adjoint relates to the elements construction. We also show how it can be used to give a purely formal construction of 2-categories of enriched categories.

\subsection*{Acknowledgments}
The first-named author acknowledges the support of JSPS Overseas Research Fellowships and the Grant Agency of the Czech Republic under the grant 22-02964S.

\section{Virtual double categories and enrichment over them}\label{sec:vdbl-and-enrichment}

\subsection{Virtual double categories}\label{subsec:vdbl}
We recall the notions of virtual double category, virtual double functor, and vertical transformation, and fix our notation. Modern references include \cite{Leinster-HOHC,Cruttwell-Shulman-unified}: the name virtual double category was introduced in \cite{Cruttwell-Shulman-unified}; in \cite{Leinster-generalized-enrichment,Leinster-HOHC} they were known as {\em fc-multicategories}.

\begin{definition}
A \emph{virtual double category} $\vdcat{A}$ consists of the following data.
\begin{itemize}
    \item First we have a category $\vdcat{A}_\vc$ called the \emph{vertical category} of $\vdcat{A}$. The objects of $\vdcat{A}_\vc$ are called the \emph{objects} of $\vdcat{A}$ and the morphisms of $\vdcat{A}_\vc$ are called the \emph{vertical morphisms} of $\vdcat{A}$. A typical vertical morphism $u\colon A\to B$ and the identity vertical morphism $\id_A\colon A\to A$ on an object $A$ are denoted by  
\[
\begin{tikzpicture}[baseline=-\the\dimexpr\fontdimen22\textfont2\relax ]
      \node(0) at (0,0.75) {$A$};
      \node(1) at (0,-0.75) {$B$};
      \draw [->] (0) to node[auto, labelsize] {$u$} (1); 
\end{tikzpicture}\qquad \text{and}\qquad 
\begin{tikzpicture}[baseline=-\the\dimexpr\fontdimen22\textfont2\relax ]
      \node(0) at (0,0.75) {$A$};
      \node(1) at (0,-0.75) {$A$};
      \draw [double equal sign distance] (0) to (1); 
\end{tikzpicture}
\]
    respectively.
    \item For each pair of objects $A,A'$ of $\vdcat{A}$, we have a set of \emph{horizontal morphisms} from $A$ to $A'$. A typical horizontal morphism $\Gamma$ from $A$ to $A'$ is denoted by 
\[
\begin{tikzpicture}[baseline=-\the\dimexpr\fontdimen22\textfont2\relax ]
      \node(0) at (0,0) {$A$};
      \node(1) at (1.5,0) {$A'$.};
      \draw [pto] (0) to node[auto, labelsize] {$\Gamma$} (1); 
\end{tikzpicture}
\]
    \item For each natural number $n$ and each configuration
\begin{equation}
\label{eqn:boundary-of-2-cell-in-vdc}
\begin{tikzpicture}[baseline=-\the\dimexpr\fontdimen22\textfont2\relax ]
      \node(00) at (0,0.75) {$A_0$};
      \node(01) at (1.5,0.75) {$A_1$};
      \node(02) at (3,0.75) {$\cdots$};
      \node(03) at (4.5,0.75) {$A_n$};
      \node(10) at (0,-0.75) {$B$};
      \node(11) at (4.5,-0.75) {$B'$}; 
      \draw [pto] (00) to node[auto, labelsize] {$\Gamma_1$} (01);  
      \draw [pto] (01) to node[auto, labelsize] {$\Gamma_2$} (02);  
      \draw [pto] (02) to node[auto, labelsize] {$\Gamma_n$} (03);  
      \draw [pto] (10) to node[auto, swap, labelsize] {$\Delta$} (11); 
      \draw [->] (00) to node[auto, swap, labelsize] {$u$} (10); 
      \draw [->] (03) to node[auto,labelsize] {$u'$} (11);
\end{tikzpicture}    
\end{equation}
    in $\vdcat{A}$, we have a set of \emph{multicells} having \eqref{eqn:boundary-of-2-cell-in-vdc} as the boundary. If $\theta$ is such a multicell, we denote it by 
\begin{equation*}
\begin{tikzpicture}[baseline=-\the\dimexpr\fontdimen22\textfont2\relax ]
      \node(00) at (0,0.75) {$A_0$};
      \node(01) at (1.5,0.75) {$A_1$};
      \node(02) at (3,0.75) {$\cdots$};
      \node(03) at (4.5,0.75) {$A_n$};
      \node(10) at (0,-0.75) {$B$};
      \node(11) at (4.5,-0.75) {$B'$.}; 
      \draw [pto] (00) to node[auto, labelsize] {$\Gamma_1$} (01);  
      \draw [pto] (01) to node[auto, labelsize] {$\Gamma_2$} (02);  
      \draw [pto] (02) to node[auto, labelsize] {$\Gamma_n$} (03);  
      \draw [pto] (10) to node[auto, swap, labelsize] {$\Delta$} (11); 
      \draw [->] (00) to node[auto, swap, labelsize] {$u$} (10); 
      \draw [->] (03) to node[auto,labelsize] {$u'$} (11);
      \node at (2.25,0) {$\theta$}; 
\end{tikzpicture}
\end{equation*}
    Notice that in the special case $n=0$, we have 
\begin{equation*}
\begin{tikzpicture}[baseline=-\the\dimexpr\fontdimen22\textfont2\relax ]
      \node(00) at (0.75,0.75) {$A_0$};
      \node(10) at (0,-0.75) {$B$};
      \node(11) at (1.5,-0.75) {$B'$.}; 
      \draw [pto] (10) to node[auto, swap, labelsize] {$\Delta$} (11); 
      \draw [->] (00) to node[auto, swap, labelsize] {$u$} (10); 
      \draw [->] (00) to node[auto,labelsize] {$u'$} (11);
      \node at (0.75,-0.2) {$\theta$}; 
\end{tikzpicture}
\end{equation*}
    A \emph{cell} in $\vdcat{A}$ is a multicell in $\vdcat{A}$ of the form   
    \begin{equation*}
\begin{tikzpicture}[baseline=-\the\dimexpr\fontdimen22\textfont2\relax ]
      \node(00) at (0,0.75) {$A_0$};
      \node(01) at (1.5,0.75) {$A_1$};
      \node(10) at (0,-0.75) {$B$};
      \node(11) at (1.5,-0.75) {$B'$.}; 
      \draw [pto] (00) to node[auto, labelsize] {$\Gamma_1$} (01);  
      \draw [pto] (10) to node[auto, swap, labelsize] {$\Delta$} (11); 
      \draw [->] (00) to node[auto, swap, labelsize] {$u$} (10); 
      \draw [->] (01) to node[auto, labelsize] {$u'$} (11);
      \node at (0.75,0) {$\theta$}; 
\end{tikzpicture}
\end{equation*}
    \item For each horizontal morphism $\Gamma\colon A\pto A'$ in $\vdcat{A}$, we have a cell
\begin{equation*}
\begin{tikzpicture}[baseline=-\the\dimexpr\fontdimen22\textfont2\relax ]
      \node(00) at (0,0.75) {$A$};
      \node(01) at (1.5,0.75) {$A'$};
      \node(10) at (0,-0.75) {$A$};
      \node(11) at (1.5,-0.75) {$A'$}; 
      \draw [pto] (00) to node[auto, labelsize] {$\Gamma$} (01);  
      \draw [pto] (10) to node[auto, swap, labelsize] {$\Gamma$} (11); 
      \draw [double equal sign distance] (00) to (10); 
      \draw [double equal sign distance] (01) to (11);
      \node at (0.75,0) {$\id_\Gamma$}; 
\end{tikzpicture}
\end{equation*}
    called the \emph{vertical identity cell} of $\Gamma$. We often denote it by 
\begin{equation*}
\begin{tikzpicture}[baseline=-\the\dimexpr\fontdimen22\textfont2\relax ]
      \node(00) at (0,0.75) {$A$};
      \node(01) at (1.5,0.75) {$A'$};
      \node(10) at (0,-0.75) {$A$};
      \node(11) at (1.5,-0.75) {$A'$.}; 
      \draw [pto] (00) to node[auto, labelsize] {$\Gamma$} (01);  
      \draw [pto] (10) to node[auto, swap, labelsize] {$\Gamma$} (11); 
      \draw [double equal sign distance] (00) to (10); 
      \draw [double equal sign distance] (01) to (11); 
\end{tikzpicture}
\end{equation*}
    \item For each natural number $m$, each $m$-tuple of natural numbers $(n_1,\dots,n_m)$, and each configuration 
\begin{equation}
\label{eqn:composable-multicells-A}
\begin{tikzpicture}[baseline=-\the\dimexpr\fontdimen22\textfont2\relax ]
      \node(00) at (0,1.5) {$A_{1,0}$};
      \node(01) at (2,1.5) {$\cdots$};
      \node(02) at (4,1.5) {$A_{1,n_1}$};
      \node(03) at (6,1.5) {$\cdots$};
      \node(04) at (8,1.5) {$A_{m,0}$};
      \node(05) at (10,1.5) {$\cdots$};
      \node(06) at (12,1.5) {$A_{m,n_m}$};
      \node(10) at (0,0) {$B_0$};
      \node(11) at (4,0) {$B_1$}; 
      \node(12) at (6,0) {$\cdots$}; 
      \node(13) at (8,0) {$B_{m-1}$};
      \node(14) at (12,0) {$B_m$}; 
      \node(20) at (0,-1.5) {$C$};
      \node(21) at (12,-1.5) {$C'$}; 
      \draw [pto] (00) to node[auto, labelsize] {$\Gamma_{1,1}$} (01);  
      \draw [pto] (01) to node[auto, labelsize] {$\Gamma_{1,n_1}$} (02);  
      \draw [pto] (02) to node[auto, labelsize] {$\Gamma_{2,1}$} (03);  
      \draw [pto] (03) to node[auto, labelsize] {$\Gamma_{m-1,n_{m-1}}$} (04);  
      \draw [pto] (04) to node[auto, labelsize] {$\Gamma_{m,1}$} (05);  
      \draw [pto] (05) to node[auto, labelsize] {$\Gamma_{m,n_{m}}$} (06);  
      \draw [pto] (10) to node[auto, labelsize] {$\Delta_1$} (11); 
      \draw [pto] (11) to node[auto, labelsize] {$\Delta_2$} (12); 
      \draw [pto] (12) to node[auto, labelsize] {$\Delta_{m-1}$} (13); 
      \draw [pto] (13) to node[auto, labelsize] {$\Delta_m$} (14); 
      \draw [pto] (20) to node[auto, swap,labelsize] {$\Lambda$} (21); 
      \draw [->] (00) to node[auto, swap, labelsize] {$u_0$} (10); 
      \draw [->] (02) to node[auto,swap,labelsize] {$u_1$} (11);
      \draw [->] (04) to node[auto,labelsize] {$u_{m-1}$} (13);
      \draw [->] (06) to node[auto,labelsize] {$u_{m}$} (14);
      \draw [->] (10) to node[auto, swap, labelsize] {$v$} (20); 
      \draw [->] (14) to node[auto,labelsize] {$v'$} (21);
      \node at (2,0.75) {$\theta_1$}; 
      \node at (10,0.75) {$\theta_m$}; 
      \node at (6,-0.75) {$\kappa$}; 
\end{tikzpicture}
\end{equation}
    (where it is understood that $A_{i,n_i}=A_{i+1,0}$ for each $1\leq i\leq m-1$)
    in $\vdcat{A}$, we have the composite multicell
\[
\begin{tikzpicture}[baseline=-\the\dimexpr\fontdimen22\textfont2\relax ]
      \node(00) at (0,1.5) {$A_{1,0}$};
      \node(01) at (2,1.5) {$\cdots$};
      \node(02) at (4,1.5) {$A_{1,n_1}$};
      \node(03) at (6,1.5) {$\cdots$};
      \node(04) at (8,1.5) {$A_{m,0}$};
      \node(05) at (10,1.5) {$\cdots$};
      \node(06) at (12,1.5) {$A_{m,n_m}$};
      \node(10) at (0,0) {$B_0$};
      \node(14) at (12,0) {$B_m$}; 
      \node(20) at (0,-1.5) {$C$};
      \node(21) at (12,-1.5) {$C'$}; 
      \draw [pto] (00) to node[auto, labelsize] {$\Gamma_{1,1}$} (01);  
      \draw [pto] (01) to node[auto, labelsize] {$\Gamma_{1,n_1}$} (02);  
      \draw [pto] (02) to node[auto, labelsize] {$\Gamma_{2,1}$} (03);  
      \draw [pto] (03) to node[auto, labelsize] {$\Gamma_{m-1,n_{m-1}}$} (04);  
      \draw [pto] (04) to node[auto, labelsize] {$\Gamma_{m,1}$} (05);  
      \draw [pto] (05) to node[auto, labelsize] {$\Gamma_{m,n_{m}}$} (06);  
      \draw [pto] (20) to node[auto, swap,labelsize] {$\Lambda$} (21); 
      \draw [->] (00) to node[auto, swap, labelsize] {$u_0$} (10); 
      \draw [->] (06) to node[auto,labelsize] {$u_{m}$} (14);
      \draw [->] (10) to node[auto, swap, labelsize] {$v$} (20); 
      \draw [->] (14) to node[auto,labelsize] {$v'$} (21);
      \node at (6,0) {$\kappa\circ(\theta_1,\theta_2,\dots,\theta_m)$}; 
\end{tikzpicture}
\]  
    in $\vdcat{A}$. However, note that this notation is inadequate when $m=0$; in this case, for each configuration
    \[\begin{tikzpicture}[baseline=-\the\dimexpr\fontdimen22\textfont2\relax ]
      \node(20) at (0.75,1.5) {$A$};
      \node(00) at (0.75,0) {$B$};
      \node(10) at (0,-1.5) {$C$};
      \node(11) at (1.5,-1.5) {$C'$}; 
      \draw [pto] (10) to node[auto, swap, labelsize] {$\Lambda$} (11); 
      \draw [->] (20) to node[auto, labelsize] {$u$} (00); 
      \draw [->] (00) to node[auto, swap, labelsize] {$v$} (10); 
      \draw [->] (00) to node[auto,labelsize] {$v'$} (11);
      \node at (0.75,-0.95) {$\kappa$}; 
\end{tikzpicture}\]
in $\vdcat{A}$, we have the composite multicell
    \[\begin{tikzpicture}[baseline=-\the\dimexpr\fontdimen22\textfont2\relax ]
      \node(20) at (0.75,1.5) {$A$};
      \node(00) at (0.375,0) {$B$};
      \node(01) at (1.125,0) {$B$};
      \node(10) at (0,-1.5) {$C$};
      \node(11) at (1.5,-1.5) {$C'$.}; 
      \draw [pto] (10) to node[auto, swap, labelsize] {$\Lambda$} (11); 
      \draw [->] (20) to node[auto,swap, labelsize] {$u$} (00); 
      \draw [->] (20) to node[auto, labelsize] {$u$} (01); 
      \draw [->] (00) to node[auto, swap, labelsize] {$v$} (10); 
      \draw [->] (01) to node[auto,labelsize] {$v'$} (11);
      \node at (0.75,-0.75) {$\kappa\circ (u)$}; 
\end{tikzpicture}\]
\end{itemize}
These data are subject to the evident unit and associativity axioms; see e.g.\ \cite[Definition~2.1]{Cruttwell-Shulman-unified} for more details.
\end{definition}

\begin{remark}\label{rmk:vdbl-cat-as-T-cat}
    One can identify virtual double categories with \emph{$T$-categories} \cite{Burroni-T-cats,Leinster-HOHC}, where $T$ is the free category monad on the category $\mathbf{Gph}$ of graphs.
    In particular, a virtual double category $\vdcat{A}$ has an underlying \emph{$T$-graph}, which is a diagram 
\begin{equation*}
\begin{tikzpicture}[baseline=-\the\dimexpr\fontdimen22\textfont2\relax ]
      \node(0) at (0,0) {$TA_0$};
      \node(1)at (2,0) {$A_1$};
      \node(2) at (4,0) {$A_0$};
      \draw [->] (1) to node[auto,swap, labelsize] {$d_1$}  (0);
      \draw [->] (1) to node[auto,labelsize] {$d_0$}  (2);
\end{tikzpicture}
\end{equation*}
    in $\mathbf{Gph}$. The graph $A_0$ has objects and horizontal morphisms of $\vdcat{A}$ as its vertices and edges, respectively, whereas the graph $A_1$ has vertical morphisms and multicells of $\vdcat{A}$ as its vertices and edges.
    See \cite[III.3]{Burroni-T-cats} or \cite[Chapter~5]{Leinster-HOHC} for details.
\end{remark}

\begin{definition}
Given virtual double categories $\vdcat{A}$ and $\vdcat{B}$, a \emph{virtual double functor} $F\colon\vdcat{A}\to\vdcat{B}$ consists of the following data.
\begin{itemize}
    \item First we have a functor $F_\vc\colon\vdcat{A}_\vc\to\vdcat{B}_\vc$ between the vertical categories. We write the image of an object $A$ or a vertical morphism $u$ in $\vdcat{A}$ under $F_\vc$ simply as $FA$ and $Fu$, respectively.
    \item We have a function mapping each horizontal morphism $\Gamma\colon A\pto A'$ in $\vdcat{A}$ to a horizontal morphism $F\Gamma\colon FA\pto FA'$ in $\vdcat{B}$.
    \item We have a function mapping each multicell 
\begin{equation*}
\begin{tikzpicture}[baseline=-\the\dimexpr\fontdimen22\textfont2\relax ]
      \node(00) at (0,0.75) {$A_0$};
      \node(01) at (1.5,0.75) {$A_1$};
      \node(02) at (3,0.75) {$\cdots$};
      \node(03) at (4.5,0.75) {$A_n$};
      \node(10) at (0,-0.75) {$A$};
      \node(11) at (4.5,-0.75) {$A'$}; 
      \draw [pto] (00) to node[auto, labelsize] {$\Gamma_1$} (01);  
      \draw [pto] (01) to node[auto, labelsize] {$\Gamma_2$} (02);  
      \draw [pto] (02) to node[auto, labelsize] {$\Gamma_n$} (03);  
      \draw [pto] (10) to node[auto, swap, labelsize] {$\Gamma$} (11); 
      \draw [->] (00) to node[auto, swap, labelsize] {$u$} (10); 
      \draw [->] (03) to node[auto,labelsize] {$u'$} (11);
      \node at (2.25,0) {$\theta$}; 
\end{tikzpicture}
\end{equation*}
    in $\vdcat{A}$ to a multicell 
\begin{equation*}
\begin{tikzpicture}[baseline=-\the\dimexpr\fontdimen22\textfont2\relax ]
      \node(00) at (0,0.75) {$FA_0$};
      \node(01) at (1.5,0.75) {$FA_1$};
      \node(02) at (3,0.75) {$\cdots$};
      \node(03) at (4.5,0.75) {$FA_n$};
      \node(10) at (0,-0.75) {$FA$};
      \node(11) at (4.5,-0.75) {$FA'$}; 
      \draw [pto] (00) to node[auto, labelsize] {$F\Gamma_1$} (01);  
      \draw [pto] (01) to node[auto, labelsize] {$F\Gamma_2$} (02);  
      \draw [pto] (02) to node[auto, labelsize] {$F\Gamma_n$} (03);  
      \draw [pto] (10) to node[auto, swap, labelsize] {$F\Gamma$} (11); 
      \draw [->] (00) to node[auto, swap, labelsize] {$Fu$} (10); 
      \draw [->] (03) to node[auto,labelsize] {$Fu'$} (11);
      \node at (2.25,0) {$F\theta$}; 
\end{tikzpicture}
\end{equation*}
    in $\vdcat{B}$.
\end{itemize}
These data are subject to the following axioms. 
\begin{itemize}
    \item For each horizontal morphism $\Gamma$ in $\vdcat{A}$, we have $F(\id_\Gamma)=\id_{F\Gamma}$.
    \item For each configuration \eqref{eqn:composable-multicells-A} in $\vdcat{A}$, we have $F\bigl(\kappa\circ(\theta_1,\theta_2,\dots,\theta_m)\bigr)=F\kappa\circ(F\theta_1,F\theta_2,\dots,F\theta_m)$. 
\end{itemize}
\end{definition}

\begin{definition}
Given a parallel pair of virtual double functors $F,G\colon\vdcat{A}\to\vdcat{B}$, a \emph{vertical transformation} $\alpha\colon F\to G$ consists of the following data.
\begin{itemize}
    \item First we have a natural transformation $\alpha_\vc\colon F_\vc\to G_\vc$. For each object $A$ in $\vdcat{A}$, its component at $A$ is a vertical morphism
\[
\begin{tikzpicture}[baseline=-\the\dimexpr\fontdimen22\textfont2\relax ]
      \node(0) at (0,0.75) {$FA$};
      \node(1) at (0,-0.75) {$GA$};
      \draw [->] (0) to node[auto, labelsize] {$\alpha_A$} (1); 
\end{tikzpicture}
\]
    in $\vdcat{B}$.
    \item For each horizontal morphism $\Gamma\colon A\pto A'$ in $\vdcat{A}$, we have a cell
\begin{equation*}
\begin{tikzpicture}[baseline=-\the\dimexpr\fontdimen22\textfont2\relax ]
      \node(00) at (0,0.75) {$FA$};
      \node(01) at (1.5,0.75) {$FA'$};
      \node(10) at (0,-0.75) {$GA$};
      \node(11) at (1.5,-0.75) {$GA'$}; 
      \draw [pto] (00) to node[auto, labelsize] {$F\Gamma$} (01);  
      \draw [pto] (10) to node[auto, swap, labelsize] {$G\Gamma$} (11); 
      \draw [->] (00) to node[auto, swap, labelsize] {$\alpha_A$} (10); 
      \draw [->] (01) to node[auto, labelsize] {$\alpha_{A'}$} (11);
      \node at (0.75,0) {$\alpha_\Gamma$}; 
\end{tikzpicture}
\end{equation*}
    in $\vdcat{B}$.
\end{itemize}
These data are subject to the following axiom.
\begin{itemize}
    \item For each multicell 
\begin{equation*}
\begin{tikzpicture}[baseline=-\the\dimexpr\fontdimen22\textfont2\relax ]
      \node(00) at (0,0.75) {$A_0$};
      \node(01) at (1.5,0.75) {$A_1$};
      \node(02) at (3,0.75) {$\cdots$};
      \node(03) at (4.5,0.75) {$A_n$};
      \node(10) at (0,-0.75) {$A$};
      \node(11) at (4.5,-0.75) {$A'$}; 
      \draw [pto] (00) to node[auto, labelsize] {$\Gamma_1$} (01);  
      \draw [pto] (01) to node[auto, labelsize] {$\Gamma_2$} (02);  
      \draw [pto] (02) to node[auto, labelsize] {$\Gamma_n$} (03);  
      \draw [pto] (10) to node[auto, swap, labelsize] {$\Gamma$} (11); 
      \draw [->] (00) to node[auto, swap, labelsize] {$u$} (10); 
      \draw [->] (03) to node[auto,labelsize] {$u'$} (11);
      \node at (2.25,0) {$\theta$}; 
\end{tikzpicture}
\end{equation*}
    in $\vdcat{A}$, we have the equality $\alpha_{\Gamma}\circ (F\theta)= G\theta\circ(\alpha_{\Gamma_1},\alpha_{\Gamma_2},\dots,\alpha_{\Gamma_n})$ in $\vdcat{B}$.
\end{itemize}
\end{definition}

Large virtual double categories, virtual double functors, and vertical transformations form a 2-category $\VDbl$.

\subsection{Enrichment over a virtual double category}\label{subsec:enrichment-over-vdbl}
We recall enriched category theory over a virtual double category, following \cite{Leinster-generalized-enrichment}.
Let $\vdcat{A}$ be a virtual double category.
\begin{definition}
    An \emph{$\vdcat{A}$-category} $\ecat{C}$ consists of the following data.
    \begin{itemize}
        \item We have a set $\ob(\ecat{C})$ of \emph{objects} of $\ecat{C}$. We write $c\in \ecat{C}$ to mean $c\in \ob(\ecat{C})$. We call the $\vdcat{A}$-category $\ecat{C}$ \emph{small} if $\ob(\ecat{C})$ is a small set.
        \item For each object $c\in \ecat{C}$, we have an object $|c|^\ecat{C}$ of $\vdcat{A}$, which is called the \emph{extent} of $c$. 
        \item For each $c,c'\in\ecat{C}$, we have a horizontal morphism $\ecat{C}(c,c')\colon |c|^\ecat{C}\pto |c'|^\ecat{C}$ in $\vdcat{A}$. 
        \item For each $c\in \ecat{C}$, we have a multicell
    \begin{equation*}
    \begin{tikzpicture}[baseline=-\the\dimexpr\fontdimen22\textfont2\relax ]
      \node(20) at (1,0.75) {$|c|^\ecat{C}$};
      \node(00) at (0,-0.75) {$|c|^\ecat{C}$};
      \node(01) at (2,-0.75) {$|c|^\ecat{C}$}; 
      \draw [pto] (00) to node[auto,swap,labelsize] {$\ecat{C}(c,c)$} (01);
      \draw [double equal sign distance] (20) to (00); 
      \draw [double equal sign distance] (20) to (01); 
      \node at (1,-0.2) {$\eta_c^\ecat{C}$};
\end{tikzpicture}
    \end{equation*}
        in $\vdcat{A}$. 
        \item For each $c_0,c_1,c_2\in\ecat{C}$, we have a multicell 
\begin{equation*}
      \begin{tikzpicture}[baseline=-\the\dimexpr\fontdimen22\textfont2\relax ]
      \node(00) at (0,0.75) {$|c_0|^\ecat{C}$};
      \node(01) at (2,0.75) {$|c_1|^\ecat{C}$};
      \node(02) at (4,0.75) {$|c_2|^\ecat{C}$};
      \node(10) at (0,-0.75) {$|c_0|^\ecat{C}$};
      \node(11) at (4,-0.75) {$|c_2|^\ecat{C}$}; 
      \draw [pto] (00) to node[auto,labelsize] {$\ecat{C}(c_0,c_1)$} (01);
      \draw [pto] (01) to node[auto,labelsize] {$\ecat{C}(c_1,c_2)$} (02);
      \draw [pto] (10) to node[auto, swap, labelsize] {$\ecat{C}(c_0,c_2)$} (11); 
      \draw [double equal sign distance] (00) to (10); 
      \draw [double equal sign distance] (02) to (11);
      \node at (2,0) {$\mu_{c_0,c_1,c_2}^\ecat{C}$}; 
\end{tikzpicture}
\end{equation*}
        in $\vdcat{A}$. 
    \end{itemize}
    These data are subject to the following axioms. 
    \begin{itemize}
        \item For each $c,c'\in\ecat{C}$, we have 
        \begin{equation*}
\begin{tikzpicture}[baseline=-\the\dimexpr\fontdimen22\textfont2\relax ]
      \node(20) at (1,1.5) {$|c|^\ecat{C}$};
      \node(21) at (3,1.5) {$|c'|^\ecat{C}$};
      \node(00) at (0,0) {$|c|^\ecat{C}$};
      \node(01) at (2,0) {$|c|^\ecat{C}$};
      \node(02) at (4,0) {$|c'|^\ecat{C}$};
      \node(10) at (0,-1.5) {$|c|^\ecat{C}$};
      \node(11) at (4,-1.5) {$|c'|^\ecat{C}$}; 
      \draw [pto] (20) to node[auto, labelsize] {$\ecat{C}(c,c')$} (21);
      \draw [pto] (00) to node[auto, swap,labelsize] {$\ecat{C}(c,c)$} (01);
      \draw [pto] (01) to node[auto, swap,labelsize] {$\ecat{C}(c,c')$} (02);
      \draw [pto] (10) to node[auto, swap, labelsize] {$\ecat{C}(c,c')$} (11); 
      \draw [double equal sign distance] (20) to (00); 
      \draw [double equal sign distance] (20) to (01); 
      \draw [double equal sign distance] (21) to (02); 
      \draw [double equal sign distance] (00) to (10); 
      \draw [double equal sign distance] (02) to (11);
      \node at (1,0.5) {$\eta_c^\ecat{C}$};
      \node at (2,-0.75) {$\mu_{c,c,c'}^\ecat{C}$}; 
\end{tikzpicture}
\quad = \quad \id_{\ecat{C}(c,c')} \quad = \quad
    \begin{tikzpicture}[baseline=-\the\dimexpr\fontdimen22\textfont2\relax ]
      \node(20) at (1,1.5) {$|c|^\ecat{C}$};
      \node(21) at (3,1.5) {$|c'|^\ecat{C}$};
      \node(00) at (0,0) {$|c|^\ecat{C}$};
      \node(01) at (2,0) {$|c'|^\ecat{C}$};
      \node(02) at (4,0) {$|c'|^\ecat{C}$};
      \node(10) at (0,-1.5) {$|c|^\ecat{C}$};
      \node(11) at (4,-1.5) {$|c'|^\ecat{C}$.}; 
      \draw [pto] (20) to node[auto, labelsize] {$\ecat{C}(c,c')$} (21);
      \draw [pto] (00) to node[auto, swap,labelsize] {$\ecat{C}(c,c')$} (01);
      \draw [pto] (01) to node[auto, swap,labelsize] {$\ecat{C}(c',c')$} (02);
      \draw [pto] (10) to node[auto, swap, labelsize] {$\ecat{C}(c,c')$} (11); 
      \draw [double equal sign distance] (20) to (00); 
      \draw [double equal sign distance] (21) to (01); 
      \draw [double equal sign distance] (21) to (02); 
      \draw [double equal sign distance] (00) to (10); 
      \draw [double equal sign distance] (02) to (11);
      \node at (3,0.5) {$\eta_{c'}^\ecat{C}$};
      \node at (2,-0.75) {$\mu_{c,c',c'}^\ecat{C}$}; 
\end{tikzpicture}
        \end{equation*}
        \item For each $c_0,c_1,c_2,c_3\in\ecat{C}$, we have 
        \begin{equation*}
            \begin{tikzpicture}[baseline=-\the\dimexpr\fontdimen22\textfont2\relax ]
      \node(00) at (0,1.5) {$|c_0|^\ecat{C}$};
      \node(01) at (2,1.5) {$|c_1|^\ecat{C}$};
      \node(02) at (4,1.5) {$|c_2|^\ecat{C}$};
      \node(03) at (6,1.5) {$|c_3|^\ecat{C}$};
      \node(10) at (0,0) {$|c_0|^\ecat{C}$};
      \node(12) at (4,0) {$|c_2|^\ecat{C}$};
      \node(13) at (6,0) {$|c_3|^\ecat{C}$};
      \node(20) at (0,-1.5) {$|c_0|^\ecat{C}$};
      \node(23) at (6,-1.5) {$|c_3|^\ecat{C}$}; 
      \draw [pto] (00) to node[auto, labelsize] {$\ecat{C}(c_0,c_1)$} (01);
      \draw [pto] (01) to node[auto, labelsize] {$\ecat{C}(c_1,c_2)$} (02);
      \draw [pto] (02) to node[auto, labelsize] {$\ecat{C}(c_2,c_3)$} (03);
      \draw [pto] (10) to node[auto, labelsize] {$\ecat{C}(c_0,c_2)$} (12);
      \draw [pto] (12) to node[auto, labelsize] {$\ecat{C}(c_2,c_3)$} (13);
      \draw [pto] (20) to node[auto, swap, labelsize] {$\ecat{C}(c_0,c_3)$} (23);
      \draw [double equal sign distance] (00) to (10);
      \draw [double equal sign distance] (02) to (12);
      \draw [double equal sign distance] (03) to (13);
      \draw [double equal sign distance] (10) to (20);
      \draw [double equal sign distance] (13) to (23);
      \node at (2,0.75) {$\mu_{c_0,c_1,c_2}^\ecat{C}$};
      \node at (3,-0.75) {$\mu_{c_0,c_2,c_3}^\ecat{C}$}; 
\end{tikzpicture}
\quad=\quad
\begin{tikzpicture}[baseline=-\the\dimexpr\fontdimen22\textfont2\relax ]
      \node(00) at (0,1.5) {$|c_0|^\ecat{C}$};
      \node(01) at (2,1.5) {$|c_1|^\ecat{C}$};
      \node(02) at (4,1.5) {$|c_2|^\ecat{C}$};
      \node(03) at (6,1.5) {$|c_3|^\ecat{C}$};
      \node(10) at (0,0) {$|c_0|^\ecat{C}$};
      \node(11) at (2,0) {$|c_1|^\ecat{C}$};
      \node(13) at (6,0) {$|c_3|^\ecat{C}$};
      \node(20) at (0,-1.5) {$|c_0|^\ecat{C}$};
      \node(23) at (6,-1.5) {$|c_3|^\ecat{C}$.}; 
      \draw [pto] (00) to node[auto, labelsize] {$\ecat{C}(c_0,c_1)$} (01);
      \draw [pto] (01) to node[auto, labelsize] {$\ecat{C}(c_1,c_2)$} (02);
      \draw [pto] (02) to node[auto, labelsize] {$\ecat{C}(c_2,c_3)$} (03);
      \draw [pto] (10) to node[auto, labelsize] {$\ecat{C}(c_0,c_1)$} (11);
      \draw [pto] (11) to node[auto, labelsize] {$\ecat{C}(c_1,c_3)$} (13);
      \draw [pto] (20) to node[auto, swap, labelsize] {$\ecat{C}(c_0,c_3)$} (23);
      \draw [double equal sign distance] (00) to (10);
      \draw [double equal sign distance] (01) to (11);
      \draw [double equal sign distance] (03) to (13);
      \draw [double equal sign distance] (10) to (20);
      \draw [double equal sign distance] (13) to (23);
      \node at (4,0.75) {$\mu_{c_1,c_2,c_3}^\ecat{C}$};
      \node at (3,-0.75) {$\mu_{c_0,c_1,c_3}^\ecat{C}$}; 
\end{tikzpicture}
        \end{equation*}
    \end{itemize}
\end{definition}

\begin{definition}
    Let $\ecat{C}$ and $\ecat{D}$ be $\vdcat{A}$-categories. An \emph{$\vdcat{A}$-functor} $f\colon \ecat{C}\to \ecat{D}$ consists of the following data. 
    \begin{itemize}
        \item First we have a function $\ob(f)\colon \ob(\ecat{C})\to \ob(\ecat{D})$. We write the image of $c\in \ecat{C}$ under this function as $fc\in \ecat{D}$. 
        \item For each $c\in \ecat{C}$, we have a vertical morphism 
\begin{equation*}
\begin{tikzpicture}[baseline=-\the\dimexpr\fontdimen22\textfont2\relax ]
      \node(00) at (0,0.75) {$|c|^\ecat{C}$};
      \node(10) at (0,-0.75) {$|fc|^\ecat{D}$};
      \draw [->] (00) to node[auto, labelsize] {$f_c$} (10); 
\end{tikzpicture}
\end{equation*}
    in $\vdcat{A}$. 
    \item For each $c,c'\in\ecat{C}$, we have a cell
        \begin{equation*}
\begin{tikzpicture}[baseline=-\the\dimexpr\fontdimen22\textfont2\relax ]
      \node(00) at (0,0.75) {$|c|^\ecat{C}$};
      \node(01) at (2,0.75) {$|c'|^\ecat{C}$};
      \node(10) at (0,-0.75) {$|fc|^\ecat{D}$};
      \node(11) at (2,-0.75) {$|fc'|^\ecat{D}$}; 
      \draw [pto] (00) to node[auto, labelsize] {$\ecat{C}(c,c')$} (01);  
      \draw [pto] (10) to node[auto, swap, labelsize] {$\ecat{D}(fc,fc')$} (11); 
      \draw [->] (00) to node[auto, swap, labelsize] {$f_c$} (10); 
      \draw [->] (01) to node[auto, labelsize] {$f_{c'}$} (11);
      \node at (1,0) {$\omega_{c,c'}^f$}; 
\end{tikzpicture}
\end{equation*}
    in $\vdcat{A}$. 
    \end{itemize}
    These data are subject to the following axioms.
\begin{itemize}
        \item For each $c\in \ecat{C}$, we have 
\begin{equation*}
    \begin{tikzpicture}[baseline=-\the\dimexpr\fontdimen22\textfont2\relax ]
      \node(20) at (1,1.5) {$|c|^\ecat{C}$};
      \node(00) at (0,0) {$|c|^\ecat{C}$};
      \node(01) at (2,0) {$|c|^\ecat{C}$};
      \node(10) at (0,-1.5) {$|fc|^\ecat{D}$};
      \node(11) at (2,-1.5) {$|fc|^\ecat{D}$}; 
      \draw [pto] (00) to node[auto, labelsize] {$\ecat{C}(c,c)$} (01);  
      \draw [pto] (10) to node[auto, swap, labelsize] {$\ecat{D}(fc,fc)$} (11); 
      \draw [double equal sign distance] (20) to node[auto, swap, labelsize] {} (00); 
      \draw [double equal sign distance] (20) to node[auto,labelsize] {} (01);
      \draw [->] (00) to node[auto, swap, labelsize] {$f_c$} (10); 
      \draw [->] (01) to node[auto,labelsize] {$f_c$} (11);
      \node at (1,0.65) {$\eta_c^\ecat{C}$}; 
      \node at (1,-0.75) {$\omega_{c,c}^f$}; 
\end{tikzpicture}
\quad=\quad
\begin{tikzpicture}[baseline=-\the\dimexpr\fontdimen22\textfont2\relax ]
      \node(20) at (1,1.5) {$|c|^\ecat{C}$};
      \node(00) at (1,0) {$|fc|^\ecat{D}$};
      \node(10) at (0,-1.5) {$|fc|^\ecat{D}$};
      \node(11) at (2,-1.5) {$|fc|^\ecat{D}$.}; 
      \draw [pto] (10) to node[auto, swap, labelsize] {$\ecat{D}(fc,fc)$} (11); 
      \draw [->] (20) to node[auto, labelsize] {$f_c$} (00); 
      \draw [double equal sign distance] (00) to (10); 
      \draw [double equal sign distance] (00) to (11);
      \node at (1,-0.95) {$\eta^\ecat{D}_{fc}$}; 
\end{tikzpicture}
\end{equation*}
    \item For each $c_0,c_1,c_2\in\ecat{C}$, we have 
\begin{equation*}
\label{eqn:monad-morphism-preserves-mult}
\begin{tikzpicture}[baseline=-\the\dimexpr\fontdimen22\textfont2\relax ]
      \node(00) at (0,1.5) {$|c_0|^\ecat{C}$};
      \node(02) at (2.5,1.5) {$|c_1|^\ecat{C}$};
      \node(01) at (5,1.5) {$|c_2|^\ecat{C}$};
      \node(10) at (0,0) {$|c_0|^\ecat{C}$};
      \node(11) at (5,0) {$|c_2|^\ecat{C}$}; 
      \node(20) at (0,-1.5) {$|fc_0|^\ecat{D}$};
      \node(21) at (5,-1.5) {$|fc_2|^\ecat{D}$}; 
      \draw [pto] (00) to node[auto, labelsize] {$\ecat{C}(c_0,c_1)$} (02);
      \draw [pto] (02) to node[auto, labelsize] {$\ecat{C}(c_1,c_2)$} (01);
      \draw [pto] (10) to node[auto, labelsize] {$\ecat{C}(c_0,c_2)$} (11);
      \draw [pto] (20) to node[auto,swap, labelsize] {$\ecat{D}(fc_0,fc_2)$} (21);
      \draw [double equal sign distance] (00) to (10); 
      \draw [double equal sign distance] (01) to (11);
      \draw [->] (10) to node[auto, swap, labelsize] {$f_{c_0}$} (20); 
      \draw [->] (11) to node[auto,labelsize] {$f_{c_2}$} (21);
      \node at (2.5,0.75) {$\mu^\ecat{C}_{c_0,c_1,c_2}$}; 
      \node at (2.5,-0.75) {$\omega_{c_0,c_2}^f$}; 
\end{tikzpicture}\qquad =\qquad 
\begin{tikzpicture}[baseline=-\the\dimexpr\fontdimen22\textfont2\relax ]
      \node(00) at (0,1.5) {$|c_0|^\ecat{C}$};
      \node(02) at (2.5,1.5) {$|c_1|^\ecat{C}$};
      \node(01) at (5,1.5) {$|c_2|^\ecat{C}$};
      \node(10) at (0,0) {$|fc_0|^\ecat{D}$};
      \node(12) at (2.5,0) {$|fc_1|^\ecat{D}$};
      \node(11) at (5,0) {$|fc_2|^\ecat{D}$}; 
      \node(20) at (0,-1.5) {$|fc_0|^\ecat{D}$};
      \node(21) at (5,-1.5) {$|fc_2|^\ecat{D}$.}; 
      \draw [pto] (00) to node[auto, labelsize] {$\ecat{C}(c_0,c_1)$} (02);
      \draw [pto] (02) to node[auto, labelsize] {$\ecat{C}(c_1,c_2)$} (01);
      \draw [pto] (10) to node[auto, labelsize] {$\ecat{D}(fc_0,fc_1)$} (12);
      \draw [pto] (12) to node[auto, labelsize] {$\ecat{D}(fc_1,fc_2)$} (11);
      \draw [pto] (20) to node[auto,swap, labelsize] {$\ecat{D}(fc_0,fc_2)$} (21); 
      \draw [->] (00) to node[auto, swap, labelsize] {$f_{c_0}$} (10); 
      \draw [->] (01) to node[auto,labelsize] {$f_{c_2}$} (11);
      \draw [->] (02) to node[auto,labelsize] {$f_{c_1}$} (12);
      \draw [double equal sign distance] (10) to (20); 
      \draw [double equal sign distance] (11) to (21);
      \node at (1.25,0.75) {$\omega_{c_0,c_1}^f$}; 
      \node at (3.75,0.75) {$\omega_{c_1,c_2}^f$}; 
      \node at (2.5,-0.75) {$\mu^\ecat{D}_{fc_0,fc_1,fc_2}$}; 
\end{tikzpicture}
\end{equation*}
\end{itemize}
\end{definition}

\begin{definition}
    Let $\ecat{C}$ and $\ecat{D}$ be $\vdcat{A}$-categories, and $f,{f}'\colon \ecat{C}\to \ecat{D}$ be $\vdcat{A}$-functors. An \emph{$\vdcat{A}$-natural transformation} $\alpha\colon {f}\to{f}'$ consists of a family of multicells 
\begin{equation*}
\begin{tikzpicture}[baseline=-\the\dimexpr\fontdimen22\textfont2\relax ]
      \node(20) at (1,0.75) {$|c|^\ecat{C}$};
      \node(00) at (0,-0.75) {$|fc|^\ecat{D}$};
      \node(01) at (2,-0.75) {$|f'c|^\ecat{D}$}; 
      \draw [pto] (00) to node[auto,swap,labelsize] {$\ecat{D}(fc,f'c)$} (01);
      \draw [->] (20) to node[auto,swap,labelsize] {${f}_c$} (00); 
      \draw [->] (20) to node[auto,labelsize] {${f}'_c$} (01); 
      \node at (1,-0.2) {$\alpha_c$};
\end{tikzpicture}
\end{equation*}
    in $\vdcat{A}$ indexed by $c\in \ecat{C}$, such that 
\begin{equation*}
\begin{tikzpicture}[baseline=-\the\dimexpr\fontdimen22\textfont2\relax ]
      \node(20) at (1.25,1.5) {$|c|^\ecat{C}$};
      \node(21) at (3.75,1.5) {$|c'|^\ecat{C}$};
      \node(00) at (0,0) {$|fc|^\ecat{D}$};
      \node(01) at (2.5,0) {$|f'c|^\ecat{D}$};
      \node(02) at (5,0) {$|f'c'|^\ecat{D}$};
      \node(10) at (0,-1.5) {$|fc|^\ecat{D}$};
      \node(11) at (5,-1.5) {$|f'c'|^\ecat{D}$}; 
      \draw [pto] (20) to node[auto, labelsize] {$\ecat{C}(c,c')$} (21);
      \draw [pto] (00) to node[auto, swap,labelsize] {$\ecat{D}(fc,f'c)$} (01);
      \draw [pto] (01) to node[auto, swap,labelsize] {$\ecat{D}(f'c,f'c')$} (02);
      \draw [pto] (10) to node[auto, swap, labelsize] {$\ecat{D}(fc,f'c')$} (11);
      \draw [->] (20) to node[auto,swap,labelsize] {$f_c$} (00); 
      \draw [->] (20) to node[auto,labelsize] {${f}'_c$} (01);
      \draw [->] (21) to node[auto, labelsize] {${f}'_{c'}$} (02); 
      \draw [double equal sign distance] (00) to (10); 
      \draw [double equal sign distance] (02) to (11);
      \node at (1.25,0.5) {$\alpha_c$};
      \node at (3.125,0.75) {$\omega^{f'}_{c,c'}$};
      \node at (2.5,-0.75) {$\mu^\ecat{D}_{fc,f'c,f'c'}$}; 
\end{tikzpicture}
\quad = \quad
\begin{tikzpicture}[baseline=-\the\dimexpr\fontdimen22\textfont2\relax ]
      \node(20) at (1.25,1.5) {$|c|^\ecat{C}$};
      \node(21) at (3.75,1.5) {$|c'|^\ecat{C}$};
      \node(00) at (0,0) {$|fc|^\ecat{D}$};
      \node(01) at (2.5,0) {$|fc'|^\ecat{D}$};
      \node(02) at (5,0) {$|f'c'|^\ecat{D}$};
      \node(10) at (0,-1.5) {$|fc|^\ecat{D}$};
      \node(11) at (5,-1.5) {$|f'c'|^\ecat{D}$}; 
      \draw [pto] (20) to node[auto, labelsize] {$\ecat{C}(c,c')$} (21);
      \draw [pto] (00) to node[auto, swap,labelsize] {$\ecat{D}(fc,fc')$} (01);
      \draw [pto] (01) to node[auto, swap,labelsize] {$\ecat{D}(fc',f'c')$} (02);
      \draw [pto] (10) to node[auto, swap, labelsize] {$\ecat{D}(fc,f'c')$} (11);
      \draw [->] (20) to node[auto,swap,labelsize] {${f}_c$} (00); 
      \draw [->] (21) to node[auto,swap,labelsize] {${f}_{c'}$} (01);
      \draw [->] (21) to node[auto, labelsize] {${f}'_{c'}$} (02); 
      \draw [double equal sign distance] (00) to (10); 
      \draw [double equal sign distance] (02) to (11);
      \node at (3.75,0.5) {$\alpha_{c'}$};
      \node at (1.875,0.75) {$\omega^f_{c,c'}$};
      \node at (2.5,-0.75) {$\mu^\ecat{D}_{fc,fc',f'c'}$}; 
\end{tikzpicture}
\end{equation*}
    holds for each $c,c'\in\ecat{C}$.
\end{definition}

Small $\vdcat{A}$-categories, $\vdcat{A}$-functors, and $\vdcat{A}$-natural transformations form a 2-category $\enCat{\vdcat{A}}$. The assignment $\vdcat{A}\mapsto \enCat{\vdcat{A}}$ routinely extends to a 2-functor $\Enr\colon \VDbl\to \twoCAT$. 

In particular, if we consider the case where $\vdcat{A}$ is the terminal virtual double category $1$, then $\Enr(1)$ (i.e., $\enCat{1}$; however, we will avoid the latter misleading notation) is the locally chaotic\footnote{A category is \emph{chaotic} (also known as indiscrete or codiscrete) if each of its hom-sets is a singleton. A $2$-category is \emph{locally chaotic} if each of its hom-categories is chaotic.} $2$-category whose underlying category is the category $\Set$ of small sets. 
For this reason, we write the $2$-category $\Enr(1)$ as $\Setlc$. 

\section{Explicit description of \texorpdfstring{$\vdcat{L}$}{L}}
\label{sec:L-explicitly}
In this section, we describe the left $2$-adjoint $\vdcat{L}$ of the $2$-functor $\Enr_1\colon \VDbl\to \twoCAT/\Enr(1)=\twoCAT/\Setlc$ explicitly.
First observe that for any $\vdcat{A}\in\VDbl$, the object $\Enr_1(\vdcat{A})\in \twoCAT/\Setlc$ is given by the $2$-functor $\ob\colon \enCat{\vdcat{A}}\to \Setlc$, sending each $\vdcat{A}$-category $\ecat{C}$ to its set $\ob(\ecat{C})$ of objects. 

Suppose that we are given $(F\colon \tcat{K}\to\Setlc)\in \twoCAT/\Setlc$. 
(In view of the (2-)adjunction
\[\begin{tikzpicture}[baseline=-\the\dimexpr\fontdimen22\textfont2\relax ]
      \node(0) at (0,0) {$\twoCAT$};
      \node(2) at (3,0) {$\CAT$,};
      \draw [->, transform canvas={yshift=5}] (0) to node[auto, labelsize] {$(-)_0$} (2); 
      \draw [<-, transform canvas={yshift=-5}] (0) to node[auto, swap, labelsize] {$(-)_\mathrm{lc}$} (2); 
      \path(0) to node[rotate=-90] {$\dashv$} (2);
\end{tikzpicture}\]
the $2$-functor $F\colon \tcat{K}\to \Setlc$ corresponds to a functor $\hat F\colon \tcat{K}_0\to \Set$ from the underlying category $\tcat{K}_0$ of $\tcat{K}$.)
We construct the virtual double category $\vdcat{L}F$ as follows.
\begin{itemize}
    \item An object of $\vdcat{L}F$ is a pair $(K,x)$ consisting of an object $K\in\tcat{K}$ and an element $x\in FK$.
    \item A vertical morphism of $\vdcat{L}F$ from $(K,x)$ to $(L,y)$ is a $1$-cell $f\colon K\to L$ of $\tcat{K}$ such that $(Ff)x=y$.
    (Thus the vertical category $(\vdcat{L}F)_\vc$ of $\vdcat{L}F$ is the category of elements $\Elts(\hat F)$ of the functor $\hat F\colon \tcat{K}_0\to \Set$ corresponding to the $2$-functor $F\colon \tcat{K}\to \Setlc$.)
    \item There exists at most one horizontal morphism between any given pair of objects of $\vdcat{L}F$;
    a horizontal morphism from $(K,x)$ to $(K',x')$ exists in $\vdcat{L}F$ if and only if $K=K'$.
    \item Given a configuration 
    \[\begin{tikzpicture}[baseline=-\the\dimexpr\fontdimen22\textfont2\relax ]
      \node(00) at (-0.5,0.75) {$(K,x_0)$};
      \node(01) at (1.5,0.75) {$(K,x_1)$};
      \node(02) at (3,0.75) {$\cdots$};
      \node(03) at (4.5,0.75) {$(K,x_n)$};
      \node(10) at (-0.5,-0.75) {$(L,y)$};
      \node(11) at (4.5,-0.75) {$(L,y')$}; 
      \draw [pto] (00) to node[auto, labelsize] {} (01);  
      \draw [pto] (01) to node[auto, labelsize] {} (02);  
      \draw [pto] (02) to node[auto, labelsize] {} (03);  
      \draw [pto] (10) to node[auto, swap, labelsize] {} (11); 
      \draw [->] (00) to node[auto, swap, labelsize] {$f$} (10); 
      \draw [->] (03) to node[auto,labelsize] {$f'$} (11);
    \end{tikzpicture}\]
    in $\vdcat{L}F$ (where the horizontal morphisms are left unnamed because they are determined by their domain and codomain), the multicells having this as a boundary are in bijection with the $2$-cells $f\to f'$ in $\tcat{K}$. 
\end{itemize}
Using suitable pasting of $2$-cells in $\tcat{K}$, it is straightforward to make $\vdcat{L}F$ into a virtual double category.

\begin{theorem}\label{thm:L}
This construction gives rise to the left $2$-adjoint of $\Enr_1$.  
\end{theorem}

\begin{proof}
    For any $(F\colon \tcat{K}\to\Setlc)\in \twoCAT/\Setlc$, 
there is a canonical $2$-functor $\eta_F\colon \tcat{K}\to \Enr(\vdcat{L}F)$ making the diagram
\begin{equation*}
\begin{tikzpicture}[baseline=-\the\dimexpr\fontdimen22\textfont2\relax ]
      \node(00) at (0,0.75) {$\tcat{K}$};
      \node(01) at (2,0.75) {$\enCat{\vdcat{L}F}$};
      \node(10) at (1,-0.75) {$\Setlc$}; 
      \draw [->] (00) to node[auto, labelsize] {$\eta_F$} (01);  
      \draw [->] (00) to node[auto, swap, labelsize] {$F$} (10); 
      \draw [->] (01) to node[auto, labelsize] {$\ob$} (10);
\end{tikzpicture}
\end{equation*}
commute. 
On objects, $\eta_F$ sends each $K\in\tcat{K}$ to a canonical $\vdcat{L}F$-category $\eta_F(K)$ with $\ob\bigl(\eta_F(K)\bigr)=FK$, such that the extent of $x\in FK$ is $(K,x)\in \vdcat{L}F$;
its action on $1$- and $2$-cells of $\tcat{K}$ is evident. 
It suffices to show that for any $\vdcat{A}\in\VDbl$, the composite 
\begin{multline*}
\VDbl(\vdcat{L}F,\vdcat{A})\xrightarrow{(\Enr_1)_{\vdcat{L}F,\vdcat{A}}}
\twoCAT/\Setlc\bigl(\Enr_1(\vdcat{L}F),\Enr_1(\vdcat{A}) \bigr)\\
\xrightarrow{\twoCAT/\Setlc(\eta_F,\Enr_1(\vdcat{A}))}
\twoCAT/\Setlc\bigl(F,\Enr_1(\vdcat{A}) \bigr)
\end{multline*}
is an isomorphism of categories.
We omit the details as one can also obtain the $2$-adjunction $\vdcat{L}\dashv \Enr_1$ by combining the results proved in what follows.
\end{proof}

\begin{remark}
    The $2$-category $\VDbl$ admits a fully faithful $2$-functor $\PsDbl\to \VDbl$ from the $2$-category $\PsDbl$ of (large) pseudo double categories, lax double functors, and vertical natural transformations.
    The essential image of $\PsDbl\to \VDbl$ consists of all \emph{representable} virtual double categories; see e.g.\ \cite[Theorem~5.2]{Cruttwell-Shulman-unified}.
    Now, in the virtual double category $\vdcat{L}F$ constructed above, any multicell corresponding to the identity $2$-cell on an identity $1$-cell in $\tcat{K}$ is opcartesian (in the sense of \cite[Definition~5.1]{Cruttwell-Shulman-unified}).
    It follows that $\vdcat{L}F$ is representable.
    Hence the $2$-functor $\vdcat{L}\colon \twoCAT/\Setlc\to \VDbl$ factors through $\PsDbl$ (in fact, its full sub-2-category $\StrDbl$ consisting of \emph{strict} double categories).
    Therefore it follows that the $2$-functor $\Enr\colon \PsDbl\to \twoCAT$ is also a parametric right $2$-adjoint. 
    For examples of categories enriched over a pseudo double category, see e.g.\ \cite{Cockett-Garner-restriction} and \cite[Definition~5.1]{Shulman-enriched-indexed}.
\end{remark}

\begin{remark}\label{rmk:BICAT-not-pra}
    The $2$-functor $\Enr\colon \BICAT\to \twoCAT$ is not a parametric right $2$-adjoint, i.e., the $2$-functor $\Enr_1\colon \BICAT\to \twoCAT/\Setlc$ does not have a left $2$-adjoint. 
    (See e.g.\ \cite{Street-cohomology} for the definitions of categories, functors, and natural transformations enriched over a bicategory; these are a special case of \cref{subsec:enrichment-over-vdbl}, where a bicategory is regarded as a virtual double category whose vertical category is discrete, and whose horizontal morphisms correspond to the $1$-cells of the original bicategory.)
    In \cite[Theorem~2.1]{Fujii_Lack}, we showed that $\Enr_1\colon \BICAT\to \twoCAT/\Setlc$ preserves all (weighted) limits which happen to exist in $\BICAT$. 
    Thus we cannot use arguments based on non-preservation of limits. Instead, we argue as follows.

Let $\cat{C}$ be the category with objects $x,y,z$ and non-identity
arrows $f\colon x\to z$ and $g\colon y\to z$. Regard it as a locally
discrete 2-category, and equip it with the 2-functor
$\Delta_1\colon\cat{C}\to\Setlc$ which is constant at the singleton
$1\in\Setlc$. This now determines an object $(\cat{C},\Delta_1)$ of $\twoCAT/\Setlc$. 
We claim that the $2$-functor 
\begin{equation}\label{eqn:Bicat-non-representable-functor}
\BICAT\xrightarrow{\Enr_1} \twoCAT/\Setlc\xrightarrow{\twoCAT/\Setlc((\cat{C},\Delta_1),-)}\CAT
\end{equation}
is not representable.
For
a bicategory $\bcat{B}$, to give a map
$(\cat{C},\Delta_1)\to\Enr_1(\bcat{B})$ in $\twoCAT/\Setlc$ we should give three one-object
$\bcat{B}$-categories --- in other words, three monads in $\bcat{B}$
--- and two $\bcat{B}$-functors between them. But
$\bcat{B}$-functors respect extent, and so the three one-object
$\bcat{B}$-categories would all have to live on the same object of
$\bcat{B}$.
Now, suppose to the contrary that the $2$-functor \eqref{eqn:Bicat-non-representable-functor} were representable by $\bcat{B}\in\BICAT$.

\subsection*{Step 1: the bicategory $\bcat{B}$ has a single object.}

We have a $2$-adjunction 
\[\begin{tikzpicture}[baseline=-\the\dimexpr\fontdimen22\textfont2\relax ]
      \node(0) at (0,0) {$\mathbf{SET}_\mathrm{ld}$};
      \node(2) at (3,0) {$\BICAT$,};
      \draw [<-, transform canvas={yshift=5}] (0) to node[auto, labelsize] {$\ob$} (2); 
      \draw [->, transform canvas={yshift=-5}] (0) to node[auto, swap, labelsize] {$(-)_\mathrm{ch}$} (2); 
      \path(0) to node[rotate=-90] {$\dashv$} (2);
\end{tikzpicture}\]
where $\mathbf{SET}_\mathrm{ld}$ is the locally discrete $2$-category whose underlying category is the category $\mathbf{SET}$ of large sets, $\ob$ maps each bicategory to its set of objects, and $(-)_\mathrm{ch}$ maps each set $E$ to the \emph{chaotic} bicategory $E_\mathrm{ch}$ whose set of objects is $E$ and each of whose hom-categories is the terminal category. 
Our assumption on $\bcat{B}$ implies that we have isomorphisms of categories 
\[
\mathbf{SET}_\mathrm{ld}\bigl(\ob(\bcat{B}),E\bigr)\cong 
\BICAT(\bcat{B},E_\mathrm{ch})\cong
\twoCAT/\Setlc\bigl((\cat{C},\Delta_1),\Enr_1(E_\mathrm{ch})\bigr)
\]
for each $E\in\mathbf{SET}_\mathrm{ld}$.
Since $\cat{C}$ is connected, the rightmost term is the discrete category whose set of objects is $E$. Hence $\ob(\bcat{B})$ is a singleton.

Therefore $\bcat{B}$ corresponds to a monoidal category $\cat{V}$.
We thus have a universal monoidal category $\cat{V}$ equipped with
a functor $N\colon\cat{C}\to \mathrm{Mon}(\cat{V})$.

\subsection*{Step 2: the unit $I$ of $\cat{V}$ must be initial.}
Next consider $I/\cat{V}$ with the lifted monoidal structure from
$\cat{V}$. 
The (strict monoidal) forgetful functor $U\colon I/\cat{V}\to \cat{V}$ induces an isomorphism of categories $\mathrm{Mon}(U)\colon \mathrm{Mon}(I/\cat{V})\to\mathrm{Mon}(\cat{V})$. 
Consider the composite $\cat{C}\xrightarrow{N}\mathrm{Mon}(\cat{V})\xrightarrow{\mathrm{Mon}(U)^{-1}}\mathrm{Mon}(I/\cat{V})$. By the universal property of $N$, there exists a unique lax monoidal functor $H\colon \cat{V}\to I/\cat{V}$ such that $\mathrm{Mon}(U)^{-1}N=\mathrm{Mon}(H)N$, i.e., $N=\mathrm{Mon}(UH)N$. 
The uniqueness part of the universal property of $N$ implies $UH=1_\cat{V}$. 
Therefore we have $\mathrm{Mon}(U)\mathrm{Mon}(H)=1_{\mathrm{Mon}(\cat{V})}$, and since $\mathrm{Mon}(U)$ is an isomorphism, $\mathrm{Mon}(U)^{-1}=\mathrm{Mon}(H)$ holds. 
Since $\mathrm{Mon}(U)$ maps the unit monoid on $(1_I\colon I\to I) \in I/\cat{V}$ to the unit monoid on $I\in\cat{V}$, the functor $H$ must map the object $I\in\cat{V}$ to $(1_I\colon I\to I)$. 
Now, $UH=1_\cat{V}$ implies that $H$ corresponds to a cone $(\tau_X\colon I\to X)_{X\in\cat{V}}$ over the identity functor $1_\cat{V}$, and $HI=1_I$ means that we have $\tau_I=1_I$. 
Hence $I$ is initial in $\cat{V}$. 

\subsection*{Step 3: the unit monoid $I$ must be in the image of $N$.}

Now let $\cat{V}'$ be the monoidal category obtained from $\cat{V}$ by
adjoining a new object $I'$ isomorphic to $I$. The inclusion
$K\colon\cat{V}\to\cat{V}'$ is lax monoidal, as is the functor $K'\colon
\cat{V}\to\cat{V}'$ which is identical to $K$ except that
$K'I=I'$. Since $K\neq K'$ we must have $\mathrm{Mon}(K)N\neq
\mathrm{Mon}(K')N$. But $K$ and $K'$ are identical except at $I$, and since $I$ is initial in $\cat{V}'$,
this can only happen if $I$ is in the image of $N$. 

\subsection*{Step 4: this is impossible.}

We now know that $NW=I$ for some object $W\in\cat{C}$.
Finally consider the cocartesian monoidal category
$\Fam(\cat{C})$ and the fully faithful inclusion
$J\colon\cat{C}\to\Fam(\cat{C})=\mathrm{Mon}(\Fam(\cat{C}))$. There would then
have to be a unique monoidal functor $L\colon\cat{V}\to\Fam(\cat{C})$
such that $\mathrm{Mon}(L)N=J$. For any object $C\in\cat{C}$ we have a
map $I=NW\to NC$ in $\mathrm{Mon}(\cat{V})$, and so a map
$JW=\mathrm{Mon}(L)NW\to \mathrm{Mon}(L)NC=JC$; and so, since $J$ is
full, a map $W\to C$ in $\cat{C}$. But there is no object in
$\cat{C}$ with a map to every other object. 
\end{remark}

\begin{remark}\label{rmk:MonCAT-not-pra}
    The $2$-functor $\Enr\colon \MonCAT\to \twoCAT$ is not a parametric right $2$-adjoint.
    In this case, we can show that the $2$-functor $\Enr_1\colon \MonCAT\to \twoCAT/\Setlc$ does not preserve a certain (fortuitous) equalizer in $\MonCAT$.
    Consider the parallel pair of morphisms 
\[
\begin{tikzpicture}[baseline=-\the\dimexpr\fontdimen22\textfont2\relax ]
      \node(0) at (0,0) {$\mathbb{Z}$};
      \node(1) at (3,0) {$\{a,b\}_\mathrm{ch}$};
      \draw [->, transform canvas={yshift=3}] (0) to node[auto, labelsize] {$f$}  (1);  
      \draw [->, transform canvas={yshift=-3}] (0) to node[auto, swap, labelsize] {$g$}  (1); 
\end{tikzpicture}
\]
    in $\MonCAT$, where 
    \begin{itemize}
        \item $\mathbb{Z}$ is the additive monoid (in fact group) of integers regarded as a discrete monoidal category,
        \item $\{a,b\}_\mathrm{ch}$ is the chaotic category with two objects $a,b$ regarded as a monoidal category (details of the definitions of the unit and monoidal product do not matter),
        \item $f$ is the strong monoidal functor constant at $a$, and
        \item $g$ is the strong monoidal functor sending $-1,0,1\in\mathbb{Z}$ to $a$ and all other objects of $\mathbb{Z}$ to $b$. 
    \end{itemize}
    Then this parallel pair has an equalizer in $\MonCAT$, namely the terminal monoidal category $1$; the limit cone is the unique lax (in fact strict) monoidal functor $e\colon 1\to \mathbb{Z}$, sending the unique object of $1$ to $0\in\mathbb{Z}$. 

    Now, an object of $\Enr(\mathbb{Z})=\enCat{\mathbb{Z}}$ is a set $X$ together with a function $d\colon X\times X\to \mathbb{Z}$ satisfying $d(x,x)=0$ and $d(x,y)+d(y,z)=d(x,z)$. 
    Such an object is in the image of the $2$-functor $\Enr(e)\colon \Enr(1)\to \Enr(\mathbb{Z})$ if and only if $d\colon X\times X\to \mathbb{Z}$ is constant at $0$, whereas it is in the equalizer of $\Enr(f)$ and $\Enr(g)$ if and only if the image of $d\colon X\times X\to \mathbb{Z}$ is contained in $\{-1,0,1\}$. Now, let $X=\{x,y\}$, $d(x,x)=d(y,y)=0$, $d(x,y)=1$, and $d(y,x)=-1$; this object is in the equalizer of $\Enr(f)$ and $\Enr(g)$, but not in the image of $\Enr(e)$.
\end{remark}

\section{Parametric right 2-adjoints, familial 2-functors, and polynomial 2-functors}\label{sec:poly-pra-fam}
In this section, we recall basic facts about these notions.

\subsection{Parametric right $2$-adjoints}\label{subsec:pra}
We start with the notion of \emph{parametric right $2$-adjoint}
\cite{Weber-familial}, which is a $2$-categorical version of the notion of \emph{parametric right adjoint} introduced in \cite[Section~5]{Street-petit-topos}. 
As mentioned in \cref{sec:intro}, a $2$-functor $R\colon \tcat{K}\to\tcat{L}$ from a $2$-category $\tcat{K}$ with a terminal object $1$ is a \emph{parametric right $2$-adjoint} when the $2$-functor $R_1\colon \tcat{K}\to \tcat{L}/R1$ is a right $2$-adjoint.

\begin{proposition}
    For a $2$-functor $R\colon\tcat{K}\to\tcat{L}$ where $\tcat{K}$ has a terminal object, the following are equivalent:
    \begin{enumerate}
    \item $R$ is a parametric right $2$-adjoint
        \item $R_1\colon\tcat{K}\to\tcat{L}/R1$ has a left $2$-adjoint
        \item $R_X\colon\tcat{K}/X\to\tcat{L}/RX$ has a left $2$-adjoint for every object $X$ of $\tcat{K}$.
    \end{enumerate}
\end{proposition}

\begin{proof}
    The equivalence of (1) and (2) is just the definition of parametric right $2$-adjoint. The equivalence of (2) and (3) is part of \cite[Proposition~2.6]{Weber-familial}.
\end{proof}

A straightforward consequence is the following.
\begin{proposition}
    Parametric right $2$-adjoints are closed under composition.
\end{proposition}
\begin{proof}
    Use the characterization (3) in the previous proposition.
\end{proof}
\begin{proposition}\label{2-adj-is-para-2-adj}
    If $\tcat{K}$ is a $2$-category with terminal object, any right $2$-adjoint $R\colon \tcat{K}\to \tcat{L}$ is a parametric right $2$-adjoint. 
\end{proposition}
\begin{proof}
Up to isomorphism, $R_1$ is just $R$.
\end{proof}

A parametric right 2-adjoint $R\colon\tcat{K}\to\tcat{L}$ factorizes
as a right adjoint $R_1$ followed by the forgetful
$\tcat{L}/R1\to\tcat{L}$, and so will preserve any type of limit which
is preserved by the forgetful; in particular, it will preserve any
pullbacks and equalizers which exist in $\tcat{K}$. But it may not
preserve some 2-categorical or bicategorical limits that we would like to be
preserved. To deal with this, we now turn to the stronger notion of
familial 2-functor \cite{Weber-familial}.

\subsection{Familial $2$-functors}
\label{subsec:familial}

Familial 2-functors rely on the idea of split fibration in a
2-category, which we now recall. A morphism $p\colon E\to B$ in a
2-category $\tcat{K}$ is said to be a {\em split fibration} when
\begin{itemize}
\item for each object $X$, the induced functor
  $\tcat{K}(X,p)\colon\tcat{K}(X,E)\to\tcat{K}(X,B)$ is equipped with the
  structure of a split fibration, so that in particular for any
  $e\colon X\to E$, $b\colon X\to B $, and $\beta\colon b\to pe$ there
  is a chosen cartesian lift $\overline\beta\colon \beta^\ast e\to e$
  \item for each $f\colon Y\to X$, the induced functor
    $\tcat{K}(f,E)\colon\tcat{K}(X,E)\to\tcat{K}(Y,E)$ preserves chosen
    cartesian lifts.
\end{itemize}
In the case $\tcat{K}=\CAT$ we recover the usual notion of split fibration.
Being a split fibration is structure that a morphism may bear, not
just a property: there can be more than one choice of cartesian lifts.

Given an object $B$ of the 2-category $\tcat{K}$, there is a
2-category $\Spl(B)$ equipped with a forgetful 2-functor
$U_B\colon\Spl(B)\to\tcat{K}/B$ consisting of the split fibrations $p\colon
E\to B$, morphisms between these preserving chosen cartesian lifts,
and the 2-cells are defined as in $\tcat{K}/B$; see
\cite[Section~3]{Weber-familial}. We may also write
$\Spl_{\tcat{K}}(B)$ rather than $\Spl(B)$ if there is any doubt about
the 2-category
$\tcat{K}$ in question. 

An important special case is where each
$\tcat{K}(X,p)$ has a unique lift, i.e., is a discrete fibration of categories; then the preservation condition is
automatic, and $p$ is said to be a discrete fibration. There are analogous definitions of split opfibration and discrete opfibration. All of these notions are defined representably, being defined via the corresponding notions in $\CAT$, and one sometimes speaks of representable discrete fibration, and so on, if there might be doubt as to which notion is meant. 

In particular, in the following section we shall meet a notion of discrete opfibration for virtual double categories that is strictly stronger than the one defined representably in $\VDbl$. In this paper, we shall always use {\em representable discrete opfibration} to refer to the notion defined here, namely a morphism $p\colon E\to B$ for which each $\tcat{K}(X,p)\colon \tcat{K}(X,E)\to\tcat{K}(X,B)$ is a discrete opfibration in $\CAT$.

We are now ready to make the definition.

\begin{definition}
  Let $\tcat{K}$ be a 2-category with a terminal object. A 2-functor
  $R\colon \tcat{K}\to\tcat{L}$ is {\em familial} if it is a
  parametric right $2$-adjoint, and the $2$-functor
  $R_1\colon\tcat{K}\to\tcat{L}/R1$ factorizes as
  \[
    \begin{tikzcd}
      \tcat{K} \rar["\overline{R}_1"] & \Spl(R1)
      \rar["U_{R1}"] & \tcat{L}/R1.
    \end{tikzcd}
  \]
\end{definition}
This means that:
\begin{itemize}
\item The image $RA\to R1$ of the unique $A\to 1$ under $R$ is given
  the structure of a split fibration in $\tcat{L}$
  \item The map $Rf\colon RA\to RB$ preserves chosen cartesian lifts,
    for each $f\colon A\to B$.
\end{itemize}

It follows as in \cite{Weber-familial} that any familial 2-functor
preserves isocomma objects up to equivalence, and so
preserves bicategorical pullbacks. The key point is that the forgetful
$\Spl(R1)\to\tcat{L}$ has better limit-preservation properties than
$\tcat{L}/R1\to\tcat{L}$. It also follows that $R$ preserves
fibrations, bicategorical fibrations, iso-fibrations, and so
on. It follows from \cite[Theorem~6.6]{Weber-familial} that familial 2-functors are closed under composition.

The motivating example is the 2-functor $\Fam\colon\CAT\to\CAT$; this
and many others can be found in \cite{Weber-familial}. 
The following somewhat degenerate example will nevertheless be important in what follows.

\begin{example}\label{ex:right-adj-familial}
    If $1$ is a terminal object in $\tcat{L}$ then $\Spl_{\tcat{L}}(1)\cong\tcat{L}$. Thus any right 2-adjoint between 2-categories with terminal objects is familial.
\end{example}

We shall see various further examples in later sections of the paper.

\subsection{Polynomial $2$-functors}
\label{subsec:poly}

Let $\tcat{K}$ be a $2$-category with $2$-pullbacks. Then for any $1$-cell $p\colon X\to Y$ in $\tcat{K}$, we have a $2$-adjunction 
\[
\begin{tikzpicture}[baseline=-\the\dimexpr\fontdimen22\textfont2\relax ]
      \node(0) at (0,0) {$\tcat{K}/X$};
      \node(1) at (4,0) {$\tcat{K}/Y$};
      \draw [->, transform canvas={yshift=5}] (0) to node[auto, labelsize] {$\sum_p$}  (1);  
      \draw [->, transform canvas={yshift=-5}] (1) to node[auto, labelsize] {$p^\ast$}  (0); 
      \path(0) to node[rotate=-90] {$\dashv$} (1);
\end{tikzpicture}
\]
between the (strict) slice $2$-categories, 
where $\sum_p$ maps each $(f\colon A\to X)$ to $(pf\colon {A}\to {Y})$ and 
$p^\ast$ maps each $(g\colon {B}\to {Y})$ to $(p^\ast g\colon p^\ast {B}\to {X})$ defined as the pullback
\begin{equation*}
\begin{tikzpicture}[baseline=-\the\dimexpr\fontdimen22\textfont2\relax ]
      \node(00) at (0,0.75) {$p^\ast {B}$};
      \node(01) at (1.5,0.75) {${B}$};
      \node(10) at (0,-0.75) {${X}$};
      \node(11) at (1.5,-0.75) {${Y}$}; 
      \draw [->] (00) to node[auto, labelsize] {} (01);  
      \draw [->] (10) to node[auto, swap, labelsize] {$p$} (11); 
      \draw [->] (00) to node[auto, swap, labelsize] {$p^\ast g$} (10); 
      \draw [->] (01) to node[auto, labelsize] {$g$} (11);
      \draw (0.2,0.25) to (0.5,0.25) to (0.5,0.55);
\end{tikzpicture}
\end{equation*}
in $\tcat{K}$. (Strictly speaking, the notation $p^\ast {B}$ is not precise because it also depends on $g$. The same remark applies to the notation $\prod_P{\vdcat{A}}$ introduced in the proof of \cref{prop:disc-opfib-powerful} below.)

We say that the $1$-cell $p\colon X\to Y$ in $\tcat{K}$ is
\emph{powerful} (or \emph{exponentiable}) if the $2$-functor
$p^\ast\colon\tcat{K}/Y\to\tcat{K}/X$ has a right $2$-adjoint
$\prod_p$.

If $\tcat{K}$ is a 2-category with pullbacks and we are given
morphisms
\[
  \begin{tikzcd}
    I & E \lar["s" '] \rar["p"] & B \rar["t"] & J
  \end{tikzcd}
\]
in $\tcat{K}$ with $p$ powerful (in which case the above diagram is called a \emph{polynomial} in $\tcat{K}$), then there is an induced 2-functor
\[
  \begin{tikzcd}
    \tcat{K}/I \rar["s^{*}"] & \tcat{K}/E \rar["\prod_p"] & \tcat{K}/B \rar["\sum_t"]
    & \tcat{K}/J
  \end{tikzcd}
\]
which we call $\mathcal{P}(s,p,t)$, 
and a 2-functor arising in this way is said to be {\em polynomial}. In
the important special case where $I=J=1$, so that $s$ and $t$ are
uniquely determined by their domains, we write $\mathcal{P}(p)$ rather
than $\mathcal{P}(s,p,t)$.

\begin{proposition}
  Any polynomial $2$-functor is a parametric right adjoint.
\end{proposition}

\begin{proof}
$\mathcal{P}(s,p,t)_1$ is just $\prod_p s^{*}$ which is right adjoint to
$\sum_s p^{*}$. 
\end{proof}

\begin{proposition}\label{prop:poly-familial}
  Let $\tcat{K}$ be a $2$-category with finite limits, and $p\colon E\to
  B$ a morphism in $\tcat{K}$ which is both powerful and 
  a representable discrete opfibration.
  Then $\mathcal{P}(p)\colon \tcat{K}\to\tcat{K}$ is a
  familial $2$-functor.
\end{proposition}

\begin{proof}
  This is a special case of \cite[Proposition~4.4.5]{Weber-polynomials}.
\end{proof}

\section{Powerfulness of discrete opfibrations}\label{sec:dof}
The aim of this section is to show that any discrete opfibration between virtual double categories is powerful in the $2$-dimensional sense defined in \cref{subsec:poly} (\cref{prop:disc-opfib-powerful}).
We note that the $1$-dimensional aspect of this is a special case of a result of Bowler: combining Section~2.3 and Proposition~3.1.6 of \cite{Bowler-PhD}, one concludes that all \emph{discrete Conduch\'e fibrations} (as in the latter proposition) between virtual double categories are powerful in the $1$-dimensional sense. 
We give an explicit description of the $2$-functor $\prod_P$ induced by a discrete opfibration $P$ of virtual double categories, as this is a key ingredient in the various constructions for virtual double categories we will study below.

We now introduce the notion of discrete opfibration of virtual double categories, strictly stronger than the representable notion, that was referred to in \cref{subsec:familial}.

\begin{definition}[{\cite[IV.2]{Burroni-T-cats} and \cite[Section~6.3]{Leinster-HOHC}}]
\label{def:disc-opfib}
    A virtual double functor $P\colon \vdcat{X}\to\vdcat{Y}$ is a \emph{discrete opfibration} if the following conditions are satisfied. 
    \begin{itemize}
        \item For each object $X\in \vdcat{X}$ and each vertical morphism $v\colon PX\to Y$ in $\vdcat{Y}$, there exists a unique vertical morphism $\overline{v}_X\colon X\to v_\ast X$ in $\vdcat{X}$ with $P\overline v_X=v$. (That is, the functor $P_\vc\colon \vdcat{X}_\vc\to \vdcat{Y}_\vc$ is a discrete opfibration of categories.)
        \item For each sequence of horizontal morphisms 
        \begin{equation*}
            \begin{tikzpicture}[baseline=-\the\dimexpr\fontdimen22\textfont2\relax ]
      \node(00) at (0,0) {$X_0$};
      \node(01) at (1.5,0) {$X_1$};
      \node(02) at (3,0) {$\cdots$};
      \node(03) at (4.5,0) {$X_n$}; 
      \draw [pto] (00) to node[auto, labelsize] {$\Phi_1$} (01);  
      \draw [pto] (01) to node[auto, labelsize] {$\Phi_2$} (02);  
      \draw [pto] (02) to node[auto, labelsize] {$\Phi_n$} (03);    
\end{tikzpicture}
        \end{equation*}
        in $\vdcat{X}$ and each multicell 
        \begin{equation*}
\begin{tikzpicture}[baseline=-\the\dimexpr\fontdimen22\textfont2\relax ]
      \node(00) at (0,0.75) {$PX_0$};
      \node(01) at (1.5,0.75) {$PX_1$};
      \node(02) at (3,0.75) {$\cdots$};
      \node(03) at (4.5,0.75) {$PX_n$};
      \node(10) at (0,-0.75) {$Y$};
      \node(11) at (4.5,-0.75) {$Y'$}; 
      \draw [pto] (00) to node[auto, labelsize] {$P\Phi_1$} (01);  
      \draw [pto] (01) to node[auto, labelsize] {$P\Phi_2$} (02);  
      \draw [pto] (02) to node[auto, labelsize] {$P\Phi_n$} (03);  
      \draw [pto] (10) to node[auto, swap, labelsize] {$\Psi$} (11); 
      \draw [->] (00) to node[auto, swap, labelsize] {$v$} (10); 
      \draw [->] (03) to node[auto,labelsize] {$v'$} (11);
      \node at (2.25,0) {$\nu$}; 
\end{tikzpicture}
\end{equation*}
    in $\vdcat{Y}$, there exists a unique multicell 
            \begin{equation*}
\begin{tikzpicture}[baseline=-\the\dimexpr\fontdimen22\textfont2\relax ]
      \node(00) at (0,0.75) {$X_0$};
      \node(01) at (1.5,0.75) {$X_1$};
      \node(02) at (3,0.75) {$\cdots$};
      \node(03) at (4.5,0.75) {$X_n$};
      \node(10) at (0,-0.75) {$v_\ast X_0$};
      \node(11) at (4.5,-0.75) {$v'_\ast X_n$}; 
      \draw [pto] (00) to node[auto, labelsize] {$\Phi_1$} (01);  
      \draw [pto] (01) to node[auto, labelsize] {$\Phi_2$} (02);  
      \draw [pto] (02) to node[auto, labelsize] {$\Phi_n$} (03);  
      \draw [pto] (10) to node[auto, swap, labelsize] {$\nu_\ast (\Phi_1,\dots,\Phi_n)$} (11); 
      \draw [->] (00) to node[auto, swap, labelsize] {$\overline v_{X_0}$} (10); 
      \draw [->] (03) to node[auto,labelsize] {$\overline{v'}_{X_n}$} (11);
      \node at (2.25,0) {$\overline \nu_{\Phi_1,\dots,\Phi_n}$}; 
\end{tikzpicture}
\end{equation*}
    in $\vdcat{Y}$ with $P\overline \nu_{\Phi_1,\dots,\Phi_n}=\nu$.
    \end{itemize}
\end{definition}

\begin{remark}\label{rmk:dof-for-T-cats}
    Regarding virtual double categories as $T$-categories (see \cref{rmk:vdbl-cat-as-T-cat}), a virtual double functor $P\colon \vdcat{X}\to \vdcat{Y}$ is a discrete opfibration in the sense of \cref{def:disc-opfib} if and only if the left square in the following commutative diagram is a pullback in $\mathbf{Gph}$:
\begin{equation*}
\begin{tikzpicture}[baseline=-\the\dimexpr\fontdimen22\textfont2\relax ]
      \node(0) at (0,0) {$TX_0$};
      \node(1)at (2,0) {$X_1$};
      \node(2) at (4,0) {$X_0$};
      \node(00) at (0,-1.5) {$TY_0$};
      \node(01)at (2,-1.5) {$Y_1$};
      \node(02) at (4,-1.5) {$Y_0$.};
      \draw [->] (1) to node[auto,swap, labelsize] {$d_1$}  (0);
      \draw [->] (1) to node[auto,labelsize] {$d_0$}  (2);
      \draw [->] (01) to node[auto, labelsize] {$d_1$}  (00);
      \draw [->] (01) to node[auto,swap,labelsize] {$d_0$}  (02);
      \draw [->] (0) to node[auto,swap, labelsize] {$TP_0$}  (00);
      \draw [->] (1) to node[auto,labelsize] {$P_1$}  (01);
      \draw [->] (2) to node[auto,labelsize] {$P_0$}  (02);
\end{tikzpicture}
\end{equation*}
    Discrete opfibrations are defined in these terms in \cite[IV.2]{Burroni-T-cats} (where they are said to be \emph{\'etale}) and \cite[Section~6.3]{Leinster-HOHC}, and further studied in \cite{Tholen-Yeganeh}.
\end{remark}

The proof of the following fact is straightforward. 

\begin{proposition}\label{dof-between-vdbl-is-internal-dof}
    Any discrete opfibration of virtual double categories in the sense of \cref{def:disc-opfib} is a representable discrete opfibration in $\VDbl$. 
\end{proposition}

\begin{remark}\label{rmk:dof-vs-internal-dof}
    The converse of \cref{dof-between-vdbl-is-internal-dof} does not hold. 
    For example, for any virtual double category $\vdcat{Y}$, let $\vdcat{X}$ be the virtual double subcategory of $\vdcat{Y}$ obtained by discarding all multicells which are not cells, while retaining all objects, vertical morphisms, horizontal morphisms, and cells. Then the inclusion $P\colon \vdcat{X}\to \vdcat{Y}$ is a representable discrete opfibration in $\VDbl$. (That is, the functor $\VDbl(\vdcat{A},P)\colon \VDbl(\vdcat{A},\vdcat{X})\to \VDbl(\vdcat{A},\vdcat{Y})$ is a discrete opfibration for any $\vdcat{A}\in \VDbl$. To see this, observe that if $\vdcat{A}$ has a multicell which is not a cell, then the category $\VDbl(\vdcat{A},\vdcat{X})$ is empty, and otherwise, $\VDbl(\vdcat{A},P)$ is an isomorphism of categories.) 
    However, $P\colon \vdcat{X}\to\vdcat{Y}$ fails to be a discrete opfibration in the sense of \cref{def:disc-opfib} whenever $\vdcat{Y}$ has a multicell which is not a cell. 
\end{remark}

In what follows, whenever we talk about discrete opfibrations between virtual double categories, we mean those in the sense of \cref{def:disc-opfib} rather than representable discrete opfibrations in $\VDbl$.

\begin{proposition}
    \label{prop:disc-opfib-powerful}
    Every discrete opfibration $P\colon \vdcat{X}\to\vdcat{Y}$ between virtual double categories is powerful; that is, we have a $2$-adjunction 
    \begin{equation}
\label{eqn:prod-ast}
\begin{tikzpicture}[baseline=-\the\dimexpr\fontdimen22\textfont2\relax ]
      \node(0) at (0,0) {$\VDbl/\vdcat{X}$};
      \node(1) at (4,0) {$\VDbl/\vdcat{Y}$.};
      \draw [<-, transform canvas={yshift=5}] (0) to node[auto, labelsize] {$P^\ast$}  (1);  
      \draw [<-, transform canvas={yshift=-5}] (1) to node[auto, labelsize] {$\prod_P$}  (0); 
      \path(0) to node[rotate=-90] {$\dashv$} (1);
\end{tikzpicture}
\end{equation}
\end{proposition}
\begin{proof}
We proceed as follows. First, for any $(F\colon \vdcat{A}\to \vdcat{X})\in\VDbl/\vdcat{X}$, we define $(\prod_PF\colon \prod_P\vdcat{A}\to \vdcat{Y})\in\VDbl/\vdcat{Y}$ and a morphism $\varepsilon_{F}\colon (P^\ast\prod_PF\colon P^\ast \prod_P\vdcat{A}\to \vdcat{X})\to (F\colon \vdcat{A}\to \vdcat{X})$ in $\VDbl/\vdcat{X}$. We then see that for any $(G\colon \vdcat{B}\to \vdcat{Y})\in \VDbl/\vdcat{Y}$, the composite
\begin{multline}
\label{eqn:composite-epsilon-powerful}
(\VDbl/\vdcat{Y})\left(
\left(
\begin{tikzpicture}[baseline=-\the\dimexpr\fontdimen22\textfont2\relax ]
      \node(00) at (0,0.4) {$\vdcat{B}$};
      \node(10) at (0,-0.4) {$\vdcat{Y}$}; 
      \draw [->] (00) to node[auto, labelsize] {$G$} (10);  
\end{tikzpicture}
\right),
\left(
\begin{tikzpicture}[baseline=-\the\dimexpr\fontdimen22\textfont2\relax ]
      \node(00) at (0,0.4) {$\prod_P\vdcat{A}$};
      \node(10) at (0,-0.4) {$\vdcat{Y}$}; 
      \draw [->] (00) to node[auto, labelsize] {$\prod_PF$} (10);  
\end{tikzpicture}
\right)\right)
\xrightarrow{P^\ast_{G,\prod_PF}}
(\VDbl/\vdcat{X})\left(
\left(
\begin{tikzpicture}[baseline=-\the\dimexpr\fontdimen22\textfont2\relax ]
      \node(00) at (0,0.4) {$P^\ast\vdcat{B}$};
      \node(10) at (0,-0.4) {$\vdcat{X}$}; 
      \draw [->] (00) to node[auto, labelsize] {$P^\ast G$} (10);  
\end{tikzpicture}
\right),
\left(
\begin{tikzpicture}[baseline=-\the\dimexpr\fontdimen22\textfont2\relax ]
      \node(00) at (0,0.4) {$P^\ast \prod_P\vdcat{A}$};
      \node(10) at (0,-0.4) {$\vdcat{X}$}; 
      \draw [->] (00) to node[auto, labelsize] {$P^\ast \prod_PF$} (10);  
\end{tikzpicture}
\right)\right)\\
\xrightarrow{(\VDbl/\vdcat{X})(1,\varepsilon_F)}
(\VDbl/\vdcat{X})\left(
\left(
\begin{tikzpicture}[baseline=-\the\dimexpr\fontdimen22\textfont2\relax ]
      \node(00) at (0,0.4) {$P^\ast\vdcat{B}$};
      \node(10) at (0,-0.4) {$\vdcat{X}$}; 
      \draw [->] (00) to node[auto, labelsize] {$P^\ast G$} (10);  
\end{tikzpicture}
\right),
\left(
\begin{tikzpicture}[baseline=-\the\dimexpr\fontdimen22\textfont2\relax ]
      \node(00) at (0,0.4) {$\vdcat{A}$};
      \node(10) at (0,-0.4) {$\vdcat{X}$}; 
      \draw [->] (00) to node[auto, labelsize] {$F$} (10);  
\end{tikzpicture}
\right)\right)
\end{multline}
is an isomorphism of categories. 
It then follows that the assignment $(\vdcat{A},F)\mapsto (\prod_P\vdcat{A},\prod_PF)$ uniquely extends to a $2$-functor so that $\varepsilon$ becomes $2$-natural, and becomes the counit of the $2$-adjunction \cref{eqn:prod-ast}. 

For each object $Y\in\vdcat{Y}$, we denote by $P^{-1}(Y)$ the set of all objects $X\in\vdcat{X}$ with $PX=Y$, and for each horizontal morphism $\Psi$ in $\vdcat{Y}$, we denote by $P^{-1}(\Psi)$ the set of all horizontal morphisms $\Phi$ in $\vdcat{X}$ with $P\Phi=\Psi$.
Take any $(F\colon \vdcat{A}\to \vdcat{X})\in \VDbl/\vdcat{X}$. We define the virtual double category $\prod_P\vdcat{A}$ as follows. 
\begin{itemize}
    \item An object in $\prod_P\vdcat{A}$ is a pair $(Y,\vect{A})$ of an object $Y\in\vdcat{Y}$ and a family $\vect{A}=(A_X)_{X\in P^{-1}(Y)}$ of objects in $\vdcat{A}$, satisfying $FA_X=X$ for each $X\in P^{-1}(Y)$.
    \item A vertical morphism in $\prod_P\vdcat{A}$ from $(Y,\vect{A})$ to $(Y',\vect{A}')$ is a pair $(v,\vect{f})$ of a vertical morphism $v\colon Y\to Y'$ in $\vdcat{Y}$ and a family $\vect{f}=(f_X\colon A_X\to A'_{v_\ast X})_{X\in P^{-1}(Y)}$ of vertical morphisms in $\vdcat{A}$, satisfying $Ff_X= \overline v_X\colon X\to v_\ast X$ for each $X\in P^{-1}(Y)$.
    (The vertical category $(\prod_P\vdcat{A})_\vc$ of $\prod_P\vdcat{A}$ is the category $\prod_{P_\vc}\vdcat{A}_\vc$ obtained by applying the $2$-functor $\prod_{P_\vc}\colon\CAT/\vdcat{X}_\vc\to \CAT/\vdcat{Y}_\vc$ to the object $(F_\vc\colon \vdcat{A}_\vc\to \vdcat{X}_\vc)\in \CAT/\vdcat{X}_\vc$.)
    \item A horizontal morphism in $\prod_P\vdcat{A}$ from $(Y,\vect{A})$ to $(Y',\vect{A}')$ is a pair $(\Psi,\vect{\Gamma})$ of a horizontal morphism $\Psi\colon Y\pto Y'$ in $\vdcat{Y}$ and a family $\vect{\Gamma}$ of horizontal morphisms which consists of, for each $(\Phi\colon X\pto X')\in P^{-1}(\Psi)$, a horizontal morphism $\Gamma_\Phi\colon A_X\pto A'_{X'}$ in $\vdcat{A}$ satisfying $F\Gamma_\Phi=\Phi$.
    \item A multicell in $\prod_P\vdcat{A}$ of type
\begin{equation}
\label{eqn:multicell-in-ProdP}
\begin{tikzpicture}[baseline=-\the\dimexpr\fontdimen22\textfont2\relax ]
      \node(00) at (0,0.75) {$(Y_0,\vect{A}_0)$};
      \node(01) at (2.5,0.75) {$(Y_1,\vect{A}_1)$};
      \node(02) at (5,0.75) {$\cdots$};
      \node(03) at (7.5,0.75) {$(Y_n,\vect{A}_n)$};
      \node(10) at (0,-0.75) {$(Y,\vect{A})$};
      \node(11) at (7.5,-0.75) {$(Y',\vect{A}')$};
      \draw [pto] (00) to node[auto, labelsize] {$(\Psi_1,\vect{\Gamma}_1)$} (01);
      \draw [pto] (01) to node[auto, labelsize] {$(\Psi_2,\vect{\Gamma}_2)$} (02);
      \draw [pto] (02) to node[auto, labelsize] {$(\Psi_n,\vect{\Gamma}_n)$} (03);
      \draw [pto] (10) to node[auto, swap, labelsize] {$(\Psi,\vect{\Gamma})$} (11);
      \draw [->] (00) to node[auto, swap, labelsize] {$(v,\vect{f})$} (10); 
      \draw [->] (03) to node[auto,labelsize] {$(v',\vect{f}')$} (11);
      \node at (3.75,0) {$(\nu,\vect{\alpha})$};
\end{tikzpicture}
\end{equation}
is a pair of a multicell 
\begin{equation*}
\begin{tikzpicture}[baseline=-\the\dimexpr\fontdimen22\textfont2\relax ]
      \node(00) at (0,0.75) {$Y_0$};
      \node(01) at (1.5,0.75) {$Y_1$};
      \node(02) at (3,0.75) {$\cdots$};
      \node(03) at (4.5,0.75) {$Y_n$};
      \node(10) at (0,-0.75) {$Y$};
      \node(11) at (4.5,-0.75) {$Y'$}; 
      \draw [pto] (00) to node[auto, labelsize] {$\Psi_1$} (01);  
      \draw [pto] (01) to node[auto, labelsize] {$\Psi_2$} (02);  
      \draw [pto] (02) to node[auto, labelsize] {$\Psi_n$} (03);  
      \draw [pto] (10) to node[auto, swap, labelsize] {$\Psi$} (11); 
      \draw [->] (00) to node[auto, swap, labelsize] {$v$} (10); 
      \draw [->] (03) to node[auto,labelsize] {$v'$} (11);
      \node at (2.25,0) {$\nu$}; 
\end{tikzpicture}
\end{equation*}
in $\vdcat{Y}$ and a family $\vect{\alpha}$ of multicells which consists of, for each sequence of horizontal morphisms 
\begin{equation*}
    \begin{tikzpicture}[baseline=-\the\dimexpr\fontdimen22\textfont2\relax ]
      \node(00) at (0,0) {$X_0$};
      \node(01) at (1.5,0) {$X_1$};
      \node(02) at (3,0) {$\cdots$};
      \node(03) at (4.5,0) {$X_n$};
      \draw [pto] (00) to node[auto, labelsize] {$\Phi_1$} (01);
      \draw [pto] (01) to node[auto, labelsize] {$\Phi_2$} (02);
      \draw [pto] (02) to node[auto, labelsize] {$\Phi_n$} (03);
\end{tikzpicture}
\end{equation*}
in $\vdcat{X}$ which is mapped by $P$ to 
\begin{equation*}
    \begin{tikzpicture}[baseline=-\the\dimexpr\fontdimen22\textfont2\relax ]
      \node(00) at (0,0) {$Y_0$};
      \node(01) at (1.5,0) {$Y_1$};
      \node(02) at (3,0) {$\cdots$};
      \node(03) at (4.5,0) {$Y_n$,};
      \draw [pto] (00) to node[auto, labelsize] {$\Psi_1$} (01);
      \draw [pto] (01) to node[auto, labelsize] {$\Psi_2$} (02);
      \draw [pto] (02) to node[auto, labelsize] {$\Psi_n$} (03);
\end{tikzpicture}
\end{equation*}
a multicell 
\begin{equation*}
\begin{tikzpicture}[baseline=-\the\dimexpr\fontdimen22\textfont2\relax ]
      \node(00) at (0,0.75) {$A_{0,X_0}$};
      \node(01) at (2,0.75) {$A_{1,X_1}$};
      \node(02) at (4,0.75) {$\cdots$};
      \node(03) at (6,0.75) {$A_{n,X_n}$};
      \node(10) at (0,-0.75) {$A_{v_\ast X_0}$};
      \node(11) at (6,-0.75) {$A'_{v'_\ast X_n}$}; 
      \draw [pto] (00) to node[auto, labelsize] {${\Gamma}_{1,\Phi_1}$} (01);  
      \draw [pto] (01) to node[auto, labelsize] {${\Gamma}_{2,\Phi_2}$} (02);  
      \draw [pto] (02) to node[auto, labelsize] {${\Gamma}_{n,\Phi_n}$} (03);  
      \draw [pto] (10) to node[auto, swap, labelsize] {${\Gamma}_{\nu_\ast(\Phi_1,\dots,\Phi_n)}$} (11); 
      \draw [->] (00) to node[auto, swap, labelsize] {$f_{X_0}$} (10); 
      \draw [->] (03) to node[auto,labelsize] {$f'_{X_n}$} (11);
      \node at (3,0) {$\alpha_{\Phi_1,\dots,\Phi_n}$}; 
\end{tikzpicture}
\end{equation*}
in $\vdcat{A}$ satisfying $F\alpha_{\Phi_1,\dots,\Phi_n}= \overline{\nu}_{\Phi_1,\dots,\Phi_n}$.
\end{itemize}
With the evident operations, $\prod_P\vdcat{A}$ becomes a virtual double category.
We have the virtual double functor $\prod_PF\colon\prod_P\vdcat{A}\to \vdcat{Y}$ given by the first projection.
This completes the definition of $(\prod_PF\colon \prod_P\vdcat{A}\to\vdcat{Y})\in\VDbl/\vdcat{Y}$.

Now take the pullback 
\begin{equation*}
\begin{tikzpicture}[baseline=-\the\dimexpr\fontdimen22\textfont2\relax ]
      \node(00) at (0,0.75) {$P^\ast \prod_P\vdcat{A}$};
      \node(01) at (2,0.75) {$\prod_P\vdcat{A}$};
      \node(10) at (0,-0.75) {$\vdcat{X}$};
      \node(11) at (2,-0.75) {$\vdcat{Y}$}; 
      \draw [->] (00) to node[auto, labelsize] {} (01);  
      \draw [->] (10) to node[auto, swap, labelsize] {$P$} (11); 
      \draw [->] (00) to node[auto, swap, labelsize] {$P^\ast \prod_PF$} (10); 
      \draw [->] (01) to node[auto, labelsize] {$\prod_PF$} (11);
      \draw (0.2,0.25) to (0.5,0.25) to (0.5,0.55);
\end{tikzpicture}
\end{equation*}
in $\VDbl$.
Explicitly, the virtual double category $P^\ast \prod_P\vdcat{A}$ can be described as follows. 
\begin{itemize}
    \item An object in $P^\ast \prod_P\vdcat{A}$ is a triple $(Y,\vect{A},\dot{X})$ consisting of an object $(Y,\vect{A})$ in $\prod_P\vdcat{A}$ and an object $\dot{X}\in \vdcat{X}$ satisfying $(\prod_PF)(Y,\vect{A})=P\dot{X}$, or equivalently, $\dot{X}\in P^{-1}(Y)$. 
    \item A vertical morphism in $P^\ast \prod_P\vdcat{A}$ from $(Y,\vect{A},\dot{X})$ to $(Y',\vect{A}',\dot{X}')$ is a triple $(v,\vect{f},u)$ consisting of a vertical morphism $(v,\vect{f})\colon (Y,\vect{A})\to (Y',\vect{A}')$ in $\prod_P\vdcat{A}$ and a vertical morphism $u\colon \dot{X}\to\dot{X}'$ in $\vdcat{X}$ satisfying $(\prod_PF)(v,\vect{f})=Pu$. Notice that we necessarily have $u=\overline v_{\dot{X}}$, and hence a vertical morphism in $P^\ast \prod_P\vdcat{A}$ from $(Y,\vect{A},\dot{X})$ to $(Y',\vect{A}',\dot{X}')$ can be equivalently given by a vertical morphism $(v,\vect{f})\colon (Y,\vect{A})\to (Y',\vect{A}')$ in $\prod_P\vdcat{A}$ satisfying $v_\ast \dot{X} = \dot{X}'$. 
    \item A horizontal morphism in $P^\ast \prod_P\vdcat{A}$ from $(Y,\vect{A},\dot{X})$ to $(Y',\vect{A}',\dot{X}')$ is a triple $(\Psi,\vect{\Gamma},\dot{\Phi})$ consisting of a horizontal morphism $(\Psi,\vect{\Gamma})\colon (Y,\vect{A})\pto (Y',\vect{A}')$ in $\prod_P\vdcat{A}$ and a horizontal morphism $\dot{\Phi}\colon \dot{X}\pto \dot{X}'$ in $\vdcat{X}$ satisfying $(\prod_PF)(\Psi,\vect{\Gamma})=P\dot{\Phi}$, or equivalently, $\dot{\Phi}\in P^{-1}(\Psi)$.
    \item A multicell in $P^\ast\prod_P\vdcat{A}$ of type
        \begin{equation}
\label{eqn:multicell-in-PastProdP}
\begin{tikzpicture}[baseline=-\the\dimexpr\fontdimen22\textfont2\relax ]
      \node(00) at (0,0.75) {$(Y_0,\vect{A}_0,\dot{X}_0)$};
      \node(01) at (3.5,0.75) {$(Y_1,\vect{A}_1,\dot{X}_1)$};
      \node(02) at (6.5,0.75) {$\cdots$};
      \node(03) at (9.5,0.75) {$(Y_n,\vect{A}_n,\dot{X}_n)$};
      \node(10) at (0,-0.75) {$(Y,\vect{A},\dot{X})$};
      \node(11) at (9.5,-0.75) {$(Y',\vect{A}',\dot{X}')$}; 
      \draw [pto] (00) to node[auto, labelsize] {$(\Psi_1,\vect{\Gamma}_1,\dot{\Phi}_1)$} (01);    
      \draw [pto] (01) to node[auto, labelsize] {$(\Psi_2,\vect{\Gamma}_2,\dot{\Phi}_2)$} (02);   
      \draw [pto] (02) to node[auto, labelsize] {$(\Psi_n,\vect{\Gamma}_n,\dot{\Phi}_n)$} (03);  
      \draw [pto] (10) to node[auto, swap, labelsize] {$(\Psi,\vect{\Gamma},\dot{\Phi})$} (11); 
      \draw [->] (00) to node[auto, swap, labelsize] {$(v,\vect{f},u)$} (10); 
      \draw [->] (03) to node[auto,labelsize] {$(v',\vect{f}',u)$} (11);
      \node at (4.75,0) {$(\nu,\vect{\alpha},\chi)$}; 
\end{tikzpicture}
\end{equation}
    consists of a multicell \eqref{eqn:multicell-in-ProdP} in $\prod_P\vdcat{A}$ and a multicell 
    \begin{equation*}
\begin{tikzpicture}[baseline=-\the\dimexpr\fontdimen22\textfont2\relax ]
      \node(00) at (0,0.75) {$\dot{X}_0$};
      \node(01) at (1.5,0.75) {$\dot{X}_1$};
      \node(02) at (3,0.75) {$\cdots$};
      \node(03) at (4.5,0.75) {$\dot{X}_n$};
      \node(10) at (0,-0.75) {$v_\ast \dot{X}_0$};
      \node(11) at (4.5,-0.75) {$v'_\ast \dot{X}_n$}; 
      \draw [pto] (00) to node[auto, labelsize] {$\dot{\Phi}_1$} (01);  
      \draw [pto] (01) to node[auto, labelsize] {$\dot{\Phi}_2$} (02);  
      \draw [pto] (02) to node[auto, labelsize] {$\dot{\Phi}_n$} (03);  
      \draw [pto] (10) to node[auto, swap, labelsize] {$\dot{\Phi}$} (11); 
      \draw [->] (00) to node[auto, swap, labelsize] {$\overline v_{\dot{X}_0}$} (10); 
      \draw [->] (03) to node[auto,labelsize] {$\overline{v'}_{\dot{X}_n}$} (11);
      \node at (2.25,0) {$\chi$}; 
\end{tikzpicture}
\end{equation*}
    in $\vdcat{X}$ satisfying $(\prod_PF)(\nu,\vect{\alpha})=P\chi$. Notice that we necessarily have $\chi=\overline \nu_{\dot{\Phi}_1,\dots,\dot{\Phi}_n}$, and hence the multicell $(\nu,\vect{\alpha},\chi)$ can be equivalently given by the multicell $(\nu,\vect{\alpha})$ in $\prod_P\vdcat{A}$ satisfying $\nu_\ast (\dot{\Phi}_1,\dots,\dot{\Phi}_n)=\dot{\Phi}$. 
\end{itemize}
The virtual double functor $P^\ast \prod_PF\colon P^\ast\prod_P\vdcat{A}\to \vdcat{X}$ is given by the third projection. 

Next we define the morphism $\varepsilon_{F}\colon (P^\ast\prod_PF\colon P^\ast \prod_P\vdcat{A}\to \vdcat{X})\to (F\colon \vdcat{A}\to \vdcat{X})$ in $\VDbl/\vdcat{X}$, i.e., a virtual double functor $\varepsilon_F\colon P^\ast\prod_P\vdcat{A}\to \vdcat{A}$ making the diagram 
\begin{equation*}
\begin{tikzpicture}[baseline=-\the\dimexpr\fontdimen22\textfont2\relax ]
      \node(00) at (0,0.75) {$P^\ast \prod_P\vdcat{A}$};
      \node(01) at (2,0.75) {$\vdcat{A}$};
      \node(10) at (1,-0.75) {$\vdcat{X}$};
      \draw [->] (00) to node[auto, labelsize] {$\varepsilon_F$} (01);   
      \draw [->] (00) to node[auto, swap, labelsize] {$P^\ast \prod_PF$} (10); 
      \draw [->] (01) to node[auto, labelsize] {$F$} (10);
\end{tikzpicture}
\end{equation*}
commute. 
\begin{itemize}
    \item An object $(Y,\vect{A},\dot{X})\in P^\ast\prod_P\vdcat{A}$ is mapped by $\varepsilon_F$ to $A_{\dot{X}}\in\vdcat{A}$. 
    \item A vertical morphism $(v,\vect{f},u)\colon (Y,\vect{A},\dot{X})\to (Y',\vect{A}',\dot{X}')$ in $P^\ast \prod_P\vdcat{A}$ is mapped by $\varepsilon_F$ to $f_{\dot{X}}\colon A_{\dot{X}}\to A'_{v_\ast \dot{X}}=A'_{\dot{X}'}$ in $\vdcat{A}$.
    \item A horizontal morphism $(\Psi,\vect{\Gamma},\dot{\Phi})\colon (Y,\vect{A},\dot{X})\pto (Y',\vect{A}',\dot{X}')$ in $P^\ast \prod_P\vdcat{A}$ is mapped by $\varepsilon_F$ to $\Gamma_{\dot{\Phi}}\colon A_{\dot{X}}\pto A'_{\dot{X}'}$ in $\vdcat{A}$. 
    \item A multicell \eqref{eqn:multicell-in-PastProdP} in $P^\ast \prod_P\vdcat{A}$ is mapped by $\varepsilon_F$ to 
    \begin{equation*}
\begin{tikzpicture}[baseline=-\the\dimexpr\fontdimen22\textfont2\relax ]
      \node(00) at (0,0.75) {$A_{0,\dot{X}_0}$};
      \node(01) at (2,0.75) {$A_{1,\dot{X}_1}$};
      \node(02) at (4,0.75) {$\cdots$};
      \node(03) at (6,0.75) {$A_{n,\dot{X}_n}$};
      \node(10) at (0,-0.75) {$A_{\dot{X}}$};
      \node(11) at (6,-0.75) {$A'_{\dot{X}'}$}; 
      \draw [pto] (00) to node[auto, labelsize] {${\Gamma}_{1,\dot{\Phi}_1}$} (01);  
      \draw [pto] (01) to node[auto, labelsize] {${\Gamma}_{2,\dot{\Phi}_2}$} (02);  
      \draw [pto] (02) to node[auto, labelsize] {${\Gamma}_{n,\dot{\Phi}_n}$} (03);  
      \draw [pto] (10) to node[auto, swap, labelsize] {${\Gamma}_{\dot{\Phi}}$} (11); 
      \draw [->] (00) to node[auto, swap, labelsize] {$f_{\dot{X}_0}$} (10); 
      \draw [->] (03) to node[auto,labelsize] {$f'_{\dot{X}_n}$} (11);
      \node at (3,0) {$\alpha_{\dot{\Phi}_1,\dots,\dot{\Phi}_n}$}; 
\end{tikzpicture}
\end{equation*}
in $\vdcat{A}$.
\end{itemize}

Now it is straightforward to check that \eqref{eqn:composite-epsilon-powerful} is an isomorphism of categories.
\end{proof}

In a forthcoming paper \cite{TCat-Conduche}, we give a far larger class of virtual double functors which are powerful. In fact we work in the more general setting of $T$-categories for a monad $T$, subject to various conditions. As recalled in \cref{rmk:vdbl-cat-as-T-cat}, virtual double categories correspond to the case where $T$ is the monad on $\Gph$ whose Eilenberg--Moore category is $\Cat$.

\section{The polynomial 2-functor \texorpdfstring{$\MMat$}{Mat}}
\label{sec:MMat}
In this section, we define the $2$-functor $\MMat\colon \VDbl\to \VDbl$ which appears as the first factor in the decomposition \eqref{eqn:Enr-decomposition} of $\Enr\colon \VDbl\to\twoCAT$.

First observe that the $2$-functor $(-)_\vc\colon \VDbl\to\CAT$, mapping each virtual double category $\vdcat{A}$ to its vertical category $\vdcat{A}_\vc$, has a right $2$-adjoint:
\begin{equation*}
\begin{tikzpicture}[baseline=-\the\dimexpr\fontdimen22\textfont2\relax ]
      \node(0) at (0,0) {$\VDbl$};
      \node(0a)at (4,0) {$\CAT$.};
      \draw [->, transform canvas={yshift=5}] (0) to node[auto, labelsize] {$(-)_\vc$}  (0a);
      \draw [->, transform canvas={yshift=-5}] (0a) to node[auto, labelsize] {$(-)_\hc$}  (0);
      \path(0) to node[rotate=-90] {$\dashv$} (0a);
\end{tikzpicture}
\end{equation*}
The right $2$-adjoint $(-)_\hc$ maps each category $\cat{C}$ to the virtual double category $\cat{C}_\hc$ whose vertical category is $\cat{C}$, and such that 
\begin{itemize}
    \item for each pair $C,C'$ of objects of $\cat{C}$, there exists precisely one horizontal morphism $C\pto C'$ in $\cat{C}_\hc$, and 
    \item for each configuration 
\begin{equation}
\label{eqn:boundary-of-2-cell-in-C-hc}
\begin{tikzpicture}[baseline=-\the\dimexpr\fontdimen22\textfont2\relax ]
      \node(00) at (0,0.75) {$C_0$};
      \node(01) at (1.5,0.75) {$C_1$};
      \node(02) at (3,0.75) {$\cdots$};
      \node(03) at (4.5,0.75) {$C_n$};
      \node(10) at (0,-0.75) {$C$};
      \node(11) at (4.5,-0.75) {$C'$}; 
      \draw [pto] (00) to node[auto, labelsize] {} (01);  
      \draw [pto] (01) to node[auto, labelsize] {} (02);  
      \draw [pto] (02) to node[auto, labelsize] {} (03);  
      \draw [pto] (10) to node[auto, swap, labelsize] {} (11); 
      \draw [->] (00) to node[auto, swap, labelsize] {$f$} (10); 
      \draw [->] (03) to node[auto,labelsize] {$f'$} (11);
\end{tikzpicture}    
\end{equation}
    in $\cat{C}_\hc$, there exists precisely one multicell having \cref{eqn:boundary-of-2-cell-in-C-hc} as the boundary. 
\end{itemize}
We call a virtual double category \emph{horizontally chaotic} if it is in the essential image of $(-)_\hc$.

\begin{remark}
    Although we shall not use this fact, the $2$-functor $(-)_\vc$ also has a left $2$-adjoint $(-)_\he\colon \CAT\to \VDbl$, mapping each category $\cat{C}$ to the ``horizontally empty'' virtual double category $\cat{C}_\he$ whose vertical category is $\cat{C}$, and such that there is no horizontal morphisms (and hence no multicells) in $\cat{C}_\he$. Note that we have $(-)_\vc\cong \VDbl(1_\he,-)$, where $1$ denotes the terminal category.

    The composite 
    \begin{equation*}
\begin{tikzpicture}[baseline=-\the\dimexpr\fontdimen22\textfont2\relax ]
      \node(0) at (0,0) {$\VDbl$};
      \node(1)at (2,0) {$\CAT$};
      \node(2) at (4,0) {$\VDbl$};
      \draw [->] (0) to node[auto, labelsize] {$(-)_\vc$}  (1);
      \draw [->] (1) to node[auto,labelsize] {$(-)_\he$}  (2);
\end{tikzpicture}
\end{equation*}
    coincides with the product $1_\he\times(-)\colon\VDbl\to\VDbl$ with $1_\he$, and hence its right $2$-adjoint, namely the composite 
    \begin{equation*}
\begin{tikzpicture}[baseline=-\the\dimexpr\fontdimen22\textfont2\relax ]
      \node(0) at (0,0) {$\VDbl$};
      \node(1)at (2,0) {$\CAT$};
      \node(2) at (4,0) {$\VDbl$,};
      \draw [->] (0) to node[auto, labelsize] {$(-)_\vc$}  (1);
      \draw [->] (1) to node[auto,labelsize] {$(-)_\hc$}  (2);
\end{tikzpicture}
\end{equation*}
    is the exponentiation $(-)^{1_\he}\colon \VDbl\to\VDbl$ by $1_\he$. (Although $\VDbl$ is not cartesian closed, the object $1_\he$ is powerful. In fact, for any $\cat{C}\in\CAT$, the object $\cat{C}_\he\in\VDbl$ is powerful in $\VDbl$, as the $2$-functor $\cat{C}_\he\times (-)\colon\VDbl\to \VDbl$ is the composite 
    \begin{equation*}
\begin{tikzpicture}[baseline=-\the\dimexpr\fontdimen22\textfont2\relax ]
      \node(0) at (0,0) {$\VDbl$};
      \node(1) at (2,0) {$\CAT$};
      \node(2) at (4,0) {$\CAT$};
      \node(3) at (6,0) {$\VDbl$};
      \draw [->] (0) to node[auto, labelsize] {$(-)_\vc$}  (1);
      \draw [->] (1) to node[auto,labelsize] {$\cat{C}\times (-)$}  (2);
      \draw [->] (2) to node[auto,labelsize] {$(-)_\he$}  (3);
\end{tikzpicture}
\end{equation*}
    of left $2$-adjoints.)
\end{remark}

It is easy to see that $(-)_\hc$ sends discrete opfibrations between categories to discrete opfibrations between virtual double categories (in the sense of \cref{def:disc-opfib}).
Let $\Set_\ast$ be the category of small pointed sets and $P\colon \Set_\ast \to \Set$ the forgetful functor. Then $P$ is a discrete opfibration.
This gives rise to a discrete opfibration $P_\hc\colon (\Set_\ast)_\hc\to \Set_\hc$ between virtual double categories.
By \cref{prop:disc-opfib-powerful}, we obtain $2$-functors 
\[
\begin{tikzpicture}[baseline=-\the\dimexpr\fontdimen22\textfont2\relax ]
      \node(0) at (0,0)  {$\VDbl$};
      \node(1) at (4,0)  {$\VDbl/(\Set_\ast)_\hc$};
      \node(2) at (8,0)  {$\VDbl/\Set_\hc$};
      \node(3) at (12,0)  {$\VDbl$,};
      \draw [->] (0) to node[auto,labelsize] {$(\Set_\ast)_\hc\times (-)$}  (1);
      \draw [->] (1) to node[auto,labelsize] {$\prod_{P_\hc}$}  (2);
      \draw [->] (2) to node[auto,labelsize] {forgetful}  (3);
\end{tikzpicture}
\]
whose composite is denoted by $\MMat\colon \VDbl\to\VDbl$.
In other words, $\MMat$ is the polynomial $2$-functor $\mathcal{P}(P_\hc)$ (see \cref{subsec:poly}), induced by the polynomial \eqref{eqn:polynomial-for-Mat}. 

The construction in the proof of \cref{prop:disc-opfib-powerful} allows us to describe the $2$-functor $\MMat$ explicitly:
it maps each $\vdcat{A}\in\VDbl$ to $\enMMat{\vdcat{A}}\in\VDbl$, which is the virtual double category described as follows.
\begin{itemize}
    \item The vertical category of $\enMMat{\vdcat{A}}$ is the category $\Fam(\vdcat{A}_\vc)$ of families of the vertical category $\vdcat{A}_\vc$ of $\vdcat{A}$. Hence an object in $\enMMat{\vdcat{A}}$ is a pair $(I,\vect{A})$ of a small set $I$ and an $I$-indexed family $\vect{A}=(A_i)_{i\in I}$ of objects in $\vdcat{A}$, while a vertical morphism in $\enMMat{\vdcat{A}}$ from $(I,\vect{A})$ to $(I',\vect{A}')$ is a pair $(v,\vect{f})$ of a function $v\colon I\to I'$ and a family $\vect{f}=(f_i\colon A_i\to A'_{v(i)})_{i\in I}$ of vertical morphisms in $\vdcat{A}$. 
    \item A horizontal morphism in $\enMMat{\vdcat{A}}$ from $(I,\vect{A})$ to $(I',\vect{A}')$ is a family (or a \emph{matrix}) $\vect{\Gamma}=(\Gamma_{i,i'}\colon A_i\pto A'_{i'})_{i\in I, i'\in I'}$ of horizontal morphisms in $\vdcat{A}$. 
    \item A multicell in $\enMMat{\vdcat{A}}$ of type 
\begin{equation}
\label{eqn:multicell-in-MatA}
\begin{tikzpicture}[baseline=-\the\dimexpr\fontdimen22\textfont2\relax ]
      \node(00) at (0,0.75) {$(I_0,\vect{A}_0)$};
      \node(01) at (2.5,0.75) {$(I_1,\vect{A}_1)$};
      \node(02) at (5,0.75) {$\cdots$};
      \node(03) at (7.5,0.75) {$(I_n,\vect{A}_n)$};
      \node(10) at (0,-0.75) {$(I,\vect{A})$};
      \node(11) at (7.5,-0.75) {$(I',\vect{A}')$};
      \draw [pto] (00) to node[auto, labelsize] {$\vect{\Gamma}_1$} (01);
      \draw [pto] (01) to node[auto, labelsize] {$\vect{\Gamma}_2$} (02);
      \draw [pto] (02) to node[auto, labelsize] {$\vect{\Gamma}_n$} (03);
      \draw [pto] (10) to node[auto, swap, labelsize] {$\vect{\Gamma}$} (11);
      \draw [->] (00) to node[auto, swap, labelsize] {$(v,\vect{f})$} (10); 
      \draw [->] (03) to node[auto,labelsize] {$(v',\vect{f}')$} (11);
      \node at (3.75,0) {$\vect{\alpha}$};
\end{tikzpicture}
\end{equation}
is a family
\begin{equation*}
\vect{\alpha}\quad =\quad \left(
\begin{tikzpicture}[baseline=-\the\dimexpr\fontdimen22\textfont2\relax ]
      \node(00) at (0,0.75) {$A_{0,i_0}$};
      \node(01) at (2,0.75) {$A_{1,i_1}$};
      \node(02) at (4,0.75) {$\cdots$};
      \node(03) at (6,0.75) {$A_{n,i_n}$};
      \node(10) at (0,-0.75) {$A_{v(i_0)}$};
      \node(11) at (6,-0.75) {$A'_{v'(i_n)}$}; 
      \draw [pto] (00) to node[auto, labelsize] {${\Gamma}_{1,i_0,i_1}$} (01);  
      \draw [pto] (01) to node[auto, labelsize] {${\Gamma}_{2,i_1,i_2}$} (02);  
      \draw [pto] (02) to node[auto, labelsize] {${\Gamma}_{n,i_{n-1},i_n}$} (03);  
      \draw [pto] (10) to node[auto, swap, labelsize] {${\Gamma}_{v(i_0),v'(i_n)}$} (11); 
      \draw [->] (00) to node[auto, swap, labelsize] {$f_{i_0}$} (10); 
      \draw [->] (03) to node[auto,labelsize] {$f'_{i_n}$} (11);
      \node at (3,0) {$\alpha_{i_0,\dots,i_n}$}; 
\end{tikzpicture}
\right)_{i_0\in I_0,i_1\in I_1,\dots, i_n\in I_n}
\end{equation*}
of multicells in $\vdcat{A}$.
\end{itemize}
The virtual double category structure on $\enMMat{\vdcat{A}}$ is evident.

\begin{theorem}\label{prop:MMat}
    The $2$-functor $\MMat\colon \VDbl\to \VDbl$ is familial.
\end{theorem}
\begin{proof}
    This follows from the construction of $\MMat$ as a polynomial 2-functor given above along with \cref{prop:poly-familial,dof-between-vdbl-is-internal-dof}.
\end{proof}

\section{Horizontal units in a virtual double category}\label{sec:units}
In this section, we review the notion of horizontal unit in a virtual double category \cite[Section~5]{Cruttwell-Shulman-unified} and construct various $2$-adjunctions involving them. 

\begin{definition}
    Let $\vdcat{D}$ be a virtual double category and $D$ an object in $\vdcat{D}$. A \emph{horizontal unit} at $D$ consists of 
    \begin{itemize}
        \item a horizontal morphism $\hu_D\colon D\pto D$ in $\vdcat{D}$, and 
        \item a multicell 
\begin{equation*}
\begin{tikzpicture}[baseline=-\the\dimexpr\fontdimen22\textfont2\relax ]
      \node(00) at (0.75,0.75) {$D$};
      \node(10) at (0,-0.75) {$D$};
      \node(11) at (1.5,-0.75) {$D$}; 
      \draw [pto] (10) to node[auto, swap, labelsize] {$\hu_D$} (11); 
      \draw [double equal sign distance] (00) to (10); 
      \draw [double equal sign distance] (00) to (11);
      \node at (0.75,-0.2) {$\widetilde\eta_D$}; 
\end{tikzpicture}
\end{equation*}
    in $\vdcat{D}$, 
    \end{itemize}
    such that $\widetilde\eta_D$ is \emph{opcartesian}. This last statement means that for each multicell 
\begin{equation*}
\begin{tikzpicture}[baseline=-\the\dimexpr\fontdimen22\textfont2\relax ]
      \node(00) at (0,0.75) {$D_0$};
      \node(01) at (1.5,0.75) {$D_1$};
      \node(02) at (3,0.75) {$\cdots$};
      \node(03) at (4.5,0.75) {$D_n$};
      \node(10) at (0,-0.75) {$E$};
      \node(11) at (4.5,-0.75) {$E'$}; 
      \draw [pto] (00) to node[auto, labelsize] {$\Gamma_1$} (01);  
      \draw [pto] (01) to node[auto, labelsize] {$\Gamma_2$} (02);  
      \draw [pto] (02) to node[auto, labelsize] {$\Gamma_n$} (03);  
      \draw [pto] (10) to node[auto, swap, labelsize] {$\Delta$} (11); 
      \draw [->] (00) to node[auto, swap, labelsize] {$u$} (10); 
      \draw [->] (03) to node[auto,labelsize] {$u'$} (11);
      \node at (2.25,0) {$\theta$}; 
\end{tikzpicture}
\end{equation*}
and each $0\leq i\leq n$ with $D_i=D$, there exists a unique multicell 
\begin{equation*}
\begin{tikzpicture}[baseline=-\the\dimexpr\fontdimen22\textfont2\relax ]
      \node(00) at (0,0.75) {$D_0$};
      \node(01) at (1.5,0.75) {$\cdots$};
      \node(02) at (3,0.75) {$D_{i-1}$};
      \node(03) at (4.5,0.75) {$D$};
      \node(04) at (6,0.75) {$D$};
      \node(05) at (7.5,0.75) {$D_{i+1}$};
      \node(06) at (9,0.75) {$\cdots$};
      \node(07) at (10.5,0.75) {$D_n$};
      \node(10) at (0,-0.75) {$E$};
      \node(11) at (10.5,-0.75) {$E'$}; 
      \draw [pto] (00) to node[auto, labelsize] {$\Gamma_1$} (01);  
      \draw [pto] (01) to node[auto, labelsize] {$\Gamma_{i-1}$} (02);  
      \draw [pto] (02) to node[auto, labelsize] {$\Gamma_{i}$} (03); \draw [pto] (03) to node[auto, labelsize] {$\hu_D$} (04);
      \draw [pto] (04) to node[auto, labelsize] {$\Gamma_{i+1}$} (05);  
      \draw [pto] (05) to node[auto, labelsize] {$\Gamma_{i+2}$} (06);  
      \draw [pto] (06) to node[auto, labelsize] {$\Gamma_{n}$} (07);  
      \draw [pto] (10) to node[auto, swap, labelsize] {$\Delta$} (11); 
      \draw [->] (00) to node[auto, swap, labelsize] {$u$} (10); 
      \draw [->] (07) to node[auto,labelsize] {$u'$} (11);
      \node at (5.25,0) {$\theta'$}; 
\end{tikzpicture}
\end{equation*}
in $\vdcat{D}$ such that
\[
    \begin{tikzpicture}[baseline=-\the\dimexpr\fontdimen22\textfont2\relax ]
      \node(20) at (0.75,1.5) {$D_0$};
      \node(21) at (2.25,1.5) {$\cdots$};
      \node(23) at (3.75,1.5) {$D$};
      \node(24) at (5.25,1.5) {$\cdots$};
      \node(25) at (6.75,1.5) {$D_n$};
      \node(00) at (0,0) {$D_0$};
      \node(01) at (1.5,0) {$\cdots$};
      \node(02) at (3,0) {$D$};
      \node(03) at (4.5,0) {$D$};
      \node(04) at (6,0) {$\cdots$};
      \node(05) at (7.5,0) {$D_n$};
      \node(10) at (0,-1.5) {$E$};
      \node(11) at (7.5,-1.5) {$E'$}; 
      \draw [pto] (20) to node[auto, labelsize] {$\Gamma_1$} (21);
      \draw [pto] (21) to node[auto, labelsize] {$\Gamma_{i}$} (23);
      \draw [pto] (23) to node[auto, labelsize] {$\Gamma_{i+1}$} (24);
      \draw [pto] (24) to node[auto, labelsize] {$\Gamma_n$} (25);
      \draw [pto] (00) to node[auto, swap,labelsize] {$\Gamma_1$} (01);
      \draw [pto] (01) to node[auto, swap,labelsize] {$\Gamma_{i}$} (02);
      \draw [pto] (02) to node[auto, swap,labelsize] {$\hu_{D}$} (03);
      \draw [pto] (03) to node[auto, swap,labelsize] {$\Gamma_{i+1}$} (04);
      \draw [pto] (04) to node[auto, swap,labelsize] {$\Gamma_n$} (05);
      \draw [pto] (10) to node[auto, swap, labelsize] {$\Delta$} (11); 
      \draw [double equal sign distance] (20) to (00); 
      \draw [double equal sign distance] (23) to (02); 
      \draw [double equal sign distance] (23) to (03); 
      \draw [double equal sign distance] (25) to (05); 
      \draw [->] (00) to node[auto, swap, labelsize] {$u$} (10); 
      \draw [->] (05) to node[auto,labelsize] {$u'$} (11);
      \node at (3.75,0.55) {$\widetilde\eta_D$};
      \node at (3.75,-0.75) {$\theta'$}; 
\end{tikzpicture}\quad=\quad\theta.
\]

By a \emph{unital virtual double category}, we mean a virtual double category $\vdcat{D}$ equipped with a chosen family 
$\bigl((\hu_D,\widetilde\eta_D)\bigr)_{D\in\vdcat{D}}$ of horizontal units for each object $D$ in $\vdcat{D}$.

Let $\vdcat{D}$ and $\vdcat{E}$ be unital virtual double categories. 
A virtual double functor $F\colon \vdcat{D}\to\vdcat{E}$ is \emph{unital} if it preserves the chosen horizontal units strictly, i.e., if we have $F\hu_D^{\vdcat{D}}=\hu_{FD}^\vdcat{E}$ and $F\widetilde\eta_D^{\vdcat{D}}=\widetilde\eta_{FD}^{\vdcat{E}}$ for each object $D\in \vdcat{D}$.
\end{definition}

We have a 2-category $\VDbln$ whose objects are large unital virtual double categories, whose morphisms are unital virtual double functors, and whose 2-cells are arbitrary vertical natural transformations. We have the evident forgetful 2-functor $\UU\colon \VDbln\to\VDbl$.

In the rest of this section, we show that $\UU$ has both left and right $2$-adjoints, and that there is a $2$-adjunction between $\VDbln$ and $\twoCAT$:
\[
\begin{tikzpicture}[baseline=-\the\dimexpr\fontdimen22\textfont2\relax ]
      \node(0) at (0,0) {$\VDbl$};
      \node(1) at (4,0) {$\VDbln$};
      \node(2) at (8,0) {$\twoCAT$.};
      \draw [->, bend left = 20] (0) to node[auto, labelsize] {$\FF$}  (1);  
      \draw [->] (1) to node[midway,fill=white, labelsize] {$\UU$}  (0); 
      \draw [->, bend right = 20] (0) to node[auto, swap,labelsize] {$\MMod$}  (1); 
      \path(0) to node[rotate=-90, transform canvas={yshift=8}] {$\dashv$} (1);
      \path(0) to node[rotate=-90, transform canvas={yshift=-8}] {$\dashv$} (1);
    \draw [->, transform canvas={yshift=5}] (1) to node[auto, labelsize] {$\V$}  (2);  
    \draw [<-,transform canvas={yshift=-5}] (1) to node[auto, swap,labelsize] {$\VV$}  (2);  
    \path(1) to node[rotate=-90] {$\dashv$} (2);
\end{tikzpicture}
\]

\begin{remark}
    Just as the virtual double categories can be identified with the $T$-categories where $T$ is the free category monad on the category $\mathbf{Gph}$ of graphs (see \cref{rmk:vdbl-cat-as-T-cat}), the unital virtual double categories are the $S$-categories where $S$ is the (non-cartesian) free category monad on the category $\mathbf{RGph}$ of reflexive graphs; see \cite{Fujii-Lack-nerve}.
\end{remark}

\subsection{The 2-adjunction $\UU\dashv \MMod$}
We start with the 2-adjunction
\begin{equation}\label{eqn:U-Mod-adjunction}
\begin{tikzpicture}[baseline=-\the\dimexpr\fontdimen22\textfont2\relax ]
      \node(0) at (0,0) {$\VDbl$};
      \node(1) at (4,0) {$\VDbln$,};
      \draw [<-, transform canvas={yshift=5}] (0) to node[auto, labelsize] {$\UU$}  (1);  
      \draw [<-, transform canvas={yshift=-5}] (1) to node[auto, labelsize] {$\MMod$}  (0); 
      \path(0) to node[rotate=-90] {$\dashv$} (1);
\end{tikzpicture}
\end{equation}
where $\UU$ is the forgetful 2-functor.
This is a variant of \cite[Proposition~5.14]{Cruttwell-Shulman-unified}.
We establish \cref{eqn:U-Mod-adjunction} by constructing the $2$-functor $\MMod$, the unit $\Inc\colon 1_{\VDbln}\to \MMod\UU$, and the counit
$\Und\colon \UU\MMod\to 1_{\VDbl}$.

For each virtual double category $\vdcat{A}$, the virtual double category $\MMod(\vdcat{A})$ is defined in \cite[Definition~5.3.1]{Leinster-HOHC},\footnote{As pointed out in \cite[footnote~3]{Arkor-McDermott-nerve}, the definition of nullary multicell in $\MMod(\vdcat{A})$ in \cite[Definition~5.3.1]{Leinster-HOHC} is incomplete.} \cite[Definition~2.8]{Cruttwell-Shulman-unified}, and \cite[Definition~3.9]{Arkor-McDermott-nerve}. 
An object in $\MMod(\vdcat{A})$ is a \emph{horizontal monad} $\mnd{T}=(A,T,\eta,\mu)$ in $\vdcat{A}$, consisting of an object $A$, a horizontal morphism $T\colon A\pto A$, and suitable multicells $\eta$ and $\mu$ in $\vdcat{A}$.  
A vertical (resp.\ horizontal) morphism in $\MMod(\vdcat{A})$ is called a \emph{monad morphism} (resp.\ \emph{monad bimodule}) in $\vdcat{A}$.
Given a horizontal monad $\mnd{T}=(A,T,\eta,\mu)$ in $\vdcat{A}$,
the underlying horizontal morphism $T$ of $\mnd{T}$ can be made into a monad bimodule $(T,\mu,\mu)$ from $\mnd{T}$ to $\mnd{T}$ via the multiplication $\mu$ of $\mnd{T}$,
and we choose this (together with the unit $\eta$ of $\mnd{T}$ as the opcartesian cell) as the horizontal unit at the object $\mnd{T}$ in $\MMod(\vdcat{A})$ \cite[Proposition~5.5]{Cruttwell-Shulman-unified}. 
In this way we can make $\MMod(\vdcat{A})$ into an object of $\VDbln$.
The assignment $\vdcat{A}\mapsto \MMod(\vdcat{A})$ routinely extends to a $2$-functor
$\MMod\colon\VDbl\to \VDbln$. 

We define the virtual double functor $\Und_\vdcat{A}\colon \UU\MMod(\vdcat{A})\to\vdcat{A}$ as the evident forgetful functor.
The $2$-naturality of $\Und\colon \UU\MMod\to 1_{\VDbl}$ is straightforward.

We next define a unital virtual double functor $\Inc_\vdcat{D}\colon \vdcat{D}\to \MMod(\UU\vdcat{D})$ for each unital virtual double category $\vdcat{D}$. 

\begin{definition}\label{def:omega-lambda-rho-in-unital-vdbl}
    Let $\vdcat{D}$ be a unital virtual double category.
    \begin{itemize}
        \item For each vertical morphism $u\colon D\to E$ in $\vdcat{D}$, we define the cell
\begin{equation*}
\begin{tikzpicture}[baseline=-\the\dimexpr\fontdimen22\textfont2\relax ]
      \node(00) at (0,0.75) {$D$};
      \node(01) at (1.5,0.75) {$D$};
      \node(10) at (0,-0.75) {$E$};
      \node(11) at (1.5,-0.75) {$E$}; 
      \draw [pto] (00) to node[auto, labelsize] {$\hu_D$} (01);  
      \draw [pto] (10) to node[auto, swap, labelsize] {$\hu_E$} (11); 
      \draw [->] (00) to node[auto, swap, labelsize] {$u$} (10); 
      \draw [->] (01) to node[auto, labelsize] {$u$} (11);
      \node at (0.75,0) {$\widetilde\omega_u$}; 
\end{tikzpicture}
\end{equation*}
in $\vdcat{D}$ as the unique cell such that 
\begin{equation*}
\begin{tikzpicture}[baseline=-\the\dimexpr\fontdimen22\textfont2\relax ]
      \node(20) at (0.75,1.5) {$D$};
      \node(00) at (0,0) {$D$};
      \node(01) at (1.5,0) {$D$};
      \node(10) at (0,-1.5) {$E$};
      \node(11) at (1.5,-1.5) {$E$}; 
      \draw [pto] (00) to node[auto, swap,labelsize] {$\hu_{D}$} (01);
      \draw [pto] (10) to node[auto, swap, labelsize] {$\hu_E$} (11); 
      \draw [double equal sign distance] (20) to (00); 
      \draw [double equal sign distance] (20) to (01); 
      \draw [->] (00) to node[auto, swap, labelsize] {$u$} (10); 
      \draw [->] (01) to node[auto, labelsize] {$u$} (11);
      \node at (0.75,0.5) {$\widetilde\eta_D$};
      \node at (0.75,-0.75) {$\widetilde\omega_u$}; 
\end{tikzpicture}
\quad=\quad
\begin{tikzpicture}[baseline=-\the\dimexpr\fontdimen22\textfont2\relax ]
      \node(20) at (0.75,1.5) {$D$};
      \node(00) at (0.75,0) {$E$};
      \node(10) at (0,-1.5) {$E$};
      \node(11) at (1.5,-1.5) {$E$.}; 
      \draw [pto] (10) to node[auto, swap, labelsize] {$\hu_E$} (11); 
      \draw [double equal sign distance] (00) to (10); 
      \draw [double equal sign distance] (00) to (11); 
      \draw [->] (20) to node[auto,labelsize] {$u$} (00); 
      \node at (0.75,-1) {$\widetilde\eta_E$}; 
\end{tikzpicture}
\end{equation*}
        \item For each horizontal morphism $\Gamma\colon D\pto D'$ in $\vdcat{D}$, we define the multicells 
\begin{equation*}
  \begin{tikzpicture}[baseline=-\the\dimexpr\fontdimen22\textfont2\relax ]
      \node(00) at (0,0.75) {$D$};
      \node(01) at (1.5,0.75) {$D$};
      \node(02) at (3,0.75) {$D'$};
      \node(10) at (0,-0.75) {$D$};
      \node(11) at (3,-0.75) {$D'$}; 
      \draw [pto] (00) to node[auto,labelsize] {$\hu_{D}$} (01);
      \draw [pto] (01) to node[auto,labelsize] {$\Gamma$} (02);
      \draw [pto] (10) to node[auto, swap, labelsize] {$\Gamma$} (11); 
      \draw [double equal sign distance] (00) to (10); 
      \draw [double equal sign distance] (02) to (11);
      \node at (1.5,0) {$\widetilde\lambda_\Gamma$}; 
\end{tikzpicture}
\quad\text{and}\quad
    \begin{tikzpicture}[baseline=-\the\dimexpr\fontdimen22\textfont2\relax ]
      \node(00) at (0,0.75) {$D$};
      \node(01) at (1.5,0.75) {$D'$};
      \node(02) at (3,0.75) {$D'$};
      \node(10) at (0,-0.75) {$D$};
      \node(11) at (3,-0.75) {$D'$}; 
      \draw [pto] (00) to node[auto,labelsize] {$\Gamma$} (01);
      \draw [pto] (01) to node[auto,labelsize] {$\hu_{D'}$} (02);
      \draw [pto] (10) to node[auto, swap, labelsize] {$\Gamma$} (11); 
      \draw [double equal sign distance] (00) to (10); 
      \draw [double equal sign distance] (02) to (11);
      \node at (1.5,0) {$\widetilde\rho_\Gamma$}; 
\end{tikzpicture}
\end{equation*}
in $\vdcat{D}$ as the unique multicells such that 
\begin{equation*}
\begin{tikzpicture}[baseline=-\the\dimexpr\fontdimen22\textfont2\relax ]
      \node(20) at (0.75,1.5) {$D$};
      \node(21) at (2.25,1.5) {$D'$};
      \node(00) at (0,0) {$D$};
      \node(01) at (1.5,0) {$D$};
      \node(02) at (3,0) {$D'$};
      \node(10) at (0,-1.5) {$D$};
      \node(11) at (3,-1.5) {$D'$}; 
      \draw [pto] (20) to node[auto, labelsize] {$\Gamma$} (21);
      \draw [pto] (00) to node[auto, swap,labelsize] {$\hu_{D}$} (01);
      \draw [pto] (01) to node[auto, swap,labelsize] {$\Gamma$} (02);
      \draw [pto] (10) to node[auto, swap, labelsize] {$\Gamma$} (11); 
      \draw [double equal sign distance] (20) to (00); 
      \draw [double equal sign distance] (20) to (01); 
      \draw [double equal sign distance] (21) to (02); 
      \draw [double equal sign distance] (00) to (10); 
      \draw [double equal sign distance] (02) to (11);
      \node at (0.75,0.5) {$\widetilde\eta_D$};
      \node at (1.5,-0.75) {$\widetilde\lambda_\Gamma$}; 
\end{tikzpicture}
\quad = \quad \id_\Gamma \quad = \quad
    \begin{tikzpicture}[baseline=-\the\dimexpr\fontdimen22\textfont2\relax ]
      \node(20) at (0.75,1.5) {$D$};
      \node(21) at (2.25,1.5) {$D'$};
      \node(00) at (0,0) {$D$};
      \node(01) at (1.5,0) {$D'$};
      \node(02) at (3,0) {$D'$};
      \node(10) at (0,-1.5) {$D$};
      \node(11) at (3,-1.5) {$D'$.}; 
      \draw [pto] (20) to node[auto, labelsize] {$\Gamma$} (21);
      \draw [pto] (00) to node[auto, swap,labelsize] {$\Gamma$} (01);
      \draw [pto] (01) to node[auto, swap,labelsize] {$\hu_{D'}$} (02);
      \draw [pto] (10) to node[auto, swap, labelsize] {$\Gamma$} (11); 
      \draw [double equal sign distance] (20) to (00); 
      \draw [double equal sign distance] (21) to (01); 
      \draw [double equal sign distance] (21) to (02); 
      \draw [double equal sign distance] (00) to (10); 
      \draw [double equal sign distance] (02) to (11);
      \node at (2.25,0.5) {$\widetilde\eta_{D'}$};
      \node at (1.5,-0.75) {$\widetilde\rho_\Gamma$}; 
\end{tikzpicture}
\end{equation*}
\end{itemize}
\end{definition}

\begin{proposition}\label{defining-Inc}
    Let $\vdcat{D}$ be a unital virtual double category. 
    \begin{enumerate}
        \item For each object $D$ in $\vdcat{D}$, we have $\widetilde\lambda_{\hu_D}=\widetilde\rho_{\hu_D}$; 
        we call this multicell $\widetilde\mu_D$.
        \item For each object $D$ in $\vdcat{D}$, $(D,\hu_D,\widetilde\eta_D,\widetilde\mu_D)$ is a horizontal monad in $\vdcat{D}$; we call this horizontal monad $\Inc_\vdcat{D}(D)$. 
        \item For each vertical morphism $u\colon D\to E$ in $\vdcat{D}$, $(u,\widetilde\omega_u)$ is a monad morphism from $\Inc_\vdcat{D}(D)$ to $\Inc_\vdcat{D}(E)$; we call this monad morphism $\Inc_\vdcat{D}(u)\colon \Inc_\vdcat{D}(D)\to \Inc_\vdcat{D}(E)$. 
        \item For each horizontal morphism $\Gamma\colon D\pto D'$ in $\vdcat{D}$, $(\Gamma,\widetilde\lambda_\Gamma,\widetilde\rho_\Gamma)$ is a monad bimodule from $\Inc_\vdcat{D}(D)$ to $\Inc_\vdcat{D}(D')$; we call this monad bimodule $\Inc_\vdcat{D}(\Gamma)\colon \Inc_\vdcat{D}(D)\pto\Inc_\vdcat{D}(D')$. 
        \item Each multicell 
        \begin{equation*}
\begin{tikzpicture}[baseline=-\the\dimexpr\fontdimen22\textfont2\relax ]
      \node(00) at (0,0.75) {$D_0$};
      \node(01) at (1.5,0.75) {$D_1$};
      \node(02) at (3,0.75) {$\cdots$};
      \node(03) at (4.5,0.75) {$D_n$};
      \node(10) at (0,-0.75) {$E$};
      \node(11) at (4.5,-0.75) {$E'$}; 
      \draw [pto] (00) to node[auto, labelsize] {$\Gamma_1$} (01);  
      \draw [pto] (01) to node[auto, labelsize] {$\Gamma_2$} (02);  
      \draw [pto] (02) to node[auto, labelsize] {$\Gamma_n$} (03);  
      \draw [pto] (10) to node[auto, swap, labelsize] {$\Delta$} (11); 
      \draw [->] (00) to node[auto, swap, labelsize] {$u$} (10); 
      \draw [->] (03) to node[auto,labelsize] {$u'$} (11);
      \node at (2.25,0) {$\theta$}; 
\end{tikzpicture}
\end{equation*}
        in $\vdcat{D}$ becomes a multicell 
\begin{equation*}
\begin{tikzpicture}[baseline=-\the\dimexpr\fontdimen22\textfont2\relax ]
      \node(00) at (0,0.75) {$\Inc_\vdcat{D}(D_0)$};
      \node(01) at (2.5,0.75) {$\Inc_\vdcat{D}(D_1)$};
      \node(02) at (5,0.75) {$\cdots$};
      \node(03) at (7.5,0.75) {$\Inc_\vdcat{D}(D_n)$};
      \node(10) at (0,-0.75) {$\Inc_\vdcat{D}(E)$};
      \node(11) at (7.5,-0.75) {$\Inc_\vdcat{D}(E')$}; 
      \draw [pto] (00) to node[auto, labelsize] {$\Inc_\vdcat{D}(\Gamma_1)$} (01);  
      \draw [pto] (01) to node[auto, labelsize] {$\Inc_\vdcat{D}(\Gamma_2)$} (02);  
      \draw [pto] (02) to node[auto, labelsize] {$\Inc_\vdcat{D}(\Gamma_n)$} (03);  
      \draw [pto] (10) to node[auto, swap, labelsize] {$\Inc_\vdcat{D}(\Delta)$} (11); 
      \draw [->] (00) to node[auto, swap, labelsize] {$\Inc_\vdcat{D}(u)$} (10); 
      \draw [->] (03) to node[auto,labelsize] {$\Inc_\vdcat{D}(u')$} (11);
      \node at (3.75,0) {$\theta$}; 
\end{tikzpicture}
\end{equation*}
        in $\MMod(\vdcat{D})$. 
    \end{enumerate}
\end{proposition}
\begin{proof}
    Apply the uniqueness part of the universal property of opcartesian multicells of the form $\widetilde\eta_D$ repeatedly.
\end{proof}

Thus for any unital virtual double category $\vdcat{D}$, we obtain a unital virtual double functor $\Inc_\vdcat{D}\colon \vdcat{D}\to \MMod(\UU\vdcat{D})$. 
The $2$-naturality of $\Inc\colon 1_\VDbln\to \MMod\UU$ follows from the following proposition.

\begin{proposition}
\label{prop:strictly-unital-pres-mu-omega-lambda-rho}
    Let $\vdcat{D}$ and $\vdcat{E}$ be unital virtual double categories, and $F\colon \vdcat{D}\to \vdcat{E}$ be a unital virtual double functor. 
    \begin{enumerate}
        \item For each object $D$ in $\vdcat{D}$, we have $F\widetilde\mu^\vdcat{D}_D = \widetilde\mu_{FD}^\vdcat{E}$. 
        \item For each vertical morphism $u$ in $\vdcat{D}$, we have $F\widetilde\omega_u^\vdcat{D} = \widetilde \omega_{Fu}^\vdcat{E}$.
        \item For each horizontal morphism $\Gamma$ in $\vdcat{D}$, we have $F\widetilde\lambda_\Gamma^\vdcat{D}=\widetilde\lambda_{F\Gamma}^\vdcat{E}$ and $F\widetilde\rho_{\Gamma}^\vdcat{D}=\widetilde\rho_{F\Gamma}^\vdcat{E}$. 
    \end{enumerate}
    If in addition we have a unital virtual double functor $G\colon\vdcat{D}\to \vdcat{E}$ and a vertical transformation $\alpha\colon F\to G$, then we have $\alpha_{\hu^\vdcat{D}_D}=\widetilde\omega^\vdcat{E}_{\alpha_D}$ for each object $D$ in $\vdcat{D}$.
\end{proposition}

Finally, we check the triangle identities.
That $\Und_{\UU\vdcat{D}}. \UU\Inc_\vdcat{D}=1_{\UU\vdcat{D}}$ holds for each unital virtual double category $\vdcat{D}$ is obvious.
The other triangle identity is a consequence of the following. 

\begin{proposition}
\label{prop:tilde-in-Mod}
    Let $\vdcat{A}$ be a virtual double category. 
    In the unital virtual double category $\MMod(\vdcat{A})$, the multicells introduced in \cref{def:omega-lambda-rho-in-unital-vdbl,defining-Inc} satisfy the following. 
    \begin{enumerate}
        \item For each object $\mnd{T}=(A,T,\eta,\mu)$ in $\MMod(\vdcat{A})$, we have $\widetilde\mu_\mnd{T}=\mu$.
        \item For each vertical morphism $\vmm{u}=(u,\omega)$ in $\MMod(\vdcat{A})$, we have $\widetilde\omega_{\vmm{u}}=\omega$. 
        \item For each horizontal morphism $\hbm{\Gamma}=(\Gamma,\lambda,\rho)$ in $\MMod(\vdcat{A})$, we have $\widetilde\lambda_{\hbm{\Gamma}}=\lambda$ and $\widetilde\rho_{\hbm{\Gamma}}=\rho$. 
    \end{enumerate}
\end{proposition}

This completes the construction of the $2$-adjunction \cref{eqn:U-Mod-adjunction}. 

As mentioned in \cref{sec:intro}, we define $\EEnr$ as the composite $2$-functor 
\begin{equation*}
\begin{tikzpicture}[baseline=-\the\dimexpr\fontdimen22\textfont2\relax ]
      \node(0) at (0,0)  {$\VDbl$};
      \node(1) at (3,0)  {$\VDbl$};
      \node(2) at (6.5,0)  {$\VDbln$.};
      \draw [->] (0) to node[auto, labelsize] {$\MMat$}  (1);
      \draw [->] (1) to node[auto, labelsize] {$\MMod$}  (2);
\end{tikzpicture}
\end{equation*}
For each $\vdcat{A}\in \VDbl$, the unital virtual double category $\EEnr(\vdcat{A})=\MMod(\enMMat{\vdcat{A}})$ is also denoted by $\PProf{\vdcat{A}}$ (see \cite[Example~2.9]{Cruttwell-Shulman-unified}).
The objects and vertical morphisms of $\PProf{\vdcat{A}}$ coincide with the small $\vdcat{A}$-categories and $\vdcat{A}$-functors defined in \cref{subsec:enrichment-over-vdbl}, whereas 
the $\vdcat{A}$-natural transformations appear as special cases of multicells in $\PProf{\vdcat{A}}$ (see \cref{subsec:VV-V} for details).
The horizontal morphisms of $\PProf{\vdcat{A}}$ are called the \emph{$\vdcat{A}$-profunctors}.

\begin{theorem}\label{EEnr-is-familial}
    The $2$-functor $\EEnr\colon \VDbl\to \VDbln$ sending $\vdcat{A}$ to $\PProf{\vdcat{A}}$ is given by the composite $\MMod\circ\MMat$ and so is familial.
\end{theorem}

\begin{proof}
    $\MMat$ is familial by \cref{prop:MMat} while $\MMod$ is a right $2$-adjoint and so familial by \cref{ex:right-adj-familial}.
    The claim follows from the fact that familial $2$-functors are closed under composition. 
\end{proof}

\begin{remark}
    In fact, for any unital virtual double category $\vdcat{D}$, the virtual double functors $\UU\Inc_\vdcat{D}$ and $\Und_{\UU\vdcat{D}}$ form an adjunction
    \begin{equation*}
\begin{tikzpicture}[baseline=-\the\dimexpr\fontdimen22\textfont2\relax ]
      \node(0) at (0,0) {$\UU\vdcat{D}$};
      \node(1) at (4,0) {$\UU\MMod(\UU\vdcat{D})$};
      \draw [->, transform canvas={yshift=5}] (0) to node[auto, labelsize] {$\UU\Inc_\vdcat{D}$}  (1);  
      \draw [->, transform canvas={yshift=-5}] (1) to node[auto, labelsize] {$\Und_{\UU\vdcat{D}}$}  (0); 
      \path(0) to node[rotate=-90] {$\dashv$} (1);
\end{tikzpicture}
\end{equation*}
    with identity unit in $\VDbl$.
\end{remark}

\subsection{The 2-adjunction $\FF\dashv \UU$}
Here we show that 
the forgetful $2$-functor $\UU\colon \VDbln\to \VDbl$ has a left $2$-adjoint $\FF\colon \VDbl\to \VDbln$ as well. 
Given a virtual double category $\vdcat{A}$, define the virtual double category $\FF\vdcat{A}$ as follows.
\begin{itemize}
    \item The vertical category of $\FF\vdcat{A}$ is identical to the vertical category of $\vdcat{A}$. 
    \item The horizontal morphisms of $\FF\vdcat{A}$ are those of $\vdcat{A}$ together with a new horizontal morphism $\hu_A\colon A\pto A$ for each object $A\in\vdcat{A}$. 
    \item Given a configuration 
    \begin{equation}
\label{eqn:boundary-of-2-cell-in-FA}
\begin{tikzpicture}[baseline=-\the\dimexpr\fontdimen22\textfont2\relax ]
      \node(00) at (0,0.75) {$A_0$};
      \node(01) at (1.5,0.75) {$A_1$};
      \node(02) at (3,0.75) {$\cdots$};
      \node(03) at (4.5,0.75) {$A_n$};
      \node(10) at (0,-0.75) {$B$};
      \node(11) at (4.5,-0.75) {$B'$}; 
      \draw [pto] (00) to node[auto, labelsize] {$\Gamma_1$} (01);  
      \draw [pto] (01) to node[auto, labelsize] {$\Gamma_2$} (02);  
      \draw [pto] (02) to node[auto, labelsize] {$\Gamma_n$} (03);  
      \draw [pto] (10) to node[auto, swap, labelsize] {$\Delta$} (11); 
      \draw [->] (00) to node[auto, swap, labelsize] {$u$} (10); 
      \draw [->] (03) to node[auto,labelsize] {$u'$} (11);
\end{tikzpicture}    
\end{equation}
    in $\FF\vdcat{A}$, if $\Delta\neq\hu_B$, then 
    the multicells in $\FF\vdcat{A}$ having \eqref{eqn:boundary-of-2-cell-in-FA} as the boundary are in bijection with the multicells in $\vdcat{A}$ having ``\eqref{eqn:boundary-of-2-cell-in-FA} with $\hu_A$'s erased'' as the boundary. 
    If $\Delta=\hu_B$, then there exists at most one multicell in $\FF\vdcat{A}$ having \eqref{eqn:boundary-of-2-cell-in-FA} as the boundary, and such a multicell exists in $\FF\vdcat{A}$ if and only if $u=u'$ and each horizontal morphism $\Gamma_i$ is of the form $\hu_A$. 
\end{itemize}
Then $\FF\vdcat{A}$ has the evident structure of a virtual double category. 
Moreover, for each object $A$, the unique multicell in $\FF\vdcat{A}$ having 
\begin{equation*}
\begin{tikzpicture}[baseline=-\the\dimexpr\fontdimen22\textfont2\relax ]
      \node(00) at (0.75,0.75) {$A$};
      \node(10) at (0,-0.75) {$A$};
      \node(11) at (1.5,-0.75) {$A$}; 
      \draw [pto] (10) to node[auto, swap, labelsize] {$\hu_A$} (11); 
      \draw [double equal sign distance] (00) to (10); 
      \draw [double equal sign distance] (00) to (11); 
\end{tikzpicture}
\end{equation*}
as the boundary, is opcartesian in $\FF\vdcat{A}$. Hence $\FF\vdcat{A}$ is a unital virtual double category.
We have the canonical inclusion virtual double functor $\vdcat{A}\to \UU\FF\vdcat{A}$, which can be easily seen to have the required universal property. 

Therefore $\UU\colon\VDbln\to \VDbl$ is a right $2$-adjoint, and we obtain the following variants of \cref{EEnr-is-familial}. 
\begin{proposition}
    The $2$-functors $\VDbl\xrightarrow{\EEnr}\VDbln\xrightarrow{\UU}\VDbl$, $\VDbln\xrightarrow{\UU}\VDbl\xrightarrow{\EEnr}\VDbln$, and $\VDbln\xrightarrow{\UU}\VDbl\xrightarrow{\EEnr}\VDbln\xrightarrow{\UU}\VDbl$ are familial.
\end{proposition}
\begin{proof}
    These are all composites of familial $2$-functors.
\end{proof}

\subsection{The 2-adjunction $\VV\dashv\V$}\label{subsec:VV-V}
In this subsection, we establish the 2-adjunction
\[
\begin{tikzpicture}[baseline=-\the\dimexpr\fontdimen22\textfont2\relax ]
      \node(0) at (0,0) {$\VDbln$};
      \node(1) at (4,0) {$\twoCAT$.};
      \draw [<-, transform canvas={yshift=5}] (0) to node[auto, labelsize] {$\VV$}  (1);  
      \draw [<-, transform canvas={yshift=-5}] (1) to node[auto, labelsize] {$\V$}  (0); 
      \path(0) to node[rotate=-90] {$\dashv$} (1);
\end{tikzpicture}
\]

Given a 2-category $\tcat{K}$, we define the virtual double category $\VV\tcat{K}$ as follows. 
\begin{itemize}
    \item The vertical category $(\VV\tcat{K})_\vc$ of $\VV\tcat{K}$ is the underlying category $\tcat{K}_0$ of $\tcat{K}$.
    \item For each object $K$ in $\tcat{K}$, there is a horizontal morphism $\hu_K\colon K\pto K$ in $\VV\tcat{K}$; there are no other horizontal morphisms in $\VV\tcat{K}$.
    \item For each configuration 
\begin{equation}
\label{eqn:boundary-VV-K}
\begin{tikzpicture}[baseline=-\the\dimexpr\fontdimen22\textfont2\relax ]
      \node(00) at (0,0.75) {$K$};
      \node(01) at (1.5,0.75) {$K$};
      \node(02) at (3,0.75) {$\cdots$};
      \node(03) at (4.5,0.75) {$K$};
      \node(10) at (0,-0.75) {$L$};
      \node(11) at (4.5,-0.75) {$L$}; 
      \draw [pto] (00) to node[auto, labelsize] {$\hu_K$} (01);  
      \draw [pto] (01) to node[auto, labelsize] {$\hu_K$} (02);  
      \draw [pto] (02) to node[auto, labelsize] {$\hu_K$} (03);  
      \draw [pto] (10) to node[auto, swap, labelsize] {$\hu_L$} (11); 
      \draw [->] (00) to node[auto, swap, labelsize] {$f$} (10); 
      \draw [->] (03) to node[auto,labelsize] {$f'$} (11);
\end{tikzpicture}
\end{equation}
    in $\VV\tcat{K}$, where there are $n$ copies of $\hu_K$ with $n\in\mathbb{N}$, the multicells in $\VV\tcat{K}$ having \eqref{eqn:boundary-VV-K} as the boundary are in bijection with the 2-cells from $f$ to $f'$ in $\tcat{K}$. 
\end{itemize}
Then $\VV\tcat{K}$ becomes a virtual double category under the natural operations.
Moreover, each multicell in $\VV\tcat{K}$ corresponding to the identity $2$-cell on an identity $1$-cell of $\tcat{K}$ is opcartesian. Hence we can make $\VV\tcat{K}$ into an object of $\VDbln$. 

On the other hand, given $\vdcat{D}\in\VDbln$, with chosen horizontal units $\bigl((\hu_D,\widetilde\eta_D)\bigr)_{D\in\vdcat{D}}$,
we define the 2-category $\V\vdcat{D}$ as follows. 
\begin{itemize}
    \item The underlying category $(\V\vdcat{D})_0$ of $\V\vdcat{D}$ is the vertical category $\vdcat{D}_\vc$ of $\vdcat{D}$. 
    \item For each parallel pair of morphisms $u,u'\colon D\to E$ in $(\V\vdcat{D})_0$, the set of all 2-cells from $u$ to $u'$ in $\V\vdcat{D}$ is the set of all cells 
    \begin{equation}
    \label{eqn:theta-in-VA}
\begin{tikzpicture}[baseline=-\the\dimexpr\fontdimen22\textfont2\relax ]
      \node(20) at (0,0.75) {$D$};
      \node(21) at (1.5,0.75) {$D$};
      \node(00) at (0,-0.75) {$E$};
      \node(01) at (1.5,-0.75) {$E$}; 
      \draw [pto] (20) to node[auto,labelsize] {$\hu_{D}$} (21);
      \draw [pto] (00) to node[auto,swap,labelsize] {$\hu_E$} (01);
      \draw [->] (20) to node[auto,swap,labelsize] {$u$} (00); 
      \draw [->] (21) to node[auto,labelsize] {$u'$} (01); 
      \node at (0.75,0) {$\theta$}; 
\end{tikzpicture}
\end{equation}
    in $\vdcat{D}$. 
\end{itemize}

Then one can make $\V\vdcat{D}$ into a $2$-category \cite[Proposition~6.1]{Cruttwell-Shulman-unified}. 

It is straightforward to see that we have a $2$-adjunction $\VV\dashv \V$.

Moreover, the composite $2$-functor $\VDbl\xrightarrow{\EEnr}\VDbln\xrightarrow{\V}\twoCAT$ is isomorphic to the $2$-functor $\Enr$ defined in \cref{subsec:enrichment-over-vdbl}; cf.\ \cite[Example~6.4]{Cruttwell-Shulman-unified}.
In particular, for any virtual double category $\vdcat{A}$, the $\vdcat{A}$-natural transformations correspond to the cells of the form \eqref{eqn:theta-in-VA} in $\PProf{\vdcat{A}}$.
Therefore we deduce the following.

\begin{theorem}
    The $2$-functor $\Enr\colon \VDbl\to \twoCAT$ is familial.
\end{theorem}

\section{The polynomial 2-functor \texorpdfstring{$\FFam$}{Fam} and virtual double categories of elements}
\label{sec:FFam}

For any category $\cat{C}$ with pullbacks, we have the virtual double category $\SSpan(\cat{C})$ 
whose vertical category is $\cat{C}$ and whose horizontal morphisms are spans in $\cat{C}$
\cite[Example~2.7]{Cruttwell-Shulman-unified}.
This construction defines a $2$-functor $\SSpan(-)\colon \CATpb\to \VDbl$, where $\CATpb$ is the $2$-category of large categories with pullbacks, pullback-preserving functors, and arbitrary natural transformations. 
We write the virtual double category $\SSpan(\Set)$ (resp.\ $\SSpan(\Set_\ast)$) as $\SSpan$ (resp.\ $\SSpan_\ast$). 

It is not difficult to see that the $2$-functor $\SSpan(-)\colon \CATpb\to \VDbl$ sends (pullback-preserving) discrete opfibrations between categories to discrete opfibrations between virtual double categories (in the sense of \cref{def:disc-opfib}). 
Therefore the forgetful functor $\Set_\ast\to \Set$ gives rise to a discrete opfibration $\SSpan_\ast\to \SSpan$ between virtual double categories, which we denote by $Q$. 

We define the $2$-functor $\FFam\colon\VDbl\to \VDbl$ as the composite 
\begin{equation}\label{eqn:FFam-composite}
\begin{tikzpicture}[baseline=-\the\dimexpr\fontdimen22\textfont2\relax ]
      \node(0) at (0,0)  {$\VDbl$};
      \node(1) at (4,0)  {$\VDbl/\SSpan_\ast$};
      \node(2) at (8,0)  {$\VDbl/\SSpan$};
      \node(3) at (12,0)  {$\VDbl$.};
      \draw [->] (0) to node[auto,labelsize] {$\SSpan_\ast\times (-)$}  (1);
      \draw [->] (1) to node[auto,labelsize] {$\prod_{Q}$}  (2);
      \draw [->] (2) to node[auto,labelsize] {forgetful}  (3);
\end{tikzpicture}
\end{equation}
That is, $\FFam$ is the polynomial $2$-functor $\mathcal{P}(Q)$ (see \cref{subsec:poly}). 
Using the construction given in the proof of \cref{prop:disc-opfib-powerful}, we can describe the virtual double category $\FFam(\vdcat{A})$ for $\vdcat{A}\in \VDbl$ as follows.
\begin{itemize}
    \item The vertical category of $\FFam(\vdcat{A})$ is the category $\Fam(\vdcat{A}_\vc)$ of families of $\vdcat{A}_\vc$. Thus the objects and vertical morphisms of $\FFam(\vdcat{A})$ are identical to those of $\enMMat{\vdcat{A}}$ described in \cref{sec:MMat}, and we adopt the same notation.
\item A horizontal morphism of $\FFam(\vdcat{A})$ from $(I,\vect{A})$ to $(I',\vect{A}')$ consists of a span 
   \[\begin{tikzpicture}[baseline=-\the\dimexpr\fontdimen22\textfont2\relax ]
      \node(0) at (0,-0.75) {$I$};
      \node(10) at (1.5,0.75) {$S$};
      \node(1) at (3,-0.75) {$I'$}; 
      \draw [->] (10) to node[auto, swap, labelsize] {$(-)_-$} (0); 
      \draw [->] (10) to node[auto, labelsize] {$(-)_+$} (1); 
\end{tikzpicture}\]
of small sets and an $S$-indexed family $\vect{\Gamma}=(\Gamma_s\colon A_{s_-}\pto A'_{s_+})_{s\in S}$ of horizontal morphisms in $\vdcat{A}$. 
    \item A multicell of $\FFam(\vdcat{A})$ of type
\begin{equation*}
\begin{tikzpicture}[baseline=-\the\dimexpr\fontdimen22\textfont2\relax ]
      \node(00) at (0,0.75) {$(I_0,\vect{A}_0)$};
      \node(01) at (2.5,0.75) {$(I_1,\vect{A}_1)$};
      \node(02) at (5,0.75) {$\cdots$};
      \node(03) at (7.5,0.75) {$(I_n,\vect{A}_n)$};
      \node(10) at (0,-0.75) {$(I,\vect{A})$};
      \node(11) at (7.5,-0.75) {$(I',\vect{A}')$}; 
      \draw [pto] (00) to node[auto, labelsize] {$(S_1,\vect{\Gamma}_1)$} (01);    
      \draw [pto] (01) to node[auto, labelsize] {$(S_2,\vect{\Gamma}_2)$} (02);   
      \draw [pto] (02) to node[auto, labelsize] {$(S_n,\vect{\Gamma}_n)$} (03);  
      \draw [pto] (10) to node[auto, swap, labelsize] {$(S,\vect{\Gamma})$} (11); 
      \draw [->] (00) to node[auto, swap, labelsize] {$(v,\vect{f})$} (10); 
      \draw [->] (03) to node[auto,labelsize] {$(v',\vect{f}')$} (11);
      \node at (3.75,0) {$(m,\vect{\alpha})$}; 
\end{tikzpicture}
\end{equation*}
consists of a multimap 
\begin{equation*}
\begin{tikzpicture}[baseline=-\the\dimexpr\fontdimen22\textfont2\relax ]
      \node(00) at (0,0.75) {$I_0$};
      \node(01) at (1.5,0.75) {$I_1$};
      \node(02) at (3,0.75) {$\cdots$};
      \node(03) at (4.5,0.75) {$I_n$};
      \node(10) at (0,-0.75) {$I$};
      \node(11) at (4.5,-0.75) {$I'$}; 
      \draw [pto] (00) to node[auto, labelsize] {$S_1$} (01);  
      \draw [pto] (01) to node[auto, labelsize] {$S_2$} (02);  
      \draw [pto] (02) to node[auto, labelsize] {$S_n$} (03);  
      \draw [pto] (10) to node[auto, swap, labelsize] {$S$} (11); 
      \draw [->] (00) to node[auto, swap, labelsize] {$v$} (10); 
      \draw [->] (03) to node[auto,labelsize] {$v'$} (11);
      \node at (2.25,0) {$m$}; 
\end{tikzpicture}
\end{equation*}
of spans in $\Set$ (which is a suitable function $m\colon S_1\times_{I_1} S_2\times_{I_2}\cdots \times_{I_{n-1}} S_n\to S$; see \cite[Example~2.7]{Cruttwell-Shulman-unified} for details) and a family $\vect{\alpha}$ of multicells indexed by the set $S_1\times_{I_1}S_2\times_{I_2}\dots\times_{I_{n-1}}S_n$, whose $(s_1,s_2,\dots,s_n)$-th member is a multicell of type 
\begin{equation*}
\begin{tikzpicture}[baseline=-\the\dimexpr\fontdimen22\textfont2\relax ]
      \node(00) at (0,0.75) {$A_{0,i_0}$};
      \node(01) at (1.5,0.75) {$A_{1,i_1}$};
      \node(02) at (3,0.75) {$\cdots$};
      \node(03) at (4.5,0.75) {$A_{n,i_n}$};
      \node(10) at (0,-0.75) {$A_{v(i_0)}$};
      \node(11) at (4.5,-0.75) {$A'_{v'(i_n)}$}; 
      \draw [pto] (00) to node[auto, labelsize] {${\Gamma}_{1,s_1}$} (01);  
      \draw [pto] (01) to node[auto, labelsize] {${\Gamma}_{2,s_2}$} (02);  
      \draw [pto] (02) to node[auto, labelsize] {${\Gamma}_{n,s_n}$} (03);  
      \draw [pto] (10) to node[auto, swap, labelsize] {${\Gamma}_{m(s_1,\dots,s_n)}$} (11); 
      \draw [->] (00) to node[auto, swap, labelsize] {$f_{i_0}$} (10); 
      \draw [->] (03) to node[auto,labelsize] {$f'_{i_n}$} (11);
      \node at (2.25,0) {$\alpha_{s_1,\dots,s_n}$}; 
\end{tikzpicture}
\end{equation*}
in $\vdcat{A}$, where $i_{k-1}=(s_{k})_-$ and $i_k=(s_k)_{+}$ for each $1\leq k\leq n$.
(When $n=0$, we identify the set $S_1\times_{I_1}S_2\times_{I_2}\dots\times_{I_{n-1}}S_n$ with $I_0$, and hence $\vect{\alpha}$ is an $I_0$-indexed family of multicells of $\vdcat{A}$.)
\end{itemize}

\begin{proposition}
    The $2$-functor $\FFam\colon \VDbl\to \VDbl$ is familial.
\end{proposition}
\begin{proof}
    This follows from \cref{prop:poly-familial,dof-between-vdbl-is-internal-dof}.
\end{proof}

The left $2$-adjoint of the $2$-functor $\FFam_1\colon \VDbl\to \VDbl/\SSpan$ (which is the composite of the first two factors of \cref{eqn:FFam-composite}) is 
called $\EElts\colon \VDbl/\SSpan\to \VDbl$. It is 
the composite 
\begin{equation*}
\begin{tikzpicture}[baseline=-\the\dimexpr\fontdimen22\textfont2\relax ]
      \node(0) at (0,0)  {$\VDbl/\SSpan$};
      \node(1) at (4,0)  {$\VDbl/\SSpan_\ast$};
      \node(2) at (8,0)  {$\VDbl$.};
      \draw [->] (0) to node[auto,labelsize] {$Q^\ast$}  (1);
      \draw [->] (1) to node[auto,labelsize] {forgetful}  (2);
\end{tikzpicture}
\end{equation*}
Hence $\EElts$ maps each $(F\colon \vdcat{A}\to \SSpan)\in \VDbl/\SSpan$ to $\EElts(F)\in\VDbl$ defined by the pullback 
\begin{equation*}
\begin{tikzpicture}[baseline=-\the\dimexpr\fontdimen22\textfont2\relax ]
      \node(00) at (0,0.75) {$\EElts(F)$};
      \node(01) at (2,0.75) {$\SSpan_\ast$};
      \node(10) at (0,-0.75) {$\vdcat{A}$};
      \node(11) at (2,-0.75) {$\SSpan$}; 
      \draw [->] (00) to node[auto, labelsize] {} (01);  
      \draw [->] (10) to node[auto, swap, labelsize] {$F$} (11); 
      \draw [->] (00) to node[auto, swap, labelsize] {$P_F$} (10); 
      \draw [->] (01) to node[auto,labelsize] {$Q$} (11);
      \draw (0.2,0.25) to (0.5,0.25) to (0.5,0.55);
\end{tikzpicture}
\end{equation*}
in $\VDbl$. It is easy to describe $\EElts(F)$ explicitly, since pullbacks in $\VDbl$ are straightforward.
\begin{itemize}
    \item The vertical category of $\EElts(F)$ is the category of elements $\Elts(F_\vc)$ of the functor $F_\vc\colon \vdcat{A}_\vc\to \SSpan_\vc=\Set$. Hence an object in $\EElts(F)$ is a pair $(A,a)$ of an object $A$ in $\vdcat{A}$ and an element $a\in FA$, and a vertical morphism $(A,a)\to (A',a')$ in $\EElts(F)$ is a vertical morphism $f\colon A\to A'$ in $\vdcat{A}$ such that $(Ff)(a)=a'$. 
    \item A horizontal morphism in $\EElts(F)$ from $(A,a)$ to $(A',a')$ is a pair $(\Gamma,\gamma)$ consisting of a horizontal morphism $\Gamma\colon A\pto A'$ in $\vdcat{A}$ and an element $\gamma\in F\Gamma$, such that 
\begin{equation}\label{eqn:horiz-in-EElts}
    \gamma_-=a\qquad\text{and}\qquad \gamma_+=a'
\end{equation}
hold. (Recall that $F$ maps $\Gamma$ to the span 
   \[\begin{tikzpicture}[baseline=-\the\dimexpr\fontdimen22\textfont2\relax ]
      \node(0) at (0,-0.75) {$FA$};
      \node(10) at (1.5,0.75) {$F\Gamma$};
      \node(1) at (3,-0.75) {$FA'$}; 
      \draw [->] (10) to node[auto, swap, labelsize] {$(-)_-$} (0); 
      \draw [->] (10) to node[auto, labelsize] {$(-)_+$} (1); 
\end{tikzpicture}\]
of sets.)
    \item A multicell in $\EElts(F)$ of the type
\begin{equation*}
\begin{tikzpicture}[baseline=-\the\dimexpr\fontdimen22\textfont2\relax ]
      \node(00) at (0,0.75) {$(A_0,a_0)$};
      \node(01) at (2.5,0.75) {$(A_1,a_1)$};
      \node(02) at (5,0.75) {$\cdots$};
      \node(03) at (7.5,0.75) {$(A_n,a_n)$};
      \node(10) at (0,-0.75) {$(A,a)$};
      \node(11) at (7.5,-0.75) {$(A',a')$}; 
      \draw [pto] (00) to node[auto, labelsize] {$(\Gamma_1,\gamma_1)$} (01);  
      \draw [pto] (01) to node[auto, labelsize] {$(\Gamma_2,\gamma_2)$} (02);  
      \draw [pto] (02) to node[auto, labelsize] {$(\Gamma_n,\gamma_n)$} (03);  
      \draw [pto] (10) to node[auto, swap, labelsize] {$(\Gamma,\gamma)$} (11); 
      \draw [->] (00) to node[auto, swap, labelsize] {$f$} (10); 
      \draw [->] (03) to node[auto,labelsize] {$f'$} (11);
      \node at (3.75,0) {$\alpha$}; 
\end{tikzpicture}
\end{equation*}
is a multicell 
\begin{equation*}
\begin{tikzpicture}[baseline=-\the\dimexpr\fontdimen22\textfont2\relax ]
      \node(00) at (0,0.75) {$A_0$};
      \node(01) at (1.5,0.75) {$A_1$};
      \node(02) at (3,0.75) {$\cdots$};
      \node(03) at (4.5,0.75) {$A_n$};
      \node(10) at (0,-0.75) {$A$};
      \node(11) at (4.5,-0.75) {$A'$}; 
      \draw [pto] (00) to node[auto, labelsize] {$\Gamma_1$} (01);  
      \draw [pto] (01) to node[auto, labelsize] {$\Gamma_2$} (02);  
      \draw [pto] (02) to node[auto, labelsize] {$\Gamma_n$} (03);  
      \draw [pto] (10) to node[auto, swap, labelsize] {$\Gamma$} (11); 
      \draw [->] (00) to node[auto, swap, labelsize] {$f$} (10); 
      \draw [->] (03) to node[auto,labelsize] {$f'$} (11);
      \node at (2.25,0) {$\alpha$}; 
\end{tikzpicture}
\end{equation*}
in $\vdcat{A}$ satisfying $(F\alpha)(\gamma_1,\gamma_2,\dots,\gamma_n)=\gamma$. 
(Notice that we have $(\gamma_1,\gamma_2,\dots,\gamma_n)\in F\Gamma_1\times_{FA_1}F\Gamma_2\times_{FA_2}\dots\times_{FA_{n-1}}F\Gamma_n$ thanks to \eqref{eqn:horiz-in-EElts}. Again, some modification is required when $n=0$.) 
\end{itemize}
The virtual double functor $P_F\colon \EElts(F)\to \vdcat{A}$ is given by the projection, and is a discrete opfibration (in the sense of \cref{def:disc-opfib}) because it is a pullback of a discrete opfibration $Q$.

Recall that the forgetful functor $\Set_\ast\to \Set$ is a universal discrete opfibration with small fibers in $\CAT$, in the sense that any discrete opfibration with small fibers in $\CAT$ can be obtained as a pullback of it.
We now show that $Q\colon \SSpan_\ast\to \SSpan$ is a universal discrete opfibration with small fibers in $\VDbl$. 
\begin{definition}
    A discrete opfibration $P\colon \vdcat{X}\to \vdcat{Y}$ between virtual double categories is said to have \emph{small fibres} if the following conditions are satisfied.
    \begin{itemize}
        \item For each object $Y\in\vdcat{Y}$, the set $P^{-1}(Y)$ of all objects $X\in \vdcat{Y}$ with $PX=Y$ is a small set. 
        \item For each horizontal morphism $\Psi$ in $\vdcat{Y}$, the set $P^{-1}(\Psi)$ of all horizontal morphisms $\Phi$ in $\vdcat{X}$ with $P\Phi=\Psi$ is a small set.
    \end{itemize}
\end{definition}

\begin{definition}\label{def:DP}
    Let $P\colon \vdcat{X}\to \vdcat{Y}$ be a discrete opfibration between virtual double categories having small fibres. Define a virtual double functor $D_P\colon \vdcat{Y}\to \SSpan$ as follows. 
    \begin{itemize}
        \item An object $Y$ of $\vdcat{Y}$ is mapped by $D_P$ to the set $P^{-1}(Y)$. 
        \item A vertical morphism $v\colon Y\to Y'$ of $\vdcat{Y}$ is mapped by $D_P$ to the function $v_\ast\colon P^{-1}(Y)\to P^{-1}(Y')$.
        \item A horizontal morphism $\Psi\colon Y\pto Y'$ of $\vdcat{Y}$ is mapped by $D_P$ to the span 
\[\begin{tikzpicture}[baseline=-\the\dimexpr\fontdimen22\textfont2\relax ]
      \node(0) at (0,-0.75) {$P^{-1}(Y)$};
      \node(10) at (1.5,0.75) {$P^{-1}(\Psi)$};
      \node(1) at (3,-0.75) {$P^{-1}(Y')$}; 
      \draw [->] (10) to node[auto, swap, labelsize] {$(-)_-$} (0); 
      \draw [->] (10) to node[auto, labelsize] {$(-)_+$} (1); 
\end{tikzpicture}\]
        of sets, where for each $(\Phi\colon X\pto X')\in P^{-1}(\Psi)$, we set $\Phi_-=X$ and $\Phi_+=X'$.
        \item A multicell 
        \begin{equation*}
\begin{tikzpicture}[baseline=-\the\dimexpr\fontdimen22\textfont2\relax ]
      \node(00) at (0,0.75) {$Y_0$};
      \node(01) at (1.5,0.75) {$Y_1$};
      \node(02) at (3,0.75) {$\cdots$};
      \node(03) at (4.5,0.75) {$Y_n$};
      \node(10) at (0,-0.75) {$Y$};
      \node(11) at (4.5,-0.75) {$Y'$}; 
      \draw [pto] (00) to node[auto, labelsize] {$\Psi_1$} (01);  
      \draw [pto] (01) to node[auto, labelsize] {$\Psi_2$} (02);  
      \draw [pto] (02) to node[auto, labelsize] {$\Psi_n$} (03);  
      \draw [pto] (10) to node[auto, swap, labelsize] {$\Psi$} (11); 
      \draw [->] (00) to node[auto, swap, labelsize] {$v$} (10); 
      \draw [->] (03) to node[auto,labelsize] {$v'$} (11);
      \node at (2.25,0) {$\nu$}; 
\end{tikzpicture}
\end{equation*}
in $\vdcat{Y}$ is mapped by $D_P$ to the function $\nu_\ast\colon P^{-1}(\Psi_1)\times_{P^{-1}(Y_1)}P^{-1}(\Psi_2)\times_{P^{-1}(Y_2)}\dots\times_{P^{-1}(Y_{n-1})}P^{-1}(\Psi_n)\to P^{-1}(\Psi)$.
    \end{itemize}
      It is straightforward to see that $D_P$ is indeed a virtual double functor.
\end{definition}

It is not hard to see that for any discrete opfibration $P\colon \vdcat{X}\to \vdcat{Y}$ with small fibers, we have 
a canonical isomorphism $(P_{D_P}\colon \EElts(D_P)\to \vdcat{Y})\cong (P\colon \vdcat{X}\to\vdcat{Y})$ in $\VDbl/\vdcat{Y}$. 
In other words, $P$ can be written as the following pullback of $Q$:
\begin{equation}\label{eqn:P-via-pullback}
\begin{tikzpicture}[baseline=-\the\dimexpr\fontdimen22\textfont2\relax ]
      \node(00) at (0,0.75) {$\vdcat{X}$};
      \node(01) at (2,0.75) {$\SSpan_\ast$};
      \node(10) at (0,-0.75) {$\vdcat{Y}$};
      \node(11) at (2,-0.75) {$\SSpan$.}; 
      \draw [->] (00) to node[auto, labelsize] {} (01);  
      \draw [->] (10) to node[auto, swap, labelsize] {$D_P$} (11); 
      \draw [->] (00) to node[auto, swap, labelsize] {$P$} (10); 
      \draw [->] (01) to node[auto,labelsize] {$Q$} (11);
      \draw (0.2,0.25) to (0.5,0.25) to (0.5,0.55);
\end{tikzpicture}
\end{equation}

We can use this fact to relate $\EEnr$ and $\FFam$. 
Consider the following diagram realizing $P_\hc\colon (\Set_\ast)_\hc\to \Set_\hc$ as a pullback of $Q\colon \SSpan_\ast\to \SSpan$: 
\begin{equation}\label{eqn:Phc-as-pb-of-Q}
\begin{tikzpicture}[baseline=-\the\dimexpr\fontdimen22\textfont2\relax ]
      \node(00) at (0,0.75) {$(\Set_\ast)_\hc$};
      \node(01) at (2,0.75) {$\SSpan_\ast$};
      \node(10) at (0,-0.75) {$\Set_\hc$};
      \node(11) at (2,-0.75) {$\SSpan$.}; 
      \draw [->] (00) to node[auto, labelsize] {} (01);  
      \draw [->] (10) to node[auto, swap, labelsize] {$J$} (11); 
      \draw [->] (00) to node[auto, swap, labelsize] {$P_\hc$} (10); 
      \draw [->] (01) to node[auto,labelsize] {$Q$} (11);
      \draw (0.2,0.25) to (0.5,0.25) to (0.5,0.55);
\end{tikzpicture}
\end{equation}
The virtual double functor $J$ is the identity on vertical categories, and maps the (unique) horizontal morphism in $\Set_\hc$ from $I$ to $I'$ to the span 
\[\begin{tikzpicture}[baseline=-\the\dimexpr\fontdimen22\textfont2\relax ]
      \node(0) at (0,-0.75) {$I$};
      \node(10) at (1.5,0.75) {$I\times I'$};
      \node(1) at (3,-0.75) {$I'$.}; 
      \draw [->] (10) to node[auto, swap, labelsize] {$\pi_1$} (0); 
      \draw [->] (10) to node[auto, labelsize] {$\pi_2$} (1); 
\end{tikzpicture}\]

\begin{proposition}\label{prop:AMat-pbk}
    For each $\vdcat{A}\in \VDbl$, we have a pullback
\[
\begin{tikzpicture}[baseline=-\the\dimexpr\fontdimen22\textfont2\relax ]
      \node(00) at (0,0.75) {$\enMMat{\vdcat{A}}$};
      \node(01) at (2,0.75) {$\FFam(\vdcat{A})$};
      \node(10) at (0,-0.75) {$\Set_\hc$};
      \node(11) at (2,-0.75) {$\SSpan$}; 
      \draw [->] (00) to node[auto, labelsize] {} (01);  
      \draw [->] (10) to node[auto, swap, labelsize] {$J$} (11); 
      \draw [->] (00) to node[auto, swap, labelsize] {$\enMMat{!_\vdcat{A}}$} (10); 
      \draw [->] (01) to node[auto,labelsize] {$\FFam(!_\vdcat{A})$} (11);
      \draw (0.2,0.25) to (0.5,0.25) to (0.5,0.55);
\end{tikzpicture}
\]
    in $\VDbl$.
\end{proposition}
\begin{proof}
    The pullback \eqref{eqn:Phc-as-pb-of-Q} readily implies the commutativity of the following (Beck--Chevalley) square 
    \begin{equation*}
\begin{tikzpicture}[baseline=-\the\dimexpr\fontdimen22\textfont2\relax ]
      \node(00) at (0,0.75) {$\VDbl/(\Set_\ast)_\hc$};
      \node(01) at (3.5,0.75) {$\VDbl/\SSpan_\ast$};
      \node(10) at (0,-0.75) {$\VDbl/\Set_\hc$};
      \node(11) at (3.5,-0.75) {$\VDbl/\SSpan$}; 
      \draw [->] (00) to node[auto, labelsize] {} (01);  
      \draw [->] (10) to node[auto, swap, labelsize] {$\sum_J$} (11); 
      \draw [<-] (00) to node[auto, swap, labelsize] {$P_\hc^\ast$} (10); 
      \draw [<-] (01) to node[auto,labelsize] {$Q^\ast$} (11);
\end{tikzpicture}
\end{equation*}
    up to an isomorphism. The claim follows by considering the commutative square obtained by taking the right $2$-adjoint of each of the above $2$-functors. 
\end{proof}

\begin{theorem}
    Let $\vdcat{A}\in \VDbl$.
    \begin{enumerate}
        \item We have a pullback 
        \[
\begin{tikzpicture}[baseline=-\the\dimexpr\fontdimen22\textfont2\relax ]
      \node(00) at (0,0.75) {$\PProf{\vdcat{A}}$};
      \node(01) at (3,0.75) {$\MMod\bigl(\FFam(\vdcat{A})\bigr)$};
      \node(10) at (0,-0.75) {$\Set_\hc$};
      \node(11) at (3,-0.75) {$\intPProf(\Set)$}; 
      \draw [->] (00) to node[auto, labelsize] {} (01);  
      \draw [->] (10) to node[auto, swap, labelsize] {} (11); 
      \draw [->] (00) to node[auto, swap, labelsize] {} (10); 
      \draw [->] (01) to node[auto,labelsize] {} (11);
      \draw (0.2,0.25) to (0.5,0.25) to (0.5,0.55);
\end{tikzpicture}
\]
        in $\VDbln$, where $\intPProf(\Set)$ is the unital virtual double category of \emph{internal} categories in $\Set$ described in e.g.\ \cite[Example~2.10]{Cruttwell-Shulman-unified}.
        \item We have a pullback 
\[
\begin{tikzpicture}[baseline=-\the\dimexpr\fontdimen22\textfont2\relax ]
      \node(00) at (0,0.75) {$\enCat{\vdcat{A}}$};
      \node(01) at (3,0.75) {$\V\Bigl(\MMod\bigl(\FFam(\vdcat{A})\bigr)\Bigr)$};
      \node(10) at (0,-0.75) {$\Setlc$};
      \node(11) at (3,-0.75) {$\Cat$}; 
      \draw [->] (00) to node[auto, labelsize] {} (01);  
      \draw [->] (10) to node[auto, swap, labelsize] {} (11); 
      \draw [->] (00) to node[auto, swap, labelsize] {} (10); 
      \draw [->] (01) to node[auto,labelsize] {} (11);
      \draw (0.2,0.25) to (0.5,0.25) to (0.5,0.55);
\end{tikzpicture}
\]
    in $\twoCAT$, where the (fully faithful) $2$-functor $\Setlc\to \Cat$ identifies $\Setlc$ with the full sub-$2$-category of $\Cat$ consisting of chaotic categories.
    \end{enumerate}
\end{theorem}

\begin{proof}
    The first pullback is obtained by applying the right 2-adjoint $\MMod\colon\VDbl\to\VDbln$ to the pullback in \cref{prop:AMat-pbk}, while the second is obtained by applying the right 2-adjoint $\V\circ\MMod\colon\VDbl\to\twoCAT$ to the same pullback.
\end{proof}

Using $\FFam$, we can also give an alternative description of $\prod_P\colon \VDbl/\vdcat{X}\to \VDbl/\vdcat{Y}$ (constructed in the proof of \cref{prop:disc-opfib-powerful}) for a discrete opfibration $P\colon \vdcat{X}\to \vdcat{Y}$ between virtual double categories,
at least when $P$ has small fibers. 
Let $D_P\colon \vdcat{Y}\to \SSpan$ be the virtual double functor defined in \cref{def:DP}. Then we have $\vdcat{X}\cong \EElts(D_P)$, and hence the unit of the $2$-adjunction $\EElts\dashv \FFam_1$ at $D_P$ is a virtual double functor $\overline{P}\colon \vdcat{Y}\to \FFam(\vdcat{X})$ making the diagram 
\begin{equation*}
\begin{tikzpicture}[baseline=-\the\dimexpr\fontdimen22\textfont2\relax ]
      \node(0) at (0,0.75) {$\vdcat{Y}$};
      \node(1) at (3,0.75) {$\FFam(\vdcat{X})$};
      \node(2) at (1.5,-0.75) {$\SSpan$};
      \draw [->] (0) to node[auto, labelsize] {$\overline P$}  (1);  
      \draw [<-] (2) to node[auto, swap, labelsize] {$\FFam(!_\vdcat{X})$}  (1); 
      \draw [->] (0) to node[auto, swap, labelsize] {$D_P$} (2); 
\end{tikzpicture}
\end{equation*}
commute. 
\begin{proposition}
    In this situation, for any $(F\colon \vdcat{A}\to \vdcat{X})\in \VDbl/\vdcat{X}$, we have a pullback 
        \[
\begin{tikzpicture}[baseline=-\the\dimexpr\fontdimen22\textfont2\relax ]
      \node(00) at (0,0.75) {$\prod_P\vdcat{A}$};
      \node(01) at (2,0.75) {$\FFam(\vdcat{A})$};
      \node(10) at (0,-0.75) {$\vdcat{Y}$};
      \node(11) at (2,-0.75) {$\FFam(\vdcat{X})$}; 
      \draw [->] (00) to node[auto, labelsize] {} (01);  
      \draw [->] (10) to node[auto, swap, labelsize] {$\overline P$} (11); 
      \draw [->] (00) to node[auto, swap, labelsize] {$\prod_PF$} (10); 
      \draw [->] (01) to node[auto,labelsize] {$\FFam(F)$} (11);
      \draw (0.2,0.25) to (0.5,0.25) to (0.5,0.55);
\end{tikzpicture}
\]  
    in $\VDbl$.
\end{proposition}
    See \cite[Lemma~2.1]{Weber-operad-poly} for an analogous result for categories. 
\begin{proof}
    It suffices to show that the diagram 
\begin{equation}
\label{eqn:prod-p-via-fam}
\begin{tikzpicture}[baseline=-\the\dimexpr\fontdimen22\textfont2\relax ]
      \node(0) at (0,-0.75) {$\VDbl/\vdcat{X}$};
      \node(1) at (1.5,0.75) {$\VDbl/\FFam(\vdcat{X})$};
      \node(2) at (3,-0.75) {$\VDbl/\vdcat{Y}$};
      \draw [->] (0) to node[auto, labelsize] {$\FFam_\vdcat{X}$}  (1);  
      \draw [<-] (2) to node[auto, swap, labelsize] {$\overline P^\ast$}  (1); 
      \draw [->] (0) to node[auto, swap, labelsize] {$\prod_P$} (2); 
\end{tikzpicture}
\end{equation}
    commutes up to isomorphism. This is equivalent to the statement that the diagram
    \begin{equation}\label{eqn:prod-p-via-fam-left}
\begin{tikzpicture}[baseline=-\the\dimexpr\fontdimen22\textfont2\relax ]
      \node(0) at (0,-0.75) {$\VDbl/\vdcat{X}$};
      \node(1) at (1.5,0.75) {$\VDbl/\FFam(\vdcat{X})$};
      \node(2) at (3,-0.75) {$\VDbl/\vdcat{Y}$,};
      \draw [<-] (0) to node[auto, labelsize] {$\overline \EElts_\vdcat{X}$}  (1);  
      \draw [->] (2) to node[auto, swap, labelsize] {$\sum_{\overline P}$}  (1); 
      \draw [<-] (0) to node[auto, swap, labelsize] {$P^\ast$} (2); 
\end{tikzpicture}
    \end{equation}
    obtained by taking the left 2-adjoints of the arrows in \cref{eqn:prod-p-via-fam}, commutes up to isomorphism.
    Here, $\overline{\EElts}_\vdcat{X}$ maps each $\bigr(\vdcat{B}\xrightarrow{H}\FFam(\vdcat{X})\bigr)\in \VDbl/\FFam(\vdcat{X})$ to 
    $\bigl(\EElts(\FFam(!_\vdcat{X}).H)\xrightarrow{\hat H} \vdcat{X}\bigr)\in \VDbl/\vdcat{X}$, where the morphism 
    $\hat H\colon \EElts(\FFam(!_\vdcat{X}).H)\to \vdcat{X}$ in $\VDbl$ is the transpose of the morphism $H\colon \bigl(\vdcat{B}\xrightarrow{H}\FFam(\vdcat{X})\xrightarrow{\FFam(!_\vdcat{X})}\SSpan\bigr)\to \bigl(\FFam(\vdcat{X})\xrightarrow{\FFam(!_\vdcat{X})} \SSpan\bigr)$ in $\VDbl/\SSpan$ with respect to the $2$-adjunction $\EElts\dashv \FFam_1$. 
    \[
\begin{tikzpicture}[baseline=-\the\dimexpr\fontdimen22\textfont2\relax ]
      \node(00) at (0,0.75) {$\EElts\bigl(\FFam(!_\vdcat{X}).H\bigr)$};
      \node(01) at (3.5,0.75) {$\EElts\bigl(\FFam(!_\vdcat{X})\bigr)$};
      \node(03) at (6,0.75) {$\SSpan_\ast$};
      \node(10) at (0,-0.75) {$\vdcat B$};
      \node(11) at (3.5,-0.75) {$\FFam(\vdcat{X})$}; 
      \node(13) at (6,-0.75) {$\SSpan$}; 
      \node(2) at (3.5,2.25) {$\vdcat{X}$}; 
      \draw [->] (00) to node[auto, labelsize] {$\EElts(H)$} (01);  
      \draw [->] (01) to node[auto, labelsize] {} (03);  
      \draw [->] (10) to node[auto, swap, labelsize] {$H$} (11);  
      \draw [->] (11) to node[auto, swap, labelsize] {$\FFam(!_\vdcat{X})$} (13); 
      \draw [->] (01) to node[auto, swap, labelsize] {$\varepsilon_\vdcat{X}$} (2); 
      \draw [->] (00) to node[auto, labelsize] {$\hat H$} (2); 
      \draw [->] (00) to node[auto, swap, labelsize] {} (10); 
      \draw [->] (01) to node[auto,swap,labelsize] {} (11);
      \draw [->] (03) to node[auto,labelsize] {$Q$} (13);
      \draw (0.2,0.25) to (0.5,0.25) to (0.5,0.55);
      \draw (3.7,0.25) to (4,0.25) to (4,0.55);
\end{tikzpicture}
\]  
    
    To show the commutativity of \eqref{eqn:prod-p-via-fam-left}, take any $(\vdcat{B}\xrightarrow{G}\vdcat{Y})\in \VDbl/\vdcat{Y}$. 
    By the triangle identity of the $2$-adjunction $\EElts\dashv\FFam_1$ at $(\vdcat{Y}\xrightarrow{D_P}\SSpan)\in \VDbl/\SSpan$, we see that $\overline\EElts_\vdcat{X}.\sum_{\overline P}$ maps $G$ to $K$ as in the following diagram. 
\[
\begin{tikzpicture}[baseline=-\the\dimexpr\fontdimen22\textfont2\relax ]
      \node(00) at (0,0.75) {$\EElts(D_PG)$};
      \node(01) at (2,0.75) {$\vdcat{X}$};
      \node(03) at (6.5,0.75) {$\SSpan_\ast$};
      \node(10) at (0,-0.75) {$\vdcat B$};
      \node(11) at (2,-0.75) {$\vdcat{Y}$}; 
      \node(12) at (4,-0.75) {$\FFam(\vdcat{X})$};
      \node(13) at (6.5,-0.75) {$\SSpan$}; 
      \draw [->] (00) to node[auto, labelsize] {$K$} (01);  
      \draw [->] (01) to node[auto, labelsize] {} (03);  
      \draw [->] (10) to node[auto, swap, labelsize] {$G$} (11); 
      \draw [->] (11) to node[auto, swap, labelsize] {$\overline P$} (12); 
      \draw [->] (12) to node[auto, swap, labelsize] {$\FFam(!_\vdcat{X})$} (13); 
      \draw [->] (00) to node[auto, swap, labelsize] {} (10); 
      \draw [->] (01) to node[auto,swap,labelsize] {$P$} (11);
      \draw [->] (03) to node[auto,labelsize] {$Q$} (13);
      \draw (0.2,0.25) to (0.5,0.25) to (0.5,0.55);
      \draw (2.2,0.25) to (2.5,0.25) to (2.5,0.55);
\end{tikzpicture}
\]  
    Here, the square on the right is a pullback by \cref{eqn:P-via-pullback}. 
    Hence the commutativity of \eqref{eqn:prod-p-via-fam-left} follows from the pasting lemma for pullbacks.
\end{proof}

\begin{remark}\label{rmk:Fam-for-PsDbl}
     If a virtual double category $\vdcat{A}$ is representable, then $\FFam(\vdcat{A})$ is also representable, and in this case the $\FFam$ construction coincides with the one given in \cite{Pare-Fam,Patterson-products}.
\end{remark}

\begin{remark}\label{rmk:Elt-for-PsDbl}
     Given $F\colon \vdcat{A}\to \SSpan$ in $\VDbl/\SSpan$, if $\vdcat{A}$ is representable, then $\EElts(F)$ is also representable, and in this case the $\EElts$ construction coincides with the one given in \cite[3.7]{Pare-Yoneda}.
\end{remark}

\bibliographystyle{alpha} %
\bibliography{myref} %
\end{document}